\documentclass[a4paper]{article}

\usepackage[ruled]{algorithm2e}
\usepackage{amsmath}
\usepackage{amssymb}
\usepackage{amsthm}
\usepackage{array}
\usepackage{authblk}
\usepackage[english]{babel}
\usepackage{bm}
\usepackage{caption}
\usepackage{cite}
\usepackage{color}
\usepackage{dsfont}
\usepackage{float}
\usepackage[T1]{fontenc}
\usepackage{geometry}
\usepackage{graphicx}
\usepackage{hhline}
\usepackage{hyperref}
\usepackage[utf8x]{inputenc}
\usepackage{latexsym}
\usepackage{makecell}
\usepackage{mathrsfs}
\usepackage{mathtools}
\usepackage{multicol}
\usepackage{multirow}
\usepackage{nomencl}
\usepackage{subcaption}
\usepackage{systeme}
\usepackage{tabu}
\usepackage[colorinlistoftodos]{todonotes}
\usepackage{url}

\SetKwInOut{Input}{Input}
\SetKwInOut{Output}{Output}
\DontPrintSemicolon
\makenomenclature
\renewcommand\nompreamble{\begin{multicols}{2}}
\renewcommand\nompostamble{\end{multicols}}

\makeatother
\numberwithin{equation}{section}
\newtheorem{thm}{Theorem}
\newtheorem{defi}{Definition}
\newtheorem{coro}{Corollary}
\newtheorem{lem}{Lemma}
\newtheorem{prob}{Problem}
\newtheorem{prop}{Proposition}

\theoremstyle{definition}
\newtheorem{exa}{Example}

\def\T{\mathrm T}
\def\rd{\mathrm d}
\def\e{\mathrm e}
\def\diag{\mathrm{diag}}
\def\BA{\mathbb A}
\def\BG{\mathbb G}
\def\BN{\mathbb N}
\def\BQ{\mathbb Q}
\def\BR{\mathbb R}
\def\BS{\mathbb S}
\def\cA{\mathcal A}
\def\cC{\mathcal C}
\def\cN{\mathcal N}
\def\cS{\mathcal S}
\def\cX{\mathcal X}
\def\Ga{\Gamma}
\def\Si{\Sigma}
\def\Om{\Omega}
\def\al{\alpha}
\def\be{\beta}
\def\ga{\gamma}
\def\ve{\varepsilon}
\def\ze{\zeta}
\def\la{\lambda}
\def\vp{\varphi}
\def\f{\frac}
\def\nb{\nabla}
\def\pa{\partial}
\def\ov{\overline}
\def\wt{\widetilde}

\AtEndDocument{
\par
\medskip
\begin{tabular}{@{}l@{}}
\textsc{}\\
{\it } \texttt{}
\end{tabular}}

\title{\bf Inverse $t$-Source Problem and a Strict Positivity Property for Coupled Subdiffusion Systems}

\author{Mohamed BenSalah\thanks{Department of Computer Sciences, University of Sousse, Rue Tahar Ben Achour, Sousse 4003, Tunisia.\\
E-mail : mohamed.bensalah@fsm.rnu.tn}\qquad
Yikan Liu\thanks{Department of Mathematics, Kyoto University, Kitashirakawa-Oiwakecho, Sakyo-ku, Kyoto 606-8502, Japan.\\
E-mail : liu.yikan.8z@kyoto-u.ac.jp}}
\date{}

\begin{document}

\maketitle

\begin{abstract}
This article is concerned with the inverse problem on determining the temporal component of the source term in a coupled time-fractional diffusion system by a single point observation. Under a non-degeneracy condition on the known spatial component, we establish the Lipschitz stability by observing all solution components by a series representation of the mild solution. To reduce the observation data, we prove the strict positivity of some fractional integral of the solution to the homogeneous problem by a modified Picard iteration. This, together with a coupled Duhamel's principle, lead us to the uniqueness of the inverse problem by observing any single solution component under a specific structural constraint on the unknown. Numerically, we propose an iterative regularizing ensemble Kalman method (IREKM) for the simultaneous recovery of the temporal sources. Through extensive numerical tests, we demonstrate its accuracy, robustness against noise, and scalability with respect to the number of components. Our findings highlight the essential roles of the non-degeneracy condition, measurement configuration, and fractional structural constraints in ensuring reliable reconstructions. The proposed framework provides both rigorous theoretical guarantees and a practical algorithmic approach for multi-component source identification in fractional diffusion systems.\medskip

Bayesian inverse problems, Iterative regularizing ensemble Kalman method\medskip

\end{abstract}


\section{Introduction}

With the maturation of fundamental theories of time-fractional partial differential equations, coupled subdiffusion systems have gradually attracted the attention of applied mathematicians in recent years. Basic well-posedness for mild solutions to weakly-coupled subdiffusion systems was established in Li, Huang and Liu \cite{LHL23}, and Li, Liu and Wada \cite{LLW26} discovered new long-time decay patterns which never occur in single equations. Meanwhile, several related inverse problems were studied, e.g., in \cite{FLL25,FL25}. Nevertheless, in comparison with the abundant achievements for the single counterpart, coupled subdiffusion systems are still awaiting further investigations especially in the directions of qualitative properties and various types of inverse problems. For the single counterpart, inverse source identification problems have been extensively studied under various conditions. For instance, inverse source problems within abstract frameworks of positive operators, including hypoelliptic diffusion and subdiffusion equations, were established in \cite{RTT19}. More recently, these formulations have been successfully extended to complex evolutionary and non-linear environments by Ashurov and his co-workers; we mention the time-dependent identification problem for a fractional telegraph equation \cite{AS24} as well as the source identification problem for a nonlinear subdiffusion equation \cite{AM25}.

Within such a circumstance, this article is concerned with an inverse problem on determining the temporal component of the source term as well as a related strict positivity property for coupled subdiffusion systems. To begin with, we formulate the mathematical model under consideration. Let $T>0$ be constant and $\Om\subset\BR^d$ ($d\in\BN:=\{1,2,\dots\}$) be an open bounded domain whose boundary $\pa\Om$ is sufficiently smooth. Fix $K\in\BN$ and let $\al_1,\dots,\al_K$ be constants satisfying $1>\al_1\ge\cdots\ge\al_K>0$. In this article, we are concerned with the following initial-boundary value problem for a coupled system of subdiffusion equations
\begin{equation}\label{model}
\begin{cases}
\displaystyle\pa_t^{\al_k}u_k+\cA_k u_k+\sum_{\ell=1}^K c_{k\ell}(\bm x)u_\ell=\sum_{\ell=1}^K g_{k\ell}(\bm x)\rho_\ell(t) & \mbox{in }\Om\times(0,T),\\
u_k=0 & \mbox{on }\pa\Om\times(0,T),
\end{cases}
\end{equation}
$k=1,\dots,K$. Here $\pa_t^{\al_k}$ denotes the inverse of the $\al_k$-th order Riemann-Liouville integral operator
$$
J^{\al_k}:L^2(0,T)\longrightarrow L^2(0,T),\quad J^{\al_k}f(t):=\int_0^t\f{\tau^{\al_k-1}}{\Ga(\al_k)}f(t-\tau)\,\rd\tau,
$$
where $\Ga(\,\cdot\,)$ stands for the Gamma function. According to \cite{KRY20}, $\pa_t^{\al_k}$ generalizes the conventional Caputo and Riemann-Liouville derivatives to its domain $D(\pa_t^{\al_k})=H_{\al_k}(0,T)$, which collects functions in some fractional Sobolev space. Then by $\pa_t^{\al_k}u_k$ we mean $u_k(\bm x,\,\cdot\,)\in H_{\al_k}(0,T)$ for a.e. $\bm x\in\Om$, and the initial condition $u_k=0$ in $\Om\times\{0\}$ only makes pointwise sense for $\al_k>1/2$. Meanwhile, $\cA_k$ denotes a self-adjoint second order elliptic operator defined by
\[
\cA_k\psi:=-\mathrm{div}(\bm A_k(\bm x)\nb\psi)=-\sum_{i,j=1}^d\pa_{x_j}\left(a_{ij}^{(k)}(\bm x)\pa_{x_i}\psi\right)
\]
for $\psi\in D(\cA_k):=H^2(\Om)\cap H_0^1(\Om)$. Here $\bm A_k=(a_{ij}^{(k)})_{1\le i,j\le d}\in C^1(\ov\Om;\BR^{d\times d})$ ($k=1,\dots,K$) are symmetric matrix-valued functions on $\ov\Om$ and there exists a constant $\kappa>0$ such that
\[
\bm A_k(\bm x)\bm\xi\cdot\bm\xi\ge\kappa|\bm\xi|^2,\quad\forall\,\bm\xi\in\BR^d,\ \forall\,\bm x\in\ov\Om,\ \forall\,k=1,\dots,K,
\]
where $\cdot$ is the inner product in $\BR^d$ and $|\bm\xi|^2:=\bm\xi\cdot\bm\xi$. 

For later convenience, we adopt a vector representation
\[
\bm u=(u_1,\dots,u_K)^\T,\quad\bm\rho:=(\rho_1,\dots,\rho_K)^\T
\]
and introduce (formal) matrices
\begin{gather*}
\pa_t^{\bm\al}:=\diag(\pa_t^{\al_1},\dots,\pa_t^{\al_K}),\quad\BA:=\diag(\cA_1,\dots,\cA_K),\\
\bm C:=(c_{k\ell})_{1\le k,\ell\le K},\quad\bm G:=(g_{k\ell})_{1\le k,\ell\le K}
\end{gather*}
to rewrite \eqref{model} concisely as
\begin{equation}\label{eq-IBVP-u0}
\begin{cases}
(\pa_t^{\bm\al}+\BA+\bm C)\bm u=\bm G(\bm x)\bm\rho(t) & \mbox{in }\Om\times(0,T),\\
\bm u=\bm0 & \mbox{on }\pa\Om\times(0,T).
\end{cases}
\end{equation}
Then $\BA\bm u$ and $\bm C\bm u$ are the principal and zeroth order terms in the spatial direction respectively, where $\bm C$ is the $\bm x$-dependent coefficient matrix coupling components in $\bm u$. On the right-hand side, $\bm G\bm\rho$ is the source term taking the form of separated variables, in which $\bm G$ and $\bm\rho$ are spatial and temporal components, respectively. Assumptions on the regularities of $\bm C,\bm G$ and $\bm\rho$ depend on different problem settings and thus will be specified later in Section \ref{sec-main}.

The first main objective of this article is the following inverse source problem on determining the temporal component $\bm\rho$ of the source term in \eqref{eq-IBVP-u0} by the single point observation data.

\begin{prob}\label{prob-ISP}
Fix $\bm x_0\in\Om$ and let $\bm u$ be the solution to \eqref{eq-IBVP-u0}. Provided that the spatial component $\bm G$ in the source term is suitably given, determine the temporal component $\bm\rho$ by the single point observation of $\bm u$ at $\{\bm x_0\}\times(0,T)$.
\end{prob}

The above problem generalizes its prototype for single time-fractional equations, namely, determine $\rho(t)$ in
\[
(\pa_t^\al+\cA)u=g(\bm x)\rho(t)
\]
by observing $u$ at a monitoring point $\bm x_0\in\Om$. The study of this problem can be traced back to Sakamoto and Yamamoto \cite{SY11}, which showed Lipschitz stability of $\rho$ with respect to $\pa_t^\al u(\bm x_0,\,\cdot\,)$ when $g(\bm x_0)\ne0$. Such a restriction on $\bm x_0$ was removed in Liu, Rundell and Yamamoto \cite{LRY16}, where the uniqueness for arbitrary $\bm x_0\in\Om$ was proved by the strict positivity of solutions to homogeneous problems. From then on, there have been explosive growth in corresponding numerical reconstruction methods and we refer to \cite{LLY19} for a topical survey until the year of 2019. As recent developments in this direction, we mention \cite{HLY20} on the local
stability in determining the orbit $\bm\rho(t)$ of a moving source, \cite{LY23} on 
the uniqueness for singular $\rho$ in negative order Sobolev spaces, \cite{AS24} on an abstract time-dependent source identification framework for fractional evolution systems, \cite{AS22} on time-dependent source problems for fractional Schr\"odinger models, \cite{AM25} on source identification formulations for nonlinear subdiffusion equations, and 
\cite{HL23} on the same problem for a time-fractional wave equation with an order 
$\al\in(1,2)$.

Regardless of the fruitful progress achieved for single equations, Problem \ref{prob-ISP} has not yet been studied for coupled subdiffusion systems to the best of our knowledge. As other types of inverse problems related to \eqref{eq-IBVP-u0}, we refer to \cite{RHY21} for the inverse coefficient problem, \cite{LHL23,L25} for the uniqueness and numerical reconstruction of orders, \cite{FLL25} for the backward problem, and \cite{FL25} for the inverse $\bm x$-source problem. However, the inverse $t$-source problem for \eqref{eq-IBVP-u0} remains open in the literature, regardless of its practical importance in environmental pollution problems. This motivates us to deal with Problem \ref{prob-ISP} from both theoretical and numerical aspects.

As was seen in the single equation case, the difficulty of Problem \ref{prob-ISP} depends heavily on the choice of $\bm x_0$, in which some strict positivity property for homogeneous problems plays an essential role. Therefore, as a highly related problem, this article also keeps an eye on the following problem concerning the positivity property for a homogeneous counterpart of \eqref{eq-IBVP-u0}.

\begin{prob}\label{prob-SPP}
Let $\bm v$ satisfy the initial-boundary value problem
\begin{equation}\label{eq-IBVP-v0}
\begin{cases}
\pa_t^{\bm\al}(\bm v-\bm g)+(\BA+\bm C)\bm v=\bm0 & \mbox{in }\Om\times(0,T),\\
\bm v=\bm0 & \mbox{on }\pa\Om\times(0,T),
\end{cases}
\end{equation}
where all components in the initial value $\bm g=(g_1,\dots,g_K)^\T$ are non-negative. If some components of $\bm g$ do not vanish identically, can we conclude certain strict positivity of $\bm v$ under some suitable assumptions?
\end{prob}

As one of the most typical features of parabolic and subdiffusion equations, the (strong) maximum principle and related properties have been studied intensively in literature. Focusing on the subdiffusion ones, Luchko \cite{L09} first established the weak maximum principle for single equations by means of an extremum principle for the Caputo derivative. This result was later improved to a strong one in \cite{LRY16}, asserting that the solution to a homogeneous equation is strictly positive a.e. in $\Om\times(0,\infty)$ if the initial value is non-negative and non-vanishing. We refer to Luchko and Yamamoto \cite{LY19} as a comprehensive review on maximum principle for time-fractional diffusion equations. For coupled subdiffusion systems, a non-negativity result was first mentioned in \cite{LY17} and then rigorously proved recently in \cite{LY25}, namely, non-negative initial value and source term yield a non-negative solution. However, there seems no stronger result as that in \cite{LRY16} available in existing literature, which inspires the proposal of Problem \ref{prob-SPP}. Remarkably, in view of the coupling effect in \eqref{eq-IBVP-v0}, the strict positivity is expected to propagate among the components of the solution. Therefore, in Problem \ref{prob-SPP} we also seek the possibility of achieving the strict positivity of $\bm v$ even if the initial values of some components vanish identically.

In this article, we prove two main results concerning Problem \ref{prob-ISP}, among which the latter relies on an affirmative answer to Problem \ref{prob-SPP}. First, in Theorem \ref{thm-Lip} we show the Lipschitz stability of $\bm\rho$ with respect to $\pa_t^{\bm\al}\bm u(\bm x_0,\,\cdot\,)$ following the methodology of \cite{SY11,HLY20} under the key assumption $\det\bm G(\bm x_0)\ne0$, which naturally generalizes the similar non-degeneracy condition for single equations. The proof requires a solution estimate involving a weakly singular integral, which is achieved via a new series representation for the mild solution (see Proposition \ref{prop-mildsol}). Second, constructing the mild solution by a modified Picard iteration, we establish an intermediate strict positivity property for \eqref{eq-IBVP-v0} in Theorem \ref{thm-SPP}, stating that some Riemann-Liouville integral of the solution is strictly positive almost everywhere. Such a drawback comes from the absence of a strict positivity property of inhomogeneous problems for single equations (see Corollary \ref{coro-SPP0}). However, it suffices to assume that only a part of initial values are non-vanishing and the coupling matrix $\bm C$ can spread the strict positivity throughout all components of the solution. Further, this incomplete strict positivity is sufficient to show Theorem \ref{thm-ISP}, that is, the uniqueness of Problem \ref{prob-ISP} by observing an arbitrary single component $u_k$ of the solution at an arbitrary point $\bm x_0$. The proof is based on a new Duhamel's principle for coupled subdiffusion systems and a structural condition on $\bm\rho$, so that the unknowns in principle reduce to a scalar-valued function in $t$.

After establishing the theoretical properties of the inverse problem, we adopt a Bayesian framework \cite{DS17,I13,S10} for the reconstruction of the unknown temporal component $\bm\rho$ from single-point observations. The Bayesian methodology provides a natural and flexible regularization mechanism \cite{S10}, as prior information and measurement noise are incorporated within a unified probabilistic formulation. In contrast to deterministic optimization-based approaches \cite{TW10,V02,EHN96}, which typically yield a single best-fit solution, the Bayesian approach characterizes the solution through a full posterior distribution. This enables not only accurate reconstruction but also a quantitative assessment of reliability and uncertainty. Such a feature is particularly important for ill-posed inverse source problems arising in time-fractional coupled systems, where stability is delicate and measurement noise may significantly influence the solution. Moreover, Bayesian inversion ensures probabilistic well-posedness, meaning that the posterior distribution concentrates around the true solution as the noise level decreases. For the numerical implementation, we employ the ad-hoc iterative regularizing ensemble Kalman method developed in \cite{CIRS18,I15,I16}, which provides an efficient and derivative-free framework for approximating the posterior distribution. A further advantage of this approach is that it avoids the derivation of adjoint equations and the computation of cost functional derivatives \cite{ILS13}, thereby significantly simplifying the implementation for strongly coupled fractional models.

Our aim is therefore to extend Bayesian inference to the inverse source problem governed by the coupled subdiffusion system \eqref{eq-IBVP-u0}, and to establish a coherent analytical and computational framework for recovering the temporal component $\bm\rho$ from limited observation data. In order to guarantee the validity of the posterior measure, we examine the well-definedness and well-posedness of the Bayesian inverse problem, which requires analyzing the continuity properties of the forward mapping associated with \eqref{eq-IBVP-u0}. The main analytical difficulty arises from the presence of multiple fractional operators $\partial_t^{\al_k}$ together with the coupling structure, necessitating delicate regularity arguments in appropriate fractional Sobolev spaces. From a computational perspective, we further explore the application of the iterative regularizing ensemble Kalman method \cite{CIRS18,I15,I16} as a robust and efficient strategy for solving the Bayesian inverse problem in this fractional coupled setting. Owing to its flexibility, the Bayesian framework can also be adapted to complex geometries \cite{ILS16}, heterogeneous media \cite{B25,XYW20,YZW21}, and nonlinear forward models \cite{ZJY18}. To the best of our knowledge, this is the first work combining Bayesian inversion and ensemble Kalman techniques for an inverse source problem governed by a coupled subdiffusion system. The proposed framework not only achieves accurate reconstructions of $\bm\rho$, but also provides reliable uncertainty quantification, which is essential in practical applications where data are sparse and contaminated by noise.

The remaining part of this article is structured as follows. In Section \ref{sec-main}, we fix notations and collect all main conclusions regarding Problems \ref{prob-ISP}--\ref{prob-SPP}, which are proved in Section \ref{sec-proof}. Then Section \ref{sec-Bayes} is devoted to the formulation of a Bayesian framework and the proposal of an iterative regularizing ensemble Kalman method for the numerical reconstruction of Problems \ref{prob-ISP}, followed by Section \ref{sec-numer} implementing various numerical experiments to illustrate the performance of the proposed algorithm. Finally, Section \ref{sec-conclude} closes this article by concluding remarks.


\section{Statements of the Main Results}\label{sec-main}

We start with preparing necessary preliminaries for explaining the main results. Throughout this article, by $(\,\cdot\,,\,\cdot\,)$ we denote the inner product on $L^2(\Om)$. For two vectors $\bm\xi=(\xi_1,\dots,\xi_K)^\T,\bm\ze=(\ze_1,\dots,\ze_K)^\T\in\BR^K$, we denote their Hadamard product as
\[
\bm\xi\odot\bm\zeta:=(\xi_1\ze_1,\dots,\xi_K\ze_K)^\T\in\BR^K.
\]
The convolution of $f,g:(0,+\infty)\longrightarrow\BR$ is defined as
\[
(f*g)(t):=\int_0^t f(t-\tau)g(\tau)\,\rd\tau,
\]
as long as the right-hand side makes sense in some function space. Then the Riemann-Liouville integral operator $J^\be$ with $\be>0$ can be represented as
\begin{equation}\label{eq-def-RL}
J^\be f=k_\be*f,\quad k_\be(t):=\f{t^{\be-1}}{\Ga(\be)}.
\end{equation}
For vector-valued functions $\bm f=(f_1,\dots,f_K)^\T,\bm g=(g_1,\dots,g_K)^\T:(0,T)\longrightarrow\BR^K$, their convolution is defined in a pointwise manner, i.e.,
\[
(\bm f*\bm g)(t)=((f_1*g_1)(t),\dots,(f_K*g_K)(t))^\T=\int_0^t\bm f(t-\tau)\odot\bm g(\tau)\,\rd\tau.
\]
At the same time, we recall the familiar Mittag-Leffler function
\[
E_{\be,\ga}(z):=\sum_{m=0}^\infty\f{z^m}{\Ga(\be m+\ga)},\quad\be>0,\ \ga\in\BR,\ z\in\mathbb C.
\]

Next, we invoke the useful eigensystem $\{(\la_n^{(k)},\vp_n^{(k)})\}_{n=1}^\infty$ of the elliptic operator $\cA_k$ ($k=1,\dots,K$), that is,
\[
\begin{cases}
\cA_k\vp_n^{(k)}=\la_n^{(k)}\vp_n^{(k)} & \mbox{in }\Om,\\
\vp_n^{(k)}=0 & \mbox{on }\pa\Om,
\end{cases}\quad n\in\BN,
\]
where $\{\la_n^{(k)}\}\subset\BR$ satisfies
\[
0<\la_1^{(k)}<\la_2^{(k)}\le\cdots,\quad\la_n^{(k)}\longrightarrow+\infty\ (n\to\infty)
\]
and $\{\vp_n^{(k)}\}\subset D(\cA_k)$ forms a complete orthonormal basis of $L^2(\Om)$. Then following the standard theory of fractional operators, one can define the fractional power $\cA_k^\be$, its domain $D(\cA_k^\be)$ as well as the corresponding norm $\|\cdot\|_{D(\cA_k^\be)}$ for all $\be\ge0$ in the same manner as \cite{LHL23}. In particular, we have
\[
\cA_k^\be\psi=\sum_{n=1}^\infty(\la_n^{(k)})^\be(\psi,\vp_n^{(k)})\vp_n^{(k)},\quad\|\psi\|_{D(\cA_k^\be)}=\|\cA_k^\be\psi\|_{L^2(\Om)}
\]
and the norm equivalence $\|\cdot\|_{D(\cA_k^\be)}=\|\cdot\|_{H^{2\be}(\Om)}$. We further denote $\BA^\be:=\diag(\cA_1^\be,\dots,\cA_K^\be)$ and define $D(\BA^\be)$ along with its norm $\|\cdot\|_{D(\BA^\be)}$ accordingly.

In the sequel, we simply write e.g.\! $\bm u(t)=\bm u(\,\cdot\,,t)$. For a matrix-valued function $\bm\Psi=(\psi_{ij})$, if each entry $\psi_{ij}$ belongs to a Banach space $X$, then we denote $\bm\Psi\in X$ for simplicity. For vectors or vector-valued functions taking values in Euclidean spaces, by default we employ the $\ell^2$ norm, but time by time we shall switch to any $\ell^p$ norm ($p\in[0,\infty]$) for convenience, as all norms in finite dimensional normed spaces are equivalent.

Now we can state the first result on the Lipschitz stability of the inverse source problem under a non-degeneracy condition.

\begin{thm}\label{thm-Lip}
Let $\be>d/4$ be constant and assume
\[
\bm C\in C^\infty(\ov\Om),\quad\bm G\in D(\BA^\be),\quad\bm\rho\in L^\infty(0,T)
\]
in \eqref{eq-IBVP-u0}. If $\det\bm G(\bm x_0)\ne0,$ then there exists a constant $C>0$ such that
\begin{equation}\label{eq-Lip}
\|\bm\rho\|_{L^\infty(0,T)}\le C\|\pa_t^{\bm\al}\bm u(\bm x_0,\,\cdot\,)\|_{L^\infty(0,T)}.
\end{equation}
In particular, the condition $\bm u=\bm0$ at $\{\bm x_0\}\times(0,T)$ implies $\bm\rho=\bm0$ in $(0,T)$.
\end{thm}

The regularity assumption on matrices $\bm C$ and $\bm G$ in Theorem \ref{thm-Lip} guarantees that the governing equation in \eqref{eq-IBVP-u0} is satisfied in a pointwise sense for any $\bm x\in\Om$, so that the observation data $\pa_t^{\bm\al}\bm u(\bm x_0,\,\cdot\,)$ on the right-hand side of \eqref{eq-Lip} is well-defined in $L^\infty(0,T)$. The $C^\infty$ assumption for $\bm C$ can be weakened depending on the spatial dimension $d$, but we refrain from spending further effort on this technical detail.

The Lipschitz stability obtained above relies heavily on the non-degeneracy condition $\det\bm G(\bm x_0)\ne0$, which requires a rather special choice of the observation point $\bm x_0$. For instance, if the known matrix $\bm G$ takes a diagonal form $\bm G=\diag(g_1,\dots,g_K)$, then it is required that $g_k(\bm x_0)\ne0$ or almost equivalently $\bm x_0\in\mathrm{supp}\,g_k$ for all $k=1,\dots,K$. In practical environment problem especially with contaminants, this means the necessity of performing the observation in a polluted area which is possibly dangerous. Therefore, from both theoretical interests and practical needs, it is more preferable to relax the choice of $\bm x_0$ to an arbitrary point in $\Om$. Besides, the observation data in \eqref{eq-Lip} turns out to be the fractional derivatives of $\bm u(\bm x_0,\,\cdot\,)$ instead of $\bm u(\bm x_0,\,\cdot\,)$ itself, which is usually unaffordable in practices in view of the observation noise. Hence, another possible improvement of Theorem \ref{thm-Lip} is the reduction of data regularity as much as possible.

Regarding the same inverse problem for scalar subdiffusion equations, Liu, Rundell and Yamamoto \cite{LRY16} first achieved the above aims and established a general uniqueness result by means of a strong maximum principle and a Duhamel's principle for subdiffusion equations. In a similar manner, we shall first generalize these principles to coupled systems and then apply them to show the uniqueness of our inverse source problem. Moreover, in Theorem \ref{thm-Lip} we managed to identify all $K$ components in $\bm\rho$ by the single point observation of all $K$ components of $\bm u$. Due to the coupling effect among components of $\bm u$, the information of some components is expected to interact with others. Hence, as an issue only for coupled systems, we shall also seek the possibility of determining some components of $\bm\rho$ by observing only some components of $\bm u$.

In the sequel, we slightly simplify the problem formulation by restricting the matrix $\bm G$ in \eqref{eq-IBVP-u0} as a diagonal one $\bm G=\diag(g_1,\dots,g_K)$, and denote $\bm g:=(g_1,\dots,g_K)^\T$. We first attempt to answer Problem \ref{prob-SPP}, i.e., a certain positivity property for the homogeneous problem \eqref{eq-IBVP-v0}. As we will deal with the (strict) positivity issue for vector- and matrix-valued functions, we first fix the notations as follows.

\begin{defi}\label{def-pos}
Let $\bm\Psi=(\psi_{ij})\in\BR^{m\times n}$ and let $\bm O\in\BR^{m\times n}$ denote the zero matrix. We write $\bm\Psi\ge\bm O$  $($respectively $\bm\Psi>\bm O)$ if $\psi_{ij}\ge0$ $($respectively $\psi_{ij}>0)$ for all $i=1,\dots,m$ and $j=1,\dots,n$.
\end{defi}

For the coupled system \eqref{eq-IBVP-v0}, the situation is definitely more complicated than its scalar prototype, so that some extra assumptions on the coefficient matrix $\bm C$ and the initial value $\bm g$ should be imposed. Below we state a weak version of the strict positivity property for \eqref{eq-IBVP-v0}.

\begin{thm}\label{thm-SPP}
Let $\bm g\in L^2(\Om)$ satisfy $\bm g\ge\bm0$ in $\Om$ and $\bm C\in L^\infty(\Om)$ satisfy
\begin{equation}\label{eq-cond-C}
c_{k\ell}\le0\quad\mbox{in }\Om,\ \forall\,k,\ell=1,\dots,K,\ k\ne\ell.
\end{equation}
Define a vector $\bm r=(r_1,\dots,r_K)^\T\in\BR^K$ and a matrix $\bm Q=(q_{k\ell})\in\BR^{K\times K}$ as
\begin{equation}\label{eq-def-rq}
\begin{gathered}
r_k=\begin{cases}
0, & \mbox{if }g_k\equiv0\mbox{ in }\Om,\\
1, & \mbox{else},
\end{cases}\quad q_{kk}=1,\quad k=1,\dots,K,\\
q_{k\ell}=\begin{cases}
0, & \mbox{if }c_{k\ell}\equiv0\mbox{ in }\Om,\\
1, & \mbox{else},
\end{cases}\quad k,\ell=1,\dots,K,\ k\ne\ell.
\end{gathered}
\end{equation}
If there exists an integer $M=0,1,\dots$ such that
\begin{equation}\label{eq-cond-rq}
\bm Q^M\bm r>\bm0,
\end{equation}
then the solution $\bm v$ to \eqref{eq-IBVP-v0} satisfies $J^{M(1-\al_K)}\bm v>\bm0$ a.e. in $\Om\times(0,T)$.
\end{thm}

We regard the above result as a weak strict positivity property, because we only conclude that some Riemann-Liouville integral of the solution $\bm v$ is strictly positive a.e. in $\Om\times(0,T)$ instead of $\bm v$ itself. The key  technical difficulty turns out to be the lack of a strict positivity property for the inhomogeneous problem for a single subdiffusion equation
\begin{equation}\label{eq-single}
\begin{cases}
\pa_t^\al(u-g)+\cA u=F & \mbox{in }\Om\times(0,T),\\
u=0 & \mbox{on }\pa\Om\times(0,T).
\end{cases}
\end{equation}
In fact, for the homogeneous problem, i.e., $F\equiv0$, the strict positivity property is well known for \eqref{eq-single} (e.g.\! \cite{LRY16}). However, in the case of  $a\equiv0$, there seems no result asserting $u>0$ even if $F\ge0,\not\equiv0$. This prevents us from obtaining a real strict positivity property for \eqref{eq-IBVP-v0}. Alternatively, in Corollary \ref{coro-SPP0} we verify the strict positivity of $J^{1-\al}u$ under some non-negativity assumption on $F$, which results in the Riemann-Liouville integral in the conclusion of Theorem \ref{thm-SPP}. In this sense, we fail to exclude the possibility for some components of $\bm v$ to vanish in a subset of $\Om\times(0,T)$, although it seems unlike to happen. Nevertheless, such an intermediate result turns out to be sufficient for proving the uniqueness of the inverse problem, as stated in Theorem \ref{thm-ISP} below.

We explain the usages of assumptions \eqref{eq-cond-C}--\eqref{eq-cond-rq} as follows. First, the non-positivity assumption \eqref{eq-cond-C} on the off-diagonal entries reflects the cooperativeness among components of $\bm v$, which is easily understood if the coupling term $\bm C\bm v$ is moved to the right-hand side of \eqref{eq-IBVP-v0}. The diagonal entries of $\bm C$ will be treated separately and their signs do not matter. Next, the vector $\bm r$ and the matrix $\bm Q$ defined in \eqref{eq-def-rq} play the role of indicators for the identical vanishing of $\bm g$ and positive effect of $-\bm C$, respectively. Indeed, it is readily seen that $r_k=1$ if and only if $g_k\ge0,\not\equiv0$ in $\Om$, indicating that $g_k$ is strictly positive in a subdomain of $\Om$. Therefore, we assume a non-negative initial condition for \eqref{eq-IBVP-v0} while allows some components of $\bm v$ to take homogeneous initial values. Similarly, the off-diagonal entries of $\bm Q$ indicate whether $-c_{k\ell}\ge0,\not\equiv0$ in $\Om$ for $k\ne\ell$.

Since some components of $\bm v$ suffer from vanishing initial conditions, their time evolutions rely completely on the supply from other components. Thus, their strict positivity depends on whether they can receive strictly positive supply from those with non-vanishing initial values. Such a requirement is fulfilled by the key assumption \eqref{eq-cond-rq}, which implies that the information of strict positivity eventually spread throughout all components. Technically, we will construct a non-negative and monotone sequence by the Picard iteration to approximate the mild solution to \eqref{eq-IBVP-v0}, in which \eqref{eq-cond-rq} guarantees the strict positivity of some Riemann-Liouville integral of the sequence after finite steps of iterations.

As a special case of Theorem \ref{thm-SPP}, we immediately obtain the following consequence.

\begin{coro}\label{coro-SPP1}
Let $\bm g\in L^2(\Om)$ and $\bm C\in L^\infty(\Om)$ satisfy \eqref{eq-cond-C}. If $g_k\ge0,\not\equiv0$ in $\Om$ for all $k=1,\dots,K,$ then the solution $\bm v$ to \eqref{eq-IBVP-v0} satisfies $\bm v>\bm0$ a.e. in $\Om\times(0,T)$.
\end{coro}

Indeed, the above corollary assumes that all components of $\bm g$ are non-negative and non-vanishing, indicating $\bm r>\bm0$ in view of definition \eqref{eq-def-rq}. This means $M=0$ and Theorem \ref{thm-SPP} simply implies $\bm v>\bm0$ a.e. in $\Om\times(0,T)$.

Based on Theorem \ref{thm-SPP} and a newly established fractional Duhamel's principle (see Lemma \ref{lem-Duhamel2}) in Section \ref{sec-proof}, it is straightforward to demonstrate the following uniqueness result concerning Problem \ref{prob-ISP} under a special condition on $\bm\rho$. 

\begin{thm}\label{thm-ISP}
Let $\be>d/4-1$ be a constant and $\bm u$ be the solution to
\begin{equation}\label{eq-IBVP-u1}
\begin{cases}
(\pa_t^{\bm\al}+\BA+\bm C)\bm u=\bm g(\bm x)\odot\bm\rho(t) & \mbox{in }\Om\times(0,T),\\
\bm u=\bm0 & \mbox{on }\pa\Om\times(0,T),
\end{cases}
\end{equation}
where
\begin{equation}\label{eq-reg}
\bm C\in C^\infty(\ov\Om),\quad\bm g\in\begin{cases}
L^2(\Om), & d=1,2,3,\\
D(\BA^\be), & d\ge4.
\end{cases}
\end{equation}
Further assume that $\bm C,\bm g$ satisfy {\rm\eqref{eq-cond-C}--\eqref{eq-cond-rq}} and $\bm\rho$ satisfies
\begin{equation}\label{eq-cond-mu}
\exists\,\mu\in W^{1,\infty}(0,T),\ \mu(0)=0\quad\mbox{such that}\quad J^{\al_k}\rho_k=\mu,\ k=1,\dots,K.
\end{equation}
Then for arbitrary $k\in\{1,\dots,K\}$ and arbitrary $\bm x_0\in\Om,$ the condition $u_k=0$ at $\{\bm x_0\}\times(0,T)$ implies $\bm\rho=\bm0$ in $(0,T)$.
\end{thm}

The above theorem greatly relaxes the setting of Theorem \ref{thm-Lip} by observing only an arbitrary component of $\bm u$ at an arbitrary point $\bm x_0$. Such a relaxation is achieved at the cost of downgrading the Lipschitz stability to merely uniqueness and alternatively assuming a non-identically vanishing condition \eqref{eq-cond-rq} on $\bm g$. At a first glance, the above theorem succeeds in identifying all components of unknown $\bm\rho$ simultaneously by the single point observation of a single component of $\bm u$. However, it relies on the restrictive structural constraint \eqref{eq-cond-mu}, which means that components in $\bm\rho$ are no longer independent but are dominated by a common underlying function $\mu$. Such a restriction originates from the process of applying the coupled Duhamel's principle to our problem. As an obvious sufficiently condition of \eqref{eq-cond-mu}, an even special but seemingly acceptable case of \eqref{eq-cond-mu} could be
\[
\al_1=\cdots=\al_K,\quad\rho_1=\cdots=\rho_K,
\]
that is, governing equations of all components share the same fractional order and the temporal part of the source term. Anyway, Theorem \ref{thm-ISP} actually recovers one scalar-valued unknown by observing one scalar-valued function, which coincides with the similar result for a single equation in \cite{LRY16} but technical more involved. Indeed, owing to the non-negativity assumption \eqref{eq-cond-rq}, the observable component $u_k$ is allowed to vanish identically at $t=0$, which is impossible for a single equation.

For the moment, the determination of independent components of $\bm\rho$ by partial solution components seems unavailable by our current methodology, which deserves further investigations in the future. Even so, Theorem \ref{thm-ISP} still greatly improves the result in Theorem \ref{thm-Lip} in the following aspects. First, the regularity of the known component $\bm g$ is lowered, since now it is enough for $\bm u(\bm x_0,\,\cdot\,)$ to make pointwise sense instead of $\pa_t^{\bm\al}\bm u(\bm x_0,\,\cdot\,)$. More importantly, the choice of the observation point $\bm x_0$ can be more flexible.


\section{Proofs of the Main Results}\label{sec-proof}

From now on, by $C>0$ we denote generic constants which may change from line to line.


\subsection{Proof of Theorem \ref{thm-Lip}}

The general strategy for showing Theorem \ref{thm-Lip} follows the same line as those for \cite[Theorem 4.4]{SY11} and \cite[Theorem 5.1]{HLY20}. More precisely, owing to the key assumption $\det\bm G(\bm x_0)\ne0$, we directly take $\bm x=\bm x_0$ in the governing equation of \eqref{eq-IBVP-u0} and multiply $\bm G^{-1}(\bm x_0)$ on both sides to obtain
\begin{align}
|\bm\rho(t)| & =|\bm G^{-1}(\bm x_0)(\pa_t^{\bm\al}+\BA+\bm C)\bm u(\bm x_0,t)|\nonumber\\
& \le C(|\pa_t^{\bm\al}\bm u(\bm x_0,t)|+|\BA\bm u(\bm x_0,t)|+|\bm C\bm u(\bm x_0,t)|).\label{eq-est-rho}
\end{align}
Then the observable data $\pa_t^{\bm\al}\bm u(\bm x_0,t)$ appears on the right-hand side, and it suffices to estimate $(\BA+\bm C)\bm u(\bm x_0,t)$ by some weakly singular integral involving $\bm\rho$. Then the desired inequality \eqref{eq-Lip} follows immediately by applying a general Gr\"onwall's inequality. Different from the scalar-valued case in \cite{SY11}, estimates via explicit solutions are no longer available for coupled systems. Therefore, this section mainly aims at establishing such an estimate for $\bm u(t)$ without using explicit solutions.

To this end, we start with recalling some basic facts on the mild solution to coupled subdiffusion equations established in Li, Huang, Liu \cite{LHL23}. First we invoke the resolvent operator $\cS_k(t):L^2(\Om)\longrightarrow L^2(\Om)$ and its formal derivative $\cS'_k(t):L^2(\Om)\longrightarrow L^2(\Om)$ generated by the eigensystem $\{(\la_n^{(k)},\vp_n^{(k)})\}$ of $\cA_k$ for $k=1,\dots,K$:
\begin{align*}
\cS_k(t)\psi & :=\sum_{n=1}^\infty E_{\al_k,1}(-\la_n^{(k)}t^{\al_k})(\psi,\vp_n^{(k)})\vp_n^{(k)},\\
\cS'_k(t)\psi & :=-t^{\al_k-1}\sum_{n=1}^\infty E_{\al_k,\al_k}(-\la_n^{(k)}t^{\al_k})(\psi,\vp_n^{(k)})\vp_n^{(k)}.
\end{align*}
We recall the following key estimate for $\cS'_k(t)$ from Li, Huang, Yamamoto \cite{LHY20}.

\begin{lem}\label{lem-resolvent}
Let $\be\ge0$ be constant. There exists a constant $C_0>0$ depending only on $\Om,\bm\al,\BA$ such that for any $\ga\in[0,1],$ any $k=1,\dots,K$ and any $\psi\in D(\cA_k^\be),$ there holds
\[
\|\cS'_k(t)\psi\|_{D(\cA_k^{\be+\ga-1})}\le C_0\|\psi\|_{D(\cA_k^\be)}t^{\al_k(1-\ga)-1}.
\]
\end{lem}

The expression above slightly generalizes the original one in \cite{LHY20} which only dealt with $\psi\in L^2(\Om)$. Since the argument is almost identical, we omit the proof here. The essence is the smoothing effect of $\cS'_k(t)$ from the initial regularity of $\psi$. As before, we denote
\[
\BS(t):=\diag(\cS_1(t),\dots,\cS_K(t)),\quad\BA^{-1}\BS'(t):=\diag(\cA_1^{-1}\cS'_1(t),\dots,\cA_K^{-1}\cS'_K(t)).
\]

For later convenience, we recall the definition of a mild solution to a slightly more general problem formulation than \eqref{eq-IBVP-u0}:
\begin{equation}\label{eq-IBVP-u2}
\begin{cases}
\pa_t^{\bm\al}(\bm u-\bm u_0)+(\BA+\bm C)\bm u=\bm F & \mbox{in }\Om\times(0,T),\\
\bm u=\bm0 & \mbox{on }\pa\Om\times(0,T).
\end{cases}
\end{equation}

\begin{defi}[{Mild solution; see \cite[Definition 1]{LHL23}}]
Fix $p\in[1,\infty]$ and let
\[
\bm C\in L^\infty(\Om),\quad\bm u_0\in L^2(\Om),\quad\bm F\in L^p(0,T;L^2(\Om)).
\]
We call $\bm u$ a mild solution to the initial-boundary value problem \eqref{eq-IBVP-u2} if it satisfies the integral equation
\begin{equation}\label{eq-mildsol}
\bm u=\bm w+\BQ\bm u\quad\mbox{in }\Om\times(0,T),
\end{equation}
where
\begin{align}
\bm w(t) & :=\BS(t)\bm u_0-\int_0^t\BA^{-1}\BS'(t-\tau)\bm F(\tau)\,\rd\tau,\label{eq-def-w}\\
\BQ\bm u(t) & :=\int_0^t\BA^{-1}\BS'(t-\tau)\bm C\bm u(\tau)\,\rd\tau.\label{eq-def-Q}
\end{align}
\end{defi}

The following lemma collects basic well-posedness results for \eqref{eq-IBVP-u2}, which, similarly to Lemma \ref{lem-resolvent}, also generalizes the regularities of $\bm u_0$ and $\bm F$ in \cite{LHL23} for flexibility.

\begin{lem}[{see \cite[Theorem 1]{LHL23}}]\label{lem-well}
Fix $\be\ge0,$ $p\in[1,\infty]$ and assume
\[
\bm C\in C^\infty(\ov\Om),\quad\bm u_0\in D(\BA^\be),\quad\bm F\in L^p(0,T;D(\BA^\be)).
\]

{\rm(i)} If $\bm F\equiv\bm0,$ then there exists a unique mild solution
\[
\bm u\in L^{1/\ga}(0,T;D(\BA^{\be+\ga}))
\]
to \eqref{eq-IBVP-u2} for any $\ga\in[0,1],$ where it is understood $1/\ga=\infty$ if $\ga=0$. Moreover, there exists a constant $C>0$ depending only on $\Om,\bm\al,\BA,\bm C,T,\ga$ such that
\begin{gather*}
\|\bm u(t)\|_{D(\BA^{\be+\ga})}\le C\|\bm u_0\|_{D(\BA^\be)}t^{-\al_1\ga},\ 0<t<T,\\
\|\bm u\|_{L^{1/\ga}(0,T;D(\BA^{\be+\ga}))}\le C\|\bm u_0\|_{D(\BA^\be)}.
\end{gather*}

{\rm(ii)} If $\bm u_0\equiv\bm0,$ then there exists a unique mild solution
\[
\bm u\in L^p(0,T;D(\BA^{\be+\ga}))
\]
to \eqref{eq-IBVP-u2} for any $\ga\in[0,1)$ and there exists a constant $C>0$ depending only on $\Om,\bm\al,\BA,\bm C,T,\ga$ such that
\[
\|\bm u\|_{L^p(0,T;D(\BA^{\be+\ga}))}\le C\|\bm F\|_{L^p(0,T;D(\BA^\be))}.
\]
\end{lem}

Although the above lemma already provides useful estimates for $\bm u$, they fail to take the form of weakly singular integrals involving $\bm F$ and thus insufficient for the proof of Theorem \ref{thm-Lip}. Therefore, we should perform more sophisticated estimates for the mild solution, especially that for the operator $\BQ$ defined in \eqref{eq-def-Q}.

\begin{lem}\label{lem-est-Q}
Let $\be\ge0$ be constant and $\bm y\in L^\infty(0,T;D(\BA^\be))$. Then $\BQ^m\bm y\in L^\infty(0,T;D(\BA^\be))$ for any $m\in\BN$. Moreover, there exists a constant $L>0$ depending only on $\Om,\bm\al,\BA,\bm C,T,K$ such that
\begin{equation}\label{eq-est-Q}
\|\BQ^m\bm y(t)\|_{D(\BA^\be)}\le L^m J^{\al_K m}\|\bm y(t)\|_{D(\BA^\be)},\quad0<t<T.
\end{equation}
\end{lem}

\begin{proof}
We show by induction. For $m=1$, first we follow the definition of $\BQ$ and apply Lemma \ref{lem-resolvent} with $\ga=0$ to estimate
\begin{align*}
\|\BQ\bm y(t)\|_{D(\BA^\be)} & =\sum_{k=1}^K\left\|\int_0^t\cA_k^{-1}\cS'_k(t-\tau)(\bm C\bm y)_k(\tau)\,\rd\tau\right\|_{D(\cA_k^\be)}\\
& \le\sum_{k=1}^K\int_0^t\left\|\cA_k^{-1}\cS'_k(t-\tau)(\bm C\bm y)_k(\tau)\right\|_{D(\cA_k^\be)}\,\rd\tau\\
& =\sum_{k=1}^K\int_0^t\left\|\cS'_k(t-\tau)(\bm C\bm y)_k(\tau)\right\|_{D(\cA_k^{\be-1})}\,\rd\tau\\
& \le C_0\sum_{k=1}^K\int_0^t\|(\bm C\bm y)_k(\tau)\|_{D(\cA_k^\be)}(t-\tau)^{\al_k-1}\,\rd\tau,
\end{align*}
where $(\bm C\bm y)_k$ denotes the $k$-th component of $\bm C\bm y$. Since $\bm C$ is smooth, there exists a constant $C_1>0$ depending only on $\bm C$ such that $\|(\bm C\bm y)_k(\tau)\|_{D(\cA_k^\be)}\le C_1\|\bm y(\tau)\|_{D(\cA_k^\be)}$ for any $k=1,\dots,K$. Therefore,
\begin{align}
\|\BQ\bm y(t)\|_{D(\BA^\be)} & \le C_0C_1\sum_{k=1}^K\int_0^t\|\bm y(\tau)\|_{D(\cA_k^\be)}(t-\tau)^{\al_k-1}\,\rd\tau\nonumber\\
& =C_0C_1\sum_{k=1}^K\Ga(\al_k)\int_0^t\|\bm y(t-\tau)\|_{D(\cA_k^\be)}\f{\tau^{\al_k-1}}{\Ga(\al_k)}\,\rd\tau\nonumber\\
& \le C_0C_1\Ga(\al_K)\sum_{k=1}^K\int_0^t\|\bm y(t-\tau)\|_{D(\cA_k^\be)}k_{\al_k}(\tau)\,\rd\tau,\label{eq-est-Q0}
\end{align}
where $k_{\al_k}$ was the integral kernel of $J^{\al_k}$ defined in \eqref{eq-def-RL} Here we estimate the kernel of $J^{\al_k}$ as
\begin{align*}
k_{\al_k}(\tau) & =\f{\tau^{\al_k-1}}{\Ga(\al_k)}\le\left\{\begin{alignedat}{2}
& \f{\tau^{\al_K-1}}{\Ga(\al_1)}, &\quad & 0<\tau\le1,\\
& \f{T^{\al_1-\al_K}}{\Ga(\al_1)}\tau^{\al_K-1}, &\quad &\tau\ge1
\end{alignedat}\right\}\\
& \le\f{\max\{1,T^{\al_1-\al_K}\}}{\Ga(\al_1)}\tau^{\al_K-1}=\f{\Ga(\al_K)}{\Ga(\al_1)}\max\{1,T^{\al_1-\al_K}\}\,k_{\al_K}(\tau).
\end{align*}
Substituting the above estimate into \eqref{eq-est-Q0}, we obtain
\begin{align}
& \quad\,\|\BQ\bm y(t)\|_{D(\BA^\be)}\nonumber\\
& \le C_0C_1\f{\Ga^2(\al_K)}{\Ga(\al_1)}\max\{1,T^{\al_1-\al_K}\}\sum_{k=1}^K\int_0^t\|\bm y(t-\tau)\|_{D(\BA^\be)}k_{\al_K}(\tau)\,\rd\tau\nonumber\\
& =L J^{\al_K}\|\bm y(t)\|_{D(\BA^\be)},\label{eq-est-Q1}
\end{align}
where
\[
L:=C_0C_1\f{\Ga^2(\al_K)}{\Ga(\al_1)}\max\{1,T^{\al_1-\al_K}\}K.
\]
Then \eqref{eq-est-Q} holds for $m=1$.

For $m\ge2$, assume that \eqref{eq-est-Q} holds until $m-1$. Then we combine \eqref{eq-est-Q} with $m-1$ and \eqref{eq-est-Q1} to deduce
\begin{align*}
\|\BQ^m\bm y(t)\|_{D(\BA^\be)} & =\|\BQ^{m-1}(\BQ\bm y)(t)\|_{D(\BA^\be)}\le L^{m-1}J^{\al_K(m-1)}\left(\|\BQ\bm y(t)\|_{D(\BA^\be)}\right)\\
& \le L^{m-1}J^{\al_K(m-1)}\left(L J^{\al_K}\|\bm y(t)\|_{D(\BA^\be)}\right)\\
& =L^m J^{\al_K m}\|\bm y(t)\|_{D(\BA^\be)}
\end{align*}
by the associative property of $J^{\al_K}$. This completes the proof of \eqref{eq-est-Q} for general $m\ge2$.
\end{proof}

In the same manner, one can estimate $\bm w$ defined in \eqref{eq-def-w} and we skip the proof here.

\begin{lem}\label{lem-est-w}
Let $\bm w$ be defined in $\eqref{eq-def-w},$ where $\bm u_0\equiv\bm0$ and $\bm F\in L^\infty(0,T;D(\BA^\be))$ with a constant $\be\ge0$. Then $\bm w\in L^\infty(0,T;D(\BA^{\be+\ga}))$ for any $\ga\in[0,1)$ and there exists a constant $L'>0$ depending only on $\Om,\bm\al,\BA,T,K$ such that
\[
\|\bm w(t)\|_{D(\BA^{\be+\ga})}\le L'J^{\al_K(1-\ga)}\|\bm F(t)\|_{D(\BA^\be)},\quad0<t<T.
\]
\end{lem}

Now we discuss a key observation from the integral equation \eqref{eq-mildsol} satisfied by the mild solution $\bm u$ to \eqref{eq-IBVP-u2}. Substituting \eqref{eq-mildsol} to the right-hand side of itself repeatedly, we immediately see that $\bm u$ satisfies
\[
\bm u=\sum_{\ell=0}^{m-1}\BQ^\ell\bm w+\BQ^m\bm u,\quad m\in\BN.
\]
Then one is encouraged to pass $m\to\infty$ above to conjecture a closed form for $\bm u$ as $\bm u=\sum_{m=0}^\infty\BQ^m\bm w$. Based on the above two lemmas, we can validate the existence of such a limit in suitable function spaces and arrive at the following conclusion.

\begin{prop}\label{prop-mildsol}
Let $\bm u$ be the mild solution to $\eqref{eq-IBVP-u2},$ where $\bm u_0\equiv\bm0$ and $\bm F\in L^\infty(0,T;D(\BA^\be))$ with a constant $\be\ge0$. Then $\bm u$ allows a series representation
\begin{equation}\label{eq-series}
\bm u=\sum_{m=0}^\infty\BQ^m\bm w\quad\mbox{in }L^\infty(0,T;D(\BA^{\be+\ga}))
\end{equation}
for any $\ga\in[0,1),$ where $\bm w$ and $\BQ$ were defined in \eqref{eq-def-w} and \eqref{eq-def-Q} respectively. Moreover, there exists a constant $C>0$ depending only on $\Om,\bm\al,\BA,\bm C,T,K,\ga$ such that
\begin{equation}\label{eq-est-u}
\|\bm u(t)\|_{D(\BA^{\be+\ga})}\le CJ^{\al_K(1-\ga)}\|\bm F(t)\|_{D(\BA^\be)},\quad0<t<T.
\end{equation}
\end{prop}

\begin{proof}
Fix any $\ga\in[0,1)$. According to the previous discussion, we shall show $\lim_{m\to\infty}\BQ^m\bm u=0$ and the convergence of $\sum_{m=0}^\infty\BQ^m\bm w$ in $L^\infty(0,T;D(\BA^{\be+\ga}))$.

By Lemma \ref{lem-well}(ii), we already know $\bm u\in L^\infty(0,T;D(\BA^{\be+\ga}))$. Replacing $\be$ in Lemma \ref{lem-est-Q} by $\be+\ga$, we can directly estimate
\begin{align*}
\|\BQ^m\bm u(t)\|_{D(\BA^{\be+\ga})} & \le L^m J^{\al_K m}\|\bm u(t)\|_{D(\BA^{\be+\ga})}\\
& \le\|\bm u\|_{L^\infty(0,T;D(\BA^{\be+\ga}))}L^m(J^{\al_K m}1)(t)\\
& =\|\bm u\|_{L^\infty(0,T;D(\BA^{\be+\ga}))}L^m\f{t^{\al_K m}}{\Ga(\al_K m+1)}\\
& \le\|\bm u\|_{L^\infty(0,T;D(\BA^{\be+\ga}))}\f{(L T^{\al_K})^m}{\Ga(\al_K m+1)}\longrightarrow0
\end{align*}
as $m\to\infty$ by Stirling's formula. Meanwhile, it is obvious that $\bm w\in L^\infty(0,T;$ $D(\BA^{\be+\ga}))$ and repeating the same estimate above yields
\[
\|\BQ^m\bm w(t)\|_{D(\BA^{\be+\ga})}\le\|\bm w\|_{L^\infty(0,T;D(\BA^{\be+\ga}))}\f{(L T^{\al_K})^m}{\Ga(\al_K m+1)}
\]
for any $m\in\BN$. Recalling the definition of the Mittag-Leffler functions, we sum up the above inequality with respect to $m$ to bound
\begin{align*}
\sum_{m=0}^\infty\|\BQ^m\bm w(t)\|_{D(\BA^{\be+\ga})} & \le\|\bm w\|_{L^\infty(0,T;D(\BA^{\be+\ga}))}\sum_{m=0}^\infty\f{(L T^{\al_K})^m}{\Ga(\al_K m+1)}\\
& =\|\bm w\|_{L^\infty(0,T;D(\BA^{\be+\ga}))}E_{\al_K,1}(L T^{\al_K})<\infty,
\end{align*}
indicating the convergence of $\sum_{m=0}^\infty\BQ^m\bm w$ in $L^\infty(0,T;D(\BA^{\be+\ga}))$ by the Weierstrass $M$-test. This validates the series representation \eqref{eq-series}.

Now it remains to establish the estimate \eqref{eq-est-u} on the basis of \eqref{eq-series}. Combining a similar estimate as above with Lemma \ref{lem-est-w}, again we employ the Mittag-Leffler function to dominate
\begin{align*}
& \quad\,\|\bm u(t)\|_{D(\BA^{\be+\ga})}\\
& \le\sum_{m=0}^\infty\|\BQ^m\bm w(t)\|_{D(\BA^{\be+\ga})}\le\sum_{m=0}^\infty L^m J^{\al_K m}\|\bm w(t)\|_{D(\BA^{\be+\ga})}\\
& \le L'\sum_{m=0}^\infty L^m J^{\al_K(m+1-\ga)}\|\bm F(t)\|_{D(\BA^\be)}\\
& =L'\sum_{m=0}^\infty\int_0^t\f{L^m\tau^{\al_K(m+1-\ga)-1}}{\Ga(\al_K(m+1-\ga))}\|\bm F(t-\tau)\|_{D(\BA^\be)}\,\rd\tau\\
& =L'\int_0^t\tau^{\al_K(1-\ga)-1}\sum_{m=0}^\infty\f{(L\tau^{\al_K})^m}{\Ga(\al_K m+\al_K(1-\ga))}\|\bm F(t-\tau)\|_{D(\BA^\be)}\,\rd\tau\\
& =L'\int_0^t\tau^{\al_K(1-\ga)-1}E_{\al_K,\al_K(1-\ga)}(L\tau^{\al_K})\|\bm F(t-\tau)\|_{D(\BA^\be)}\,\rd\tau\\
& \le L'E_{\al_K,\al_K(1-\ga)}(L T^{\al_K})\int_0^t\tau^{\al_K(1-\ga)-1}\|\bm F(t-\tau)\|_{D(\BA^\be)}\,\rd\tau\\
& =CJ^{\al_K(1-\ga)}\|\bm F(t)\|_{D(\BA^\be)},
\end{align*}
where
\[
C:=L'\Ga(\al_K(1-\ga))E_{\al_K,\al_K(1-\ga)}(L T^{\al_K}).
\]
This completes the proof of \eqref{eq-est-u}.
\end{proof}

Now we are well prepared to proceed to the proof of Theorem \ref{thm-Lip}.

To begin with, we first confirm that all terms on the right-hand side of \eqref{eq-est-rho} are well defined for a.e. $t\in(0,T)$. By the Sobolev embedding theorem and $\be>d/4$, we see that $\bm G\in D(\BA^\be)\subset C(\ov\Om)$ is defined pointwise. Next, we set
\[
\ve:=\f12\left(\be-\f d4\right)>0.
\]
By $\bm G\bm\rho\in L^\infty(0,T;D(\BA^\be))$, it follows from Lemma \ref{lem-well}(ii) that
\[
\bm u\in L^\infty(0,T;D(\BA^{d/4+1+\ve}))\subset L^\infty(0,T;C^2(\ov\Om)).
\]
Then the governing equation in \eqref{eq-IBVP-u0} holds for any $\bm x\in\ov\Om$ and a.e. $t\in(0,T)$. Especially, we have $\BA\bm u(\bm x_0,\,\cdot\,)\in L^\infty(0,T)$ and thus $\pa_t^{\bm\al}\bm u(\bm x_0,\,\cdot\,)\in L^\infty(0,T)$, i.e., the observation data is well-defined.

Now by $d/4+\ve=\be-\ve$, it follows from the Sobolev embedding theorem that
\[
|\BA\bm u(\bm x_0,t)|\le\|\BA\bm u(\,\cdot\,,t)\|_{C(\ov\Om)}\le C\|\BA\bm u(\,\cdot\,,t)\|_{D(\BA^{d/4+\ve})}=C\|\bm u(\,\cdot\,,t)\|_{D(\BA^{\be+1-\ve})}.
\]
Then we can choose $\ga=1-\ve$ and assume $\ga\ge0$ without loss of generality in Proposition \ref{prop-mildsol} to estimate
\[
|\BA\bm u(\bm x_0,t)|\le C J^{\al_K\ve}\|\bm G\bm\rho(t)\|_{D(\BA^\be)}\le C J^{\al_K\ve}|\bm\rho(t)|.
\]
Meanwhile, since $\bm C\bm u(\bm x_0,t)$ is the zeroth order term, it follows immediately from the above result that
\[
|\bm C\bm u(\bm x_0,t)|\le C|\bm u(\bm x_0,t)|\le C\|\bm u(t)\|_{D(\BA^{\be+1-\ve})}\le C J^{\al_K\ve}|\bm\rho(t)|.
\]
Hence, plugging the above two inequalities into \eqref{eq-est-rho} leads us to
\[
|\bm\rho(t)|\le C|\pa_t^{\bm\al}\bm u(\bm x_0,t)|+C J^{\al_K\ve}|\bm\rho(t)|.
\]
Finally, we take advantage of the general Gr\"onwall's inequality in Henry \cite[Lemma 7.1.1]{H81} to obtain
\[
|\bm\rho(t)|\le C|\pa_t^{\bm\al}\bm u(\bm x_0,t)|+C J^{\al_K\ve}|\pa_t^{\bm\al}\bm u(\bm x_0,t)|\le C\|\pa_t^{\bm\al}\bm u(\bm x_0,\,\cdot\,)\|_{L^\infty(0,t)}
\]
for a.e. $t\in(0,T)$, which consequently implies \eqref{eq-Lip}.


\subsection{Proof of Theorem \ref{thm-SPP}}

To investigate the strict positivity property of the homogeneous problem \eqref{eq-IBVP-v0}, we shall first prepare an auxiliary fractional Duhamel's principle for the scalar-valued problem
\begin{equation}\label{eq-IBVP-u3}
\begin{cases}
\pa_t^\al u+\cA u=F & \mbox{in }\Om\times(0,T),\\
u=0 & \mbox{on }\pa\Om\times(0,T).
\end{cases}
\end{equation}
Here $0<\al<1$ and $\cA:H^2(\Om)\cap H_0^1(\Om)\longrightarrow L^2(\Om)$ is an elliptic operator similarly to $\cA_k$ introduced previously but with a zeroth order term, i.e.,
\[
\cA\psi:=-\mathrm{div}(\bm A(\bm x)\nb\psi)+c(\bm x)\psi,
\]
where $0\le c\in L^\infty(0,T)$.

\begin{lem}\label{lem-Duhamel1}
Let $u$ be the solution to \eqref{eq-IBVP-u3} with $F\in L^\infty(0,T;L^2(\Om))$. Then $u$ satisfies
\begin{equation}\label{eq-rep-uv1}
J^{1-\al}u(t)=\int_0^t v(t;s)\,\rd s,
\end{equation}
where $v(t;s)$ solves the homogeneous problem
\begin{equation}\label{eq-IBVP-v1}
\begin{cases}
\pa_{s+}^\al(v-F(s))+\cA v=0 & \mbox{in }\Om\times(s,T),\\
v=0 & \mbox{on }\pa\Om\times(s,T)
\end{cases}
\end{equation}
with a parameter $s\in(0,T)$. Here $\pa_{s+}^\al$ is the inverse of a shifted Riemann-Liouville integral operator defined as
\[
J_{s+}^\al f(t):=\int_s^t\f{(t-\tau)^{\al-1}}{\Ga(\al)}f(\tau)\,\rd\tau.
\]
\end{lem}

\begin{proof}
As special cases of Lemma \ref{lem-well}, we know $u$ and $v(\,\cdot\,;s)$ ($0<s<T$) belong to $L^\infty(0,T;L^2(\Om))$. As before, we introduce the eigensystem $\{(\la_n,\vp_n)\}$ of $\cA$, which shares identically the same properties as those of $\cA_k$ owing to the non-negativity of the zeroth order coefficient $c$. Then according to \cite{SY11}, the explicit solutions to \eqref{eq-IBVP-v1} and \eqref{eq-IBVP-u3} read
\begin{align*}
v(t;s) & =\sum_{n=1}^\infty E_{\al,1}(-\la_n(t-s)^\al)(F(s),\vp_n)\vp_n,\\
u(t) & =\sum_{n=1}^\infty\left(\int_0^t s^{\al-1}E_{\al,\al}(-\la_n s^\al)(F(t-s),\vp_n)\,\rd s\right)\vp_n.
\end{align*}
Writing $f_n(t):=(F(t),\vp_n)$ and $h_n(t):=t^{\al-1}E_{\al,\al}(-\la_n t^\al)$, we can express
\[
v(t;s)=\sum_{n=1}^\infty E_{\al,1}(-\la_n(t-s)^\al)f_n(s)\vp_n,\quad u(t)=\sum_{n=1}^\infty(h_n*f_n)(t)\vp_n.
\]
Recalling the convolution representation of $J^\be$ in \eqref{eq-def-RL}, we have
\[
J^{1-\al}u(t)=(k_{1-\al}*u)(t)=\sum_{n=1}^\infty(k_{1-\al}*h_n*f_n)(t)\vp_n.
\]
Here we calculate
\begin{align*}
(k_{1-\al}*h_n)(t) & =\int_0^t\f{(t-s)^{-\al}}{\Ga(1-\al)}s^{\al-1}\sum_{m=0}^\infty\f{(-\la_n s^\al)^m}{\Ga(\al m+\al)}\,\rd s\\
& =\f1{\Ga(1-\al)}\sum_{m=0}^\infty\f{(-\la_n)^m}{\Ga(\al m+\al)}\int_0^t(t-s)^{-\al}s^{\al(m+1)-1}\,\rd s\\
& =\f1{\Ga(1-\al)}\sum_{m=0}^\infty\f{(-\la_n)^m}{\Ga(\al m+\al)}t^{\al m}B(1-\al,\al m+\al)\\
& =\sum_{m=0}^\infty\f{(-\la_n t^\al)^m}{\Ga(\al m+1)}=E_{\al,1}(-\la_n t^\al).
\end{align*}
Therefore, it turns out that
\begin{align*}
J^{1-\al}u(t) & =\sum_{n=1}^\infty(E_{\al,1}(-\la_n t^\al)*f_n)(t)\vp_n\\
& =\int_0^t\sum_{n=1}^\infty E_{\al,1}(-\la_n(t-s)^\al)f_n(s)\vp_n\,\rd s=\int_0^t v(t;s)\,\rd s,
\end{align*}
which is exactly \eqref{eq-rep-uv1} and thus verifies Lemma \ref{lem-Duhamel1}.
\end{proof}

As a quick consequence, we have the following corollary.

\begin{coro}\label{coro-SPP0}
Under the same setting of Lemma \ref{lem-Duhamel1}, further assume that
\begin{equation}\label{eq-pos-F}
F(s)\ge0,\not\equiv0\quad\mbox{in }\Om,\ \mbox{a.e. }s\in(0,T).
\end{equation}
Then $J^{1-\al}u>0$ a.e. in $\Om\times(0,T)$.
\end{coro}

\begin{proof}
By Lemma \ref{lem-Duhamel1}, $u$ satisfies \eqref{eq-rep-uv1} with $v(t;s)$ satisfying \eqref{eq-IBVP-v1}. Thanks to the positivity \eqref{eq-pos-F}, the strong maximum principle (e.g.\! \cite[Theorem 1]{LRY16}) guarantees
\[
v(\,\cdot\,;s)>0\quad\mbox{a.e. in }\Om\times(s,T),\ \forall\,s\in(0,T),
\]
which directly implies $J^{1-\al}u>0$ a.e. in $\Om\times(0,T)$.
\end{proof}

Next, as an intermediate step before proceeding to the proof of Theorem \ref{thm-SPP}, we prepare the weak maximum principle or equivalently the comparison principle for the coupled system \eqref{eq-IBVP-u2}.

\begin{prop}[Weak maximum principle]\label{prop-WMP}
Let $\bm u$ be the mild solution to $\eqref{eq-IBVP-u2},$ where
\[
\bm0\le\bm u_0\in L^2(\Om),\quad\bm0\le\bm F\in L^\infty(0,T;L^2(\Om))
\]
and $\bm C\in L^\infty(\Om)$ satisfy \eqref{eq-cond-C}. Then $\bm u\ge\bm0$ in $\Om\times(0,T)$.
\end{prop}

The above proposition generalizes a similar result in Luchko and Yamamoto \cite[Theorem 4]{LY25} which assumed
\[
\cA_1=\cdots=\cA_K=-\triangle=-\sum_{j=1}^d\pa_{x_j}^2.
\]
Although the argument is analogous, we still provide a proof here for completeness.

\begin{proof}[Proof of Proposition \ref{prop-WMP}]
According to Lemma \ref{lem-well}, problem \eqref{eq-IBVP-u2} owns a unique mild solution $\bm u\in L^\infty(0,T;$ $L^2(\Om))$, which can be approximated by a Picard iteration based on the integral equation \eqref{eq-mildsol}. Here we attempt to construct a sequence $\{\bm u^{(m)}\}_{m=0}^\infty$ convergent to $\bm u$ in a slightly different manner by decomposing the coefficient matrix $\bm C$ as $\bm C=\bm\Si-\bm B$ with
\[
\bm\Si:=\diag(\|c_{11}\|_{L^\infty(\Om)},\dots,\|c_{KK}\|_{L^\infty(\Om)}),\quad\bm B:=\bm\Si-\bm C.
\]
Then obviously $\bm\Si\ge\bm O$ (see Definition \ref{def-pos}) and owing to the key positivity assumption \eqref{eq-cond-C}, we have $\bm B\ge\bm O$. Further introducing
\[
\wt\cA_k:=\cA_k+\|c_{kk}\|_{L^\infty(\Om)},\ k=1,\dots,K,\quad\wt\BA:=\BA+\bm\Si=\diag(\wt\cA_1,\dots,\wt\cA_K),
\]
we can rewrite \eqref{eq-IBVP-u2} as
\[
\begin{cases}
\pa_t^{\bm\al}(\bm u-\bm u_0)+\wt\BA\bm u=\bm B\bm u+\bm F & \mbox{in }\Om\times(0,T),\\
\bm u=\bm0 & \mbox{on }\pa\Om\times(0,T).
\end{cases}
\]
Here the elliptic operators $\wt\cA_k$ and thus $\wt\BA$ share the same spectral property as that of $\cA_k,\BA$ defined previously due to the non-negativity of $\bm\Si$. Regarding $\bm B\bm u+\bm F$ on the right-hand side above as a new source term, we construct a sequence $\{\bm u^{(m)}\}$ iteratively by $\bm u^{(0)}=\bm0$ and $\bm u^{(m)}$ is the solution to the decoupled system
\begin{equation}\label{eq-IBVP-um}
\begin{cases}
\pa_t^{\bm\al}(\bm u^{(m)}-\bm u_0)+\wt\BA\bm u^{(m)}=\bm B\bm u^{(m-1)}+\bm F & \mbox{in }\Om\times(0,T),\\
\bm u^{(m)}=\bm0 & \mbox{on }\pa\Om\times(0,T).
\end{cases}
\end{equation}
for $m\in\BN$. Obviously, such a construction is equivalent to the recurrence formula in \cite[(3.4)]{LHL23} by a simple amendment of the resolvent operator $\BS$ based on $\wt\BA$, but we prefer the above expression as it is more suitable for discussing the positivity. Repeating the same argument as that in \cite{LHL23}, one can easily confirm that
\[
\bm u^{(m)}\longrightarrow\bm u\quad\mbox{in }L^\infty(0,T;L^2(\Om))\mbox{ as }m\to\infty.
\]

Now it suffices to show that $\{\bm u^{(m)}\}$ is a non-decreasing sequence, that is, $\bm u^{(m)}\ge\bm u^{(m-1)}$ in $\Om\times(0,T)$ for any $m\in\BN$. For $m=1$, it is readily seen that $\bm u^{(1)}$ solves
\[
\begin{cases}
\pa_t^{\bm\al}(\bm u^{(1)}-\bm u_0)+\wt\BA\bm u^{(1)}=\bm F & \mbox{in }\Om\times(0,T),\\
\bm u^{(1)}=\bm0 & \mbox{on }\pa\Om\times(0,T).
\end{cases}
\]
By $\bm u_0\ge\bm0$ and $\bm F\ge\bm0$, it follows directly from the weak maximum principle for scalar-valued subdiffusion equations (e.g., Luchko and Yamamoto \cite{LY19}) that $\bm u^{(1)}-\bm u^{(0)}=\bm u^{(1)}\ge\bm0$.

Now we assume $\bm u^{(m)}\ge\bm u^{(m-1)}$ in $\Om\times(0,T)$ for some $m\in\BN$. Taking difference between \eqref{eq-IBVP-um} with $m+1$ and $m$, we readily see that $\bm u^{(m+1)}-\bm u^{(m)}$ satisfies
\[
\begin{cases}
(\pa_t^{\bm\al}+\wt\BA)(\bm u^{(m+1)}-\bm u^{(m)})=\bm B(\bm u^{(m)}-\bm u^{(m-1)}) & \mbox{in }\Om\times(0,T),\\
\bm u^{(m+1)}-\bm u^{(m)}=\bm0 & \mbox{on }\pa\Om\times(0,T).
\end{cases}
\]
Then the non-negativity of $\bm B$ and $\bm u^{(m)}-\bm u^{(m-1)}$ indicates $\bm B(\bm u^{(m)}-\bm u^{(m-1)})\ge\bm0$ and thus $\bm u^{(m+1)}-\bm u^{(m)}\ge\bm0$ in $\Om\times(0,T)$ again by the scalar-valued weak maximum principle. Then we verified the monotone convergence of $\{\bm u^{(m)}\}$, and especially the limit $\bm u\ge\bm u^{(0)}=\bm0$ as desired.
\end{proof}

Now we are in a position to demonstrate Theorem \ref{thm-SPP}.

Again by Lemma \ref{lem-well}(i), problem \eqref{eq-IBVP-v0} owns a unique mild solution $\bm v\in L^\infty(0,T;L^2(\Om))$. Following the construction in Proposition \ref{prop-WMP}, we can approximate $\bm v$ via a sequence $\{\bm v^{(m)}\}$ defined by $\bm v^{(0)}=\bm0$ and $\bm v^{(m)}$ being the solution to
\[
\begin{cases}
\pa_t^{\bm\al}(\bm v^{(m)}-\bm g)+\wt\BA\bm v^{(m)}=\bm B\bm v^{(m-1)} & \mbox{in }\Om\times(0,T),\\
\bm v^{(m)}=\bm0 & \mbox{on }\pa\Om\times(0,T).
\end{cases}
\]
Since $\bm g\ge\bm0$, it follows from Proposition \ref{prop-WMP} as well as its proof that $\bm v\ge\bm0$ and $\{\bm v^{(m)}\}$ is a non-decreasing sequence. By recalling the integer $M$ in the key assumption \eqref{eq-cond-rq}, our target is to show
\begin{equation}\label{eq-pos-vM}
J^{M(1-\al_K)}\bm v^{(M+1)}>\bm0\quad\mbox{a.e. in }\Om\times(0,T).
\end{equation}
In fact, if \eqref{eq-pos-vM} holds true, then we can employ $\bm v\ge\bm v^{(M+1)}\ge\bm0$ to obtain immediately
\[
J^{M(1-\al_K)}\bm v\ge J^{M(1-\al_K)}\bm v^{(M+1)}>\bm0\quad\mbox{ a.e. in }\Om\times(0,T).
\]

To this end, we turn to the definition \eqref{eq-def-rq} of $\bm r$ and $\bm Q$, which first gives $\bm Q^m\bm r\ge\bm0$ for any $m=0,1,\dots$. Defining the index sets
\[
I_m:=\{k=1,\dots,K\mid(\bm Q^m\bm r)_k\ne0\},\quad m=0,1,\dots,
\]
we see from the key assumption \eqref{eq-cond-rq} that
\[
I_0\subset I_1\subset\cdots\subset I_M=\{1,\dots,K\}.
\]
In the sequel, we claim by induction that
\begin{equation}\label{eq-claim}
\forall\,m=0,1,\dots,M,\ \forall\,k\in I_m,\quad J^{m(1-\al_K)}v_k^{(m+1)}>0\mbox{ a.e. in }\Om\times(0,T).
\end{equation}

For $m=0$, by definition we know $k\in I_0$ if and only if $g_k\ge0,\not\equiv0$ in $\Om$. Since $v_k^{(1)}$ satisfies
\[
\begin{cases}
\pa_t^{\al_k}(v_k^{(1)}-g_k)+\wt\cA_k v_k^{(1)}=0 & \mbox{in }\Om\times(0,T),\\
v_k^{(1)}=0 & \mbox{on }\pa\Om\times(0,T),
\end{cases}
\]
the strong maximum principle for scalar subdiffusion equation (e.g.\! \cite{LRY16}) asserts $v_k^{(1)}>0$ a.e. in $\Om\times(0,T)$ for any $k\in I_0$.

Now we assume that \eqref{eq-claim} holds true for some $m-1=0,\dots,M-1$, that is,
\[
\forall\,k\in I_{m-1},\quad J^{(m-1)(1-\al_K)}v_k^{(m)}>0\mbox{ a.e. in }\Om\times(0,T).
\]
Then the monotonicity of $\{v_k^{(m)}\}$ and the associative law of Riemann-Liouville integral operators yield
\begin{align}
J^{m(1-\al_K)}v_k^{(m+1)} &\ge J^{m(1-\al_K)}v_k^{(m)}=J^{1-\al_K}\left(J^{(m-1)(1-\al_K)}v_k^{(m)}\right)\nonumber\\
& >0\quad\mbox{a.e. in }\Om\times(0,T),\ \forall\,k\in I_{m-1}.\label{eq-pos}
\end{align}
Suppose $I_m\setminus I_{m-1}\ne\emptyset$ without loss of generality and fix any $k\in I_m\setminus I_{m-1}$. By definition, it can be inferred that
\begin{equation}\label{eq-pos-B}
\left(\bm B J^{(m-1)(1-\al_K)}\bm v^{(m)}\right)_k\ge0,\not\equiv0\mbox{ in }\Om\times(0,T),\quad g_k\equiv0\mbox{ in }\Om.
\end{equation}
Then $v_k^{(m+1)}$ satisfies the initial-boundary value problem
\[
\begin{cases}
(\pa_t^{\al_k}+\wt\cA_k)v_k^{(m+1)}=(\bm B\bm v^{(m)})_k & \mbox{in }\Om\times(0,T),\\
v_k^{(m+1)}=0 & \mbox{on }\pa\Om\times(0,T).
\end{cases}
\]
By taking $J^{(m-1)(1-\al_K)}$ on both sides of the governing equation above, it reveals that $J^{(m-1)(1-\al_K)}v_k^{(m+1)}$ satisfies
\[
\begin{cases}
(\pa_t^{\al_k}+\wt\cA_k)\left(J^{(m-1)(1-\al_K)}v_k^{(m+1)}\right)\\
\quad=\left(\bm B J^{(m-1)(1-\al_K)}\bm v^{(m)}\right)_k & \mbox{in }\Om\times(0,T),\\
v_k^{(m+1)}=0 & \mbox{on }\pa\Om\times(0,T),
\end{cases}
\]
where $\pa_t^{\al_k}$ and $J^{(m-1)(1-\al_K)}$ commute due to $g_k=0$ in $\Om$. Now that the right-hand side above satisfies \eqref{eq-pos-B}, we take advantage of Corollary \ref{coro-SPP0} to conclude
\[
J^{1-\al_k}\left(J^{(m-1)(1-\al_K)}v_k^{(m+1)}\right)>0\quad\mbox{a.e. in }\Om\times(0,T)
\]
and hence
\begin{align*}
J^{m(1-\al_K)}v_k^{(m+1)} & =J^{\al_k-\al_K}\left(J^{1-\al_k}\left(J^{(m-1)(1-\al_K)}v_k^{(m+1)}\right)\right)\\
& >0\quad\mbox{a.e. in }\Om\times(0,T),\ \forall\,k\in I_m\setminus I_{m-1}.
\end{align*}
Combining this with \eqref{eq-pos}, we arrive at the claim \eqref{eq-claim} for $m$. Finally, taking $m=M$ in \eqref{eq-claim} yields
\[
J^{M(1-\al_K)}v_k^{(M+1)}>0\quad\mbox{a.e. in }\Om\times(0,T),\ \forall\,k\in I_M=\{1,\dots,K\}
\]
or equivalently \eqref{eq-pos-vM}, which completes the proof of Theorem \ref{thm-SPP}.


\subsection{Proof of Theorem \ref{thm-ISP}}

This subsection is devoted to the proof of the unique determination of $\bm\rho(t)$ in the source term of \eqref{eq-IBVP-u1} by observing a single component $u_k(\bm x_0,\,\cdot\,)$ for any $k=1,\dots,K$. As was discussed in Section \ref{sec-main}, the proof relies on the strict positivity property and the fractional Duhamel's principle in the case of a single equation. In the coupled scenario, the former was established in the previous section and we shall fill the missing piece for the latter. For single time-fractional evolution equations, there are several publications on Duhamel's principle such as \cite{US06,JLLY17,HLY20} and see Umarov \cite{U19} for a comprehensive survey. Motivated by a very recent preprint \cite{U26} for fully coupled systems of fractional order, we propose the following Duhamel's principle for coupled subdiffusion equations.

\begin{lem}\label{lem-Duhamel2}
Let $\bm u$ be the mild solution to
\begin{equation}\label{eq-IBVP-u4}
\begin{cases}
(\pa_t^{\bm\al}+\BA+\bm C)\bm u=\bm F & \mbox{in }\Om\times(0,T),\\
\bm u=\bm0 & \mbox{on }\pa\Om\times(0,T),
\end{cases}
\end{equation}
where $\bm C\in L^\infty(\Om)$ and $\bm F\in W^{1,\infty}(0,T;L^2(\Om))$. Then $\bm u$ allows the representation
\begin{equation}\label{eq-rep-uv2}
\bm u(t)=\int_0^t\bm v(t;s)\,\rd s,
\end{equation}
where $\bm v(t;s)$ solves the homogeneous problem
\begin{equation}\label{eq-IBVP-v2}
\begin{cases}
\pa_{s+}^{\bm\al}(\bm v(\,\cdot\,;s)-D_{0+}^{1-\bm\al}\bm F(s))+(\BA+\bm C)\bm v(\,\cdot\,;s)=\bm0 & \mbox{in }\Om\times(s,T),\\
\bm v(\,\cdot\,;s)=\bm0 & \mbox{on }\pa\Om\times(s,T)
\end{cases}
\end{equation}
where a parameter $s\in(0,T)$. Here
\[
\pa_{s+}^{\bm\al}=\diag(\pa_{s+}^{\al_1},\dots,\pa_{s+}^{\al_K}),\quad D_{0+}^{1-\bm\al}=\diag(D_{0+}^{1-\al_1},\dots,D_{0+}^{1-\al_K}),
\]
where $\pa_{s+}^{\al_k}$ was defined in Lemma $\ref{lem-Duhamel1}$ and $D_{0+}^{1-\al_k}:=\pa_t\circ J_{0+}^{\al_k}$ denotes the Riemann-Liouville derivative.
\end{lem}

\begin{proof}
By $\bm F\in W^{1,\infty}(0,T;L^2(\Om))$, we have $\bm u\in L^\infty(0,T;L^2(\Om))$ by Lemma \ref{lem-well}. Meanwhile, it is not difficult to verify $D_{0+}^{1-\bm\al}\bm F(s)\in L^2(\Om)$ and we see $\bm v(\,\cdot\,;s)\in L^\infty(0,T;L^2(\Om))$ for any $s\in (0,T)$ again from Lemma \ref{lem-well}. In other words, it reveals that both $\bm u$ and $\bm v(\,\cdot\,;s)$ are defined pointwise a.e. in time in the sense of $L^2(\Om)$. In view of the equivalence among fractional derivatives (see \cite{KRY20}) under this situation, we can rephrase \eqref{eq-IBVP-u4} and \eqref{eq-IBVP-v2} as
\begin{equation}\label{eq-IBVP-u5}
\begin{cases}
(\rd_t^{\bm\al}+\BA+\bm C)\bm u=\bm F & \mbox{in }\Om\times(0,T),\\
\bm u=\bm0 & \mbox{in }(\Om\times\{0\})\cup(\pa\Om\times(0,T))
\end{cases}
\end{equation}
and
\begin{equation}\label{eq-IBVP-v3}
\begin{cases}
(\rd_{s+}^{\bm\al}+\BA+\bm C)\bm v(\,\cdot\,;s)=\bm0 & \mbox{in }\Om\times(s,T),\\
\bm v(\,\cdot\,;s)=D_{0+}^{1-\bm\al}\bm F(s) & \mbox{in }\Om\times\{s\},\\
\bm v(\,\cdot\,;s)=\bm0 & \mbox{on }\pa\Om\times(s,T)
\end{cases}
\end{equation}
respectively, where $\rd_{s+}^{\bm\al}=\diag(\rd_{s+}^{\al_1},\dots,\rd_{s+}^{\al_K})$ and $\rd_{s+}^{\al_k}:=J_{0+}^{1-\al_k}\circ\pa_t$ denotes the Caputo derivative.

Now it suffices to demonstrate by direct calculation that the function $\bm u$ defined by \eqref{eq-rep-uv2} satisfies \eqref{eq-IBVP-u5}. To this end, we differentiate \eqref{eq-rep-uv2} with respect to $t$ and employ the initial condition of \eqref{eq-IBVP-v3} to get
\[
\pa_t\bm u(t)=\bm v(t;t)+\int_0^t\pa_t\bm v(t;s)\,\rd s=D_{0+}^{1-\bm\al}\bm F(t)+\int_0^t\pa_t\bm v(t;s)\,\rd s.
\]
Performing $J_{0+}^{1-\bm\al}=\diag(J_{0+}^{1-\al_1},\dots,J_{0+}^{1-\al_K})$ on both sides above, we use the formula
\[
J_{0+}^{1-\al_k}D_{0+}^{1-\al_k}f=f,\quad f\in W^{1,\infty}(0,T)
\]
to obtain
\begin{align}
\rd_t^{\bm\al}\bm u(t) & =J_{0+}^{1-\bm\al}\left\{D_{0+}^{1-\bm\al}\bm F(t)+\int_0^t\pa_t\bm v(t;s)\,\rd s\right\}\nonumber\\
& =\bm F(t)+J_{0+}^{1-\bm\al}\int_0^t\pa_t\bm v(t;s)\,\rd s.\label{eq-calc}
\end{align}
For each component of the second term on the right-hand side above, we exchange the order of integration to deduce
\begin{align*}
J_{0+}^{1-\al_k}\int_0^t\pa_t v_k(t;s)\,\rd s & =\int_0^t k_{1-\al_k}(t-s)\int_0^s\pa_s v_k(s;\tau)\,\rd\tau\rd s\\
& =\int_0^t\!\!\!\int_\tau^t k_{1-\al_k}(t-s)\pa_s v_k(s;\tau)\,\rd s\rd\tau\\
& =\int_0^t J_{\tau+}^{1-\al_k}\pa_t v_k(t;\tau)\,\rd\tau=\int_0^t\rd_{s+}^{\al_k}v_k(t;s)\,\rd s.
\end{align*}
Plugging the above equality into \eqref{eq-calc}, we utilize the governing equation of $\bm v(t;s)$ to derive
\begin{align*}
\rd_t^{\bm\al}\bm u(t) & =\bm F(t)+\int_0^t\rd_{s+}^{\bm\al}\bm v(t;s)\,\rd s=\bm F(t)-\int_0^t(\BA+\bm C)\bm v(t;s)\,\rd s\\
& =\bm F(t)-(\BA+\bm C)\int_0^t\bm v(t;s)\,\rd s=\bm F(t)-(\BA+\bm C)\bm u(t).
\end{align*}
Hence, $\bm u$ satisfies the governing equation in \eqref{eq-IBVP-u5}. The homogeneous initial and boundary conditions are obvious and the proof completed.
\end{proof}

Now we are ready to proceed to the proof of Theorem \ref{thm-ISP}.

According to the key assumption \eqref{eq-cond-mu}, we know $\rho_k\in W^{1,\infty}(0,T)$ and thus the source term $\bm g\odot\bm\rho\in W^{1,\infty}(0,T;L^2(\Om))$. Then a direct application of Lemma \ref{lem-Duhamel2} implies \eqref{eq-rep-uv2} with $\bm v(t;s)$ solving
\[
\begin{cases}
(\rd_{s+}^{\bm\al}+\BA+\bm C)\bm v(\,\cdot\,;s)=\bm0 & \mbox{in }\Om\times(s,T),\\
\bm v(\,\cdot\,;s)=\bm g\odot D_{0+}^{1-\bm\al}\bm\rho(s) & \mbox{in }\Om\times\{s\},\\
\bm v(\,\cdot\,;s)=\bm0 & \mbox{on }\pa\Om\times(s,T)
\end{cases}
\]
Thanks to \eqref{eq-cond-mu}, it turns out that $D_{0+}^{1-\al_k}\rho_k=(J^{\al_k}\rho_k)'=\mu'$ for all $k=1,\dots,K$, indicating that the initial condition of $\bm v(t;s)$ above becomes
\[
\bm v(\,\cdot\,;s)=\mu'(s)\bm g\quad\mbox{in }\Om\times\{s\}.
\]
As the initial value of all components of $\bm v(t;s)$ share a common multiplier $\mu'(s)$, it follows from the linearity of the problem that $\bm v(t;s)=\mu'(s)\bm y(t-s)$, where $\bm y$ solves the homogeneous problem \eqref{eq-IBVP-v0}. Then for any $k=1,\dots,K$, we have
\begin{equation}\label{eq-rep-mu}
u_k(t)=\int_0^t v_k(t;s)\,\rd s=(\mu'*y_k)(t).
\end{equation}
On the other hand, by the regularity assumption \eqref{eq-reg} and $\be>d/4-1$, Lemma \ref{lem-well} and the Sobolev embedding theorem assert
\begin{align*}
\bm y & \in L^1(0,T;D(\BA^{\be+1}))\subset L^1(0,T;C(\ov\Om)),\\
\bm u & \in L^\infty(0,T;D(\BA^{\be+1-\ve}))\subset L^\infty(0,T;C(\ov\Om)),
\end{align*}
where $\ve>0$ is sufficiently small. Then \eqref{eq-rep-mu} make pointwise sense in space, which allows us to take $\bm x=\bm x_0$ and employ the vanishing of the observation data to obtain
\[
0=u_k(\bm x_0,\,\cdot\,)=\mu'*y_k(\bm x_0,\,\cdot\,)\quad\mbox{in }(0,T).
\]
Recalling the integer $M$ in \eqref{eq-cond-rq}, we preform $J^{M(1-\al_K)}$ on both sides above to deduce
\begin{align*}
0 & =J^{M(1-\al_K)}(\mu'*y_k(\bm x_0,\,\cdot\,))=k_{M(1-\al_K)}*\mu'*y_k(\bm x_0,\,\cdot\,)\\
& =\mu'*(k_{M(1-\al_K)}*y_k(\bm x_0,\,\cdot\,))=\mu'*J^{M(1-\al_K)}y_k(\bm x_0,\,\cdot\,).
\end{align*}
Again from the positivity assumptions \eqref{eq-cond-C}--\eqref{eq-cond-rq}, Theorem \ref{thm-SPP} guarantees that $J^{M(1-\al_K)}\bm y>0$ a.e. in $\Om\times(0,T)$ and in particular $J^{M(1-\al_K)}y_k(\bm x_0,\,\cdot\,)>0$ in $(0,T)$. Then we can immediately conclude $\mu'=0$ in $(0,T)$ by means of the Titchmarsh convolution theorem (see \cite{T26}). Since $\mu(0)=0$ by \eqref{eq-cond-mu}, we have $J^{\al_k}\rho_k=\mu=0$ in $(0,T)$ and eventually $\rho_k=0$ in $(0,T)$ for all $k=1,\dots,K$ by the injectivity of $J^{\al_k}$.


\section{Reconstruction Methodology}\label{sec-Bayes}

We now describe the numerical strategy for recovering the unknown temporal sources and introduce the functional framework for the inverse problem. Let $\cX:=(L^2(0,T))^K$ denote the parameter space containing the unknown vector-valued source $\bm\rho=(\rho_1,\dots,\rho_K)^\T$. 
We equip $\cX$ with the norm
\[
\|\bm\rho\|_\cX:=\sum_{j=1}^K\|\rho_j\|_{L^2(0,T)}.
\]
Let $\mathcal Y:=(L^2(0,T;L^2(\Omega)))^K$ denote the state space. The parameter-to-state mapping is defined by the forward operator
\[
\BG:\cX\longrightarrow\mathcal Y,
\]
which assigns to each admissible source $\bm\rho$ the unique solution $\bm u=\BG(\bm\rho)$ of problem \eqref{eq-IBVP-u0}. The available data consist of measurements of the state at a fixed interior point $\bm x_0\in\Om$. The observation vector is denoted by $\bm y\in \BR^m$, where $m\in\BN$ represents the number of measurements. These observations are assumed to be corrupted by additive Gaussian noise and are modeled as
\begin{equation}\label{eq:data-new}
\bm y=\BG(\bm\rho)+\bm\eta,\quad\bm\eta\sim\cN(\bm0,\cC),
\end{equation}
where $\BG(\bm\rho)$ denotes the parameter-to-observable map and $\cC\in\BR^{m\times m}$ is a symmetric positive-definite covariance matrix describing the noise statistics. Here, the notation $\bm\eta\sim\mathcal N(0,\cC)$ means that $\bm\eta$ is a Gaussian random vector in $\BR^m$ with zero mean and the covariance matrix $\cC$.

Since the governing system \eqref{eq-IBVP-u0} is coupled, the simultaneous recovery of the components $\rho_1,\dots,\rho_K$ leads to an ill-posed inverse problem. To address this ill-posedness, we adopt a Bayesian formulation.


\subsection{Bayesian formulation}

Within the Bayesian framework \cite{S10,DS17}, the unknown parameter $\bm\rho$ is modeled as a random variable defined on $\cX$. Prior information is incorporated through a probability measure $\mu_0$ on $\cX$.

Given the noisy data $\bm y$, the data-misfit functional is defined by
\begin{equation}\label{eq:potential-new}
\Phi(\bm\rho;\bm y):=\f12\|\bm y-\BG(\bm\rho)\|_\cC^2=\f12|\cC^{-1/2}(\bm y-\BG(\bm\rho))|^2.
\end{equation}
This functional quantifies the discrepancy between model predictions and observations, weighted according to the noise statistics. Then Bayes' theorem defines the posterior measure $\mu^{\bm y}$ through
\[
\f{\rd\mu^{\bm y}}{\rd\mu_0}(\bm\rho)=\f1Z\exp(-\Phi(\bm\rho;\bm y)),
\]
where the normalizing constant
\[
Z:=\int_\cX\exp(-\Phi(\bm\rho;\bm y))\,\rd\mu_0(\bm\rho)
\]
ensures that $\mu^{\bm y}$ is a probability measure on $\cX$.

To guarantee that the Bayesian inverse problem is mathematically meaningful, one must verify that the posterior is well-defined and depends continuously on the data. These properties rely on the continuity of the forward operator, which is a direct consequence of the well-posedness result stated in Lemma \ref{lem-well}.

\begin{prop}
Under the same assumptions of Lemma \ref{lem-well}, let $\bm\rho,\delta\bm\rho\in\cX$ and denote $\bm u=\BG(\bm\rho),\wt{\bm u}=\BG(\bm\rho+\delta\bm\rho)$. Then there holds
\[
\|\wt{\bm u}(\bm x_0,\,\cdot\,)-\bm u(x_0,\,\cdot\,)\|_\cX\longrightarrow0\quad\mbox{as }\|\delta\bm\rho\|_\cX\to0.
\]
In other words, the mapping $\BG$ is continuous on $\cX$.
\end{prop}

As a consequence, the posterior distribution is well-defined, which further indicate the following result.
\begin{thm}
Under the same assumptions as above, the posterior measure $\mu^{\bm y}$ defined by \eqref{eq:data-new} is absolutely continuous with respect to the prior measure $\mu_0$, that is,
\[
\mu^{\bm y}\ll\mu_0.
\]
\end{thm}

\begin{proof}
Since $\BG$ is continuous, the potential $\Phi(\bm\rho;\bm y)$ is measurable and finite for $\mu_0$-almost every $\bm\rho$. The result then follows directly from \cite[Proposition 2.1]{TSA17}.
\end{proof}

We now study the stability of the Bayesian posterior with respect to perturbations in the observation data. In analogy with the deterministic stability analysis, our goal is to show that small changes in the data induce only small changes in the posterior distribution. In the Bayesian setting, this amounts to proving continuous dependence of the posterior measure on the data.

To quantify this stability, we invoke the Hellinger distance. For two probability measures $\mu$ and $\mu'$ that are absolutely continuous with respect to a common reference measure $\nu$, the Hellinger distance is defined as (see \cite[Section 4]{S10})
\[
D_{\mathrm{Hell}}(\mu,\mu'):=\left\{\f12\int\left(\sqrt{\f{\rd\mu}{\rd\nu}}-\sqrt{\f{\rd\mu'}{\rd\nu}}\right)^2\,\rd\nu\right\}^{1/2}.
\]

\begin{thm}\label{well-P}
Assume that the components $\rho_j$ $(j=1,\dots,K)$ are independent random variables and the prior measure has the product Gaussian structure
\[
\mu_0=\mu_{0,\rho_1}\otimes\cdots\otimes\mu_{0,\rho_K},\quad\mu_{0,\rho_j}=\cN(m_j,\cC_j).
\]
Then for any $\xi>0,$ there exists a constant $B(\xi)>0$ such that the estimate
\begin{equation}\label{eq-Hell}
D_{\mathrm{Hell}}(\mu^{\bm y_1},\mu^{\bm y_2})\le B(\xi)\,|\bm y_1-\bm y_2|
\end{equation}
holds for all data vectors $\bm y_1,\bm y_2$ satisfying $|\bm y_1|,|\bm y_2|\le\xi$.
\end{thm}

\begin{proof}
By the well-posedness result for the forward problem in Lemma \ref{lem-well} together with the Sobolev embedding theorem, there exists a constant $M>0$ such that
\[
\|\bm u(x_0,\cdot)\|_{\cX} \le M \|\bm \rho\|_{\cX}.
\]
Hence, the forward operator $\BG$ is bounded from $\cX$ into $\cX$. Since each marginal prior $\mu_{0,\rho_j}$ is Gaussian, Fernique's theorem (see \cite[Theorem 6.9]{S10}) implies the exponential integrability of $\|\bm \rho\|_{\cX}^2$ with respect to $\mu_0$. Consequently, the data-misfit functional satisfies the integrability and local Lipschitz conditions required in the general Bayesian well-posedness theory.

Therefore, the Lipschitz stability estimate \eqref{eq-Hell} follows directly from \cite[Proposition 2.2]{TSA17}.
\end{proof}


\subsection{Iterative regularizing ensemble Kalman method}

Deriving the posterior distribution $\mu^{\bm y}$ is merely the first step in the Bayesian inverse problem. The main computational challenge lies in extracting quantitative information from the posterior, such as expectations, variances, or credible intervals.

With the rapid development of computational methods, Markov chain Monte Carlo (MCMC) techniques have become a standard tool for sampling from posterior distributions. However, classical MCMC methods often deteriorate under mesh refinement and become prohibitively slow for nonparametric inverse problems. To address this issue, Cotter et al.\ \cite{CRSW13} introduced the preconditioned Crank-Nicolson (pCN) MCMC method, which is well defined on function spaces and remains robust under discretization refinement. Nevertheless, for complex nonlinear models, a very large number of samples is still required, leading to substantial computational cost.

An alternative strategy is to employ ensemble Kalman-based techniques, which provide approximate Bayesian inference at significantly lower computational expense. The ensemble Kalman filter (EnKF), originally introduced by Evensen \cite{E06}, has become a widely used method in data assimilation; see \cite{LSZ15} for an overview and further applications. Building on this idea, Iglesias et al.\ \cite{ILS13} proposed the iterative ensemble Kalman method (IEKM), which introduces artificial dynamics and is applicable to a broad class of inverse problems. However, the numerical experiments in \cite{ILS13} indicate that early stopping is necessary to stabilize the iteration.

Motivated by this observation, Iglesias \cite{I15} incorporated the discrepancy principle from classical iterative regularization theory into the IEKM framework, leading to the iterative regularizing ensemble Kalman method (IREKM). This modification provides a systematic stabilization mechanism and plays a role similar to localization or covariance inflation, while retaining a rigorous mathematical foundation. Further deterministic analysis of this regularization strategy was carried out in \cite{I16}, where the approach was applied to various PDE-based inverse problems. A parameterized version of IREKM was later investigated in \cite{CIRS18}.

In the present work, we adapt the IREKM proposed in \cite{I15} to solve Problem \ref{prob-ISP}. Recall that the goal of the Bayesian formulation is to infer the input $\bm\rho=(\rho_1,\dots,\rho_K)^\T\in\cX$ from measured data $\bm y$, combining prior uncertainty and observational noise within the coupled model. The prior distribution is assumed to have the product structure
\[
\mu_0(\bm\rho)=\mu_{0,\rho_1}\otimes\cdots\otimes\mu_{0,\rho_K},
\]
where $\mu_{0,\rho_j}=\cN(m_j,\cC_j)$ for $j=1,\dots,K$. Given an observation $\bm y$, the likelihood is proportional to $\exp(-\Phi(\bm\rho;\bm y))$, where $\Phi(\bm\rho;\bm y)$ is the potential defined in \eqref{eq:potential-new}. The uncertainty in the inverse estimate of $\bm\rho$ given the data $\bm y$ is then characterized by the conditional posterior distribution $\mu^{\bm y}(\bm\rho\mid\bm y)$ defined through Bayes' formula
\[
\f{\rd\mu^{\bm y}}{\rd\mu_0}(\bm\rho)\propto\exp(-\Phi(\bm\rho;\bm y)).
\]
While MCMC methods such as pCN provide a principled framework for sampling from $\mu^{y}$, their computational cost is often prohibitive in large-scale nonlinear problems. The IREKM offers a computationally efficient alternative by evolving an ensemble $\{\bm\rho^{(j)}\}_{j=1}^{N_e}$ of $N_e$ particles initialized from the prior distribution. At each iteration, the ensemble is updated through Kalman-type formulas involving empirical covariance and cross-covariance operators computed directly from the ensemble. The discrepancy principle supplies a natural stopping criterion that regularizes the iteration and prevents overfitting to noisy observations.

The resulting ensemble provides approximations of the posterior mean, variance, and other relevant statistical quantities, thereby yielding a computationally efficient, stable, and derivative-free framework for uncertainty quantification in the present inverse problem.

We now introduce the iterative procedure that formalizes this strategy within the ensemble Kalman framework. This methodology is specifically designed to overcome the limitations of conventional gradient-based optimization techniques frequently employed in inverse problems. By avoiding adjoint-based gradient evaluations, the approach substantially reduces analytical complexity and computational expense. This advantage is particularly significant in nonlinear settings, where derivatives may be difficult to derive analytically or costly to approximate numerically. Due to its efficiency, robustness, and ease of implementation, the ensemble Kalman methodology has become widely adopted across numerous scientific and engineering disciplines.

Extensive discussions of the theoretical and numerical foundations of ensemble Kalman inversion and its regularized variants can be found in \cite{E18,I15,ILS13}. In the linear framework, convergence properties are rigorously established by Schillings and Stuart in \cite{SS18}. However, theoretical convergence guarantees for the Iterative Regularizing Ensemble Kalman Method (IREKM) in fully nonlinear regimes remain largely unresolved; see \cite{I15} for further insights.

Beyond the present work, ensemble Kalman-based regularization techniques have been successfully applied to a broad range of inverse problems arising from diverse physical models. Representative examples include the simultaneous identification of parameters in plasticity models for power hardening materials \cite{TB25}, the recovery of time-dependent fractional orders in time-fractional diffusion equations motivated by shale gas applications \cite{B25}, and the Bayesian identification of fractional orders together with spatial components in time-space fractional diffusion equations \cite{ZJY18}. These studies further highlight the versatility and practical effectiveness of the IREKM framework in addressing complex nonlinear inverse problems.

Motivated by these considerations, we now present the detailed implementation of the proposed method, tailored to the nonlinear inverse problem~\ref{prob-ISP}. The complete iterative procedure is summarized in Algorithm~\ref{algo1}.
\begin{algorithm}[htbp]
\begin{enumerate}
\item[] Let $\xi\in(0,1)$ and $\tau>1/\xi$. Generate $N_e$ initial ensemble $\{\bm\rho_0^{(j)}\}_{j=1}^{N_e}=\{(\bm\rho_{1,0}^{(j)},\dots,\bm\rho_{K,0}^{(j)})^\T\}_{j=1}^{N_e}$ from the prior distribution $\mu_0(\bm\rho)$. For  $n=0,1,\dots$,
\item[{\bf 1.}] {\bf Prediction:} Calculate $\bm w_n^{(j)}=\BG(\bm\rho_n^{(j)})$ for $j=1,\dots,N_e$ and compute the ensemble mean
$$
\ov{\bm w}_n=\f1{N_e}\sum_{j=1}^{N_e}\bm w_n^{(j)}.
$$
\item[{\bf 2.}] {\bf Discrepancy principle:} Let $\delta=\|\bm y-\BG(\bm\rho^\dagger)\|_\cC$, where $\bm\rho^\dagger$ denotes the exact solution. If the residual
\begin{equation}\label{stopped}
R_n=\|\bm y-\ov{\bm w}_n\|_\cC\le\delta\tau,   
\end{equation}
then stop and return the estimate
$$
\ov{\bm\rho}_n=\left(\ov{\bm\rho}_{1,n},\dots,\ov{\bm\rho}_{K,n}\right)^\T,\quad\ov{\bm\rho}_{k,n}:=\f1{N_e}\sum_{j=1}^{N_e}\bm\rho_{k,n}^{(j)},\quad k=1,\dots,K.
$$
\item[{\bf 3.}] {\bf Analysis step:} Let 
\begin{align*}
\bm C_n^{ww} & :=\f1{N_e-1}\sum_{j=1}^{N_e}\left(\BG(\bm\rho_n^{(j)})-\ov{\bm w}_n\right)\left(\BG(\bm\rho_n^{(j)})-\ov{\bm w}_n\right)^\T,\\
\bm C_n^{\rho_k w} & :=\f1{N_e-1}\sum_{j=1}^{N_e}\left(\bm\rho_{k,n}^{(j)}-\ov{\bm\rho}_{k,n}\right)\left(\BG(\bm\rho_n^{(j)})-\ov{\bm w}_n\right)^\T,\quad k=1,\dots,K.
\end{align*}
Update each ensemble member as
$$
\bm\rho_{k,n+1}^{(j)}=\bm\rho_{k,n}^{(j)}+\bm C_n^{\rho_k w}(\bm C_n^{ww}+\nu_n\cC)^{-1}\left(\bm y-\bm w_n^{(j)}\right),\quad j=1,\dots,N_e,\ k=1,\dots,K,
$$
where $\nu_n$ is chosen as follows: Let $\nu_0$ be an initial guess, and $\nu_n^{i+1}=2^i \nu_0$. Choose $\nu_n=\nu_n^M$
where $M$ is the first integer such that 
\[
\nu_n^M\left\|\left(\bm C_n^{ww}+\nu_n^M\cC\right)^{-1}(\bm y-\ov{\bm w}_n)\right\|_\cC\ge\xi\left\|\cC^{-1} (\bm y-\ov{\bm w}_n)\right\|.
\]
\end{enumerate}
\caption{Iterative regularizing ensemble Kalman method (IREKM).}\label{algo1}
\end{algorithm}


\section{Numerical Reconstructions}\label{sec-numer}

This section presents a series of one-dimensional numerical experiments to validate Algorithm~\ref{algo1}. The goal is to reconstruct the temporal component $\bm\rho(t)=(\rho_1(t),\dots,\rho_K(t))^\T$ of the source term in the coupled system \eqref{model}. The reconstruction quality is evaluated by comparing the estimated parameters with the exact solutions used to generate synthetic data. We also report the computational performance of the method to demonstrate its efficiency and potential scalability to higher-dimensional problems.

In all numerical experiments, we restrict the spatial dimension $d=1$ and simply take $\Omega=(0,1)$. The final time is fixed as $T=1$. The forward problem is solved using a finite difference scheme on a uniform grid with mesh sizes $\Delta t=\Delta x=0.01$ in both time and space. To evaluate the reconstruction accuracy, we compute the relative errors
\[
e_{\rho_k}^n=\f{\|\ov{\bm\rho}_{k,n}-\bm\rho_k^\dagger\|}{\|\bm\rho_k^\dagger\|}\quad(k=1,\dots,K),\quad E_n=\f{\|\ov{\bm\rho}_n-\bm\rho^\dagger\|}{\|\bm\rho^\dagger\|}.
\]
Here, $\ov{\bm\rho}_{k,n}$ denotes the posterior mean of $\bm\rho_k$ at iteration $n$ of the IREKM algorithm. The synthetic observations are generated from
\[
\bm y=\BG(\bm\rho^\dagger)+\bm\eta,\quad\bm\eta\sim\cN(0,\cC),
\]
where $\cC=\sigma^2\bm I$, $\bm I$ is the identity matrix, and $\sigma>0$ denotes the standard deviation of the noise.

The ensemble size is fixed to $N_e=200$ in all tests. The algorithmic parameters are chosen as $\nu_0=0.01$, $\xi=0.8$ and $\tau\approx1/\xi$; see \cite{I15}. The prior distributions of the temporal functions $\rho_k$ are modeled as Gaussian measures
\[
\mu_{0,\rho_k}=\mu_0(\rho_k)=\cN(m_k,\cC_k),\qquad k=1,\dots,K,
\]
where $m_k$ denotes the prior mean function and $\cC_k$ is the covariance operator defined by $\cC_k=M_k A^{-s_k}$. Here $A^{-s_i}$ is the fractional Laplacian with the homogeneous Dirichlet boundary condition, where the smoothness parameter $s_i>1/2$ controls the regularity of the prior samples: larger values produce smoother realizations. The scaling constant $M_i>0$ determines the overall variance and thus the amplitude of admissible fluctuations around the mean. In all experiments, we take $M_i=100$ and $s_k=2$ for $k=1,\dots,K$.

The prior mean functions $m_k$ are chosen as piecewise constant profiles whose values match those of the exact solutions $\bm\rho_k^\dagger$ at the endpoints of the time interval. In addition, the true solution is assumed to be a realization from the same distribution as the initial ensemble. This ensures that the ensemble members share the same regularity and temporal structure as the truth, a requirement related to the invariance subspace property of ensemble Kalman methods; see \cite{I15,I16,ILS13}. As a result, the initial ensemble spans a subspace that is consistent with the true parameter, providing a stable starting point for the iterative reconstruction.

For simplicity, the elliptic operators $\cA_k$ in \eqref{eq-IBVP-u0} are all chosen as $-\pa_x^2$, and the coupling matrix $C$ is defined by
\[
c_{ij}=\begin{cases}
K+1, & i=j,\\
-1, & i\ne j,
\end{cases}
\]
which is obviously positive-definite.

All simulations are performed in MATLAB R2017a. Time integrations are computed using the trapezoidal rule, which provides a good compromise between accuracy and computational cost.


\subsection{Validation of the non-degeneracy condition}\label{sub5.1}

In this subsection, we restrict ourselves to a coupled system of two components, i.e., $K=2$. The purpose is to examine the influence of the key non-degeneracy condition $\det\bm G(x_0)\ne0$ in Theorem \ref{thm-Lip} on the reconstruction of the unknown source terms. The fractional orders are chosen as $\al_1=0.8,\al_2=0.3$ and the prior mean functions are taken to be $m_1(t)=1,m_2(t)=0$. The exact temporal sources are chosen as
\[
\rho_1(t)=\cos(2\pi t),\quad\rho_2(t)=\sin(-\pi t),
\]
and the measurement data are given by
\[
\bm y(t)=(u_1(x_0,t),u_2(x_0,t))^\T,\quad\mbox{with }x_0=0.5.
\]
To demonstrate the necessity of the determinant condition, we compare degenerate and non-degenerate configurations of the spatial component $\bm G$ by the next two examples.

\begin{exa}
We first test the degenerate configuration, that is, $\det\bm G(x_0)=0$. We consider two choices of $(g_1,g_2)$ for which the determinant condition fails:
\begin{itemize}
\item{\bf Case A:} $g_1(x)=\cos(\pi x)$, $g_2(x)=2$.
\item{\bf Case B:} $g_1(x)=2$, $g_2(x)=x-0.5$.
\end{itemize}
In both cases, the matrix $\bm G(x_0)$ becomes singular at $x_0=0.5$, so that the identifiability condition is violated.

Figures \ref{fig1}--\ref{fig2} show the reconstruction results under the noise level $\sigma=0.0001$ for Case A and Case B, respectively. The reconstruction results are displayed in Figure \ref{fig2}. Similar instability and large reconstruction errors are observed. Therefore, both degenerate cases clearly demonstrate the instability of the reconstruction when the determinant condition is not satisfied. More precisely, in Case A we have $g_1(0.5)=0$ and $g_2(0.5)\ne0$, and it is readily seen from Figure \ref{fig1} that the reconstruction of $\rho_2(t)$ performs much better than that of $\rho_1(t)$. While in Case B, one can see an opposite situation from Figure \ref{fig2}. Therefore, this example demonstrates the absence of full identifiability when $\det\bm G(x_0)=0$, that is, at least one component of the unknown becomes unrecoverable. This degeneracy leads to partial identifiability and results in instability of the associated inverse problem.
\begin{figure}[htbp]\centering
\begin{minipage}[b]{0.45\textwidth}\centering
\includegraphics[width=\linewidth]{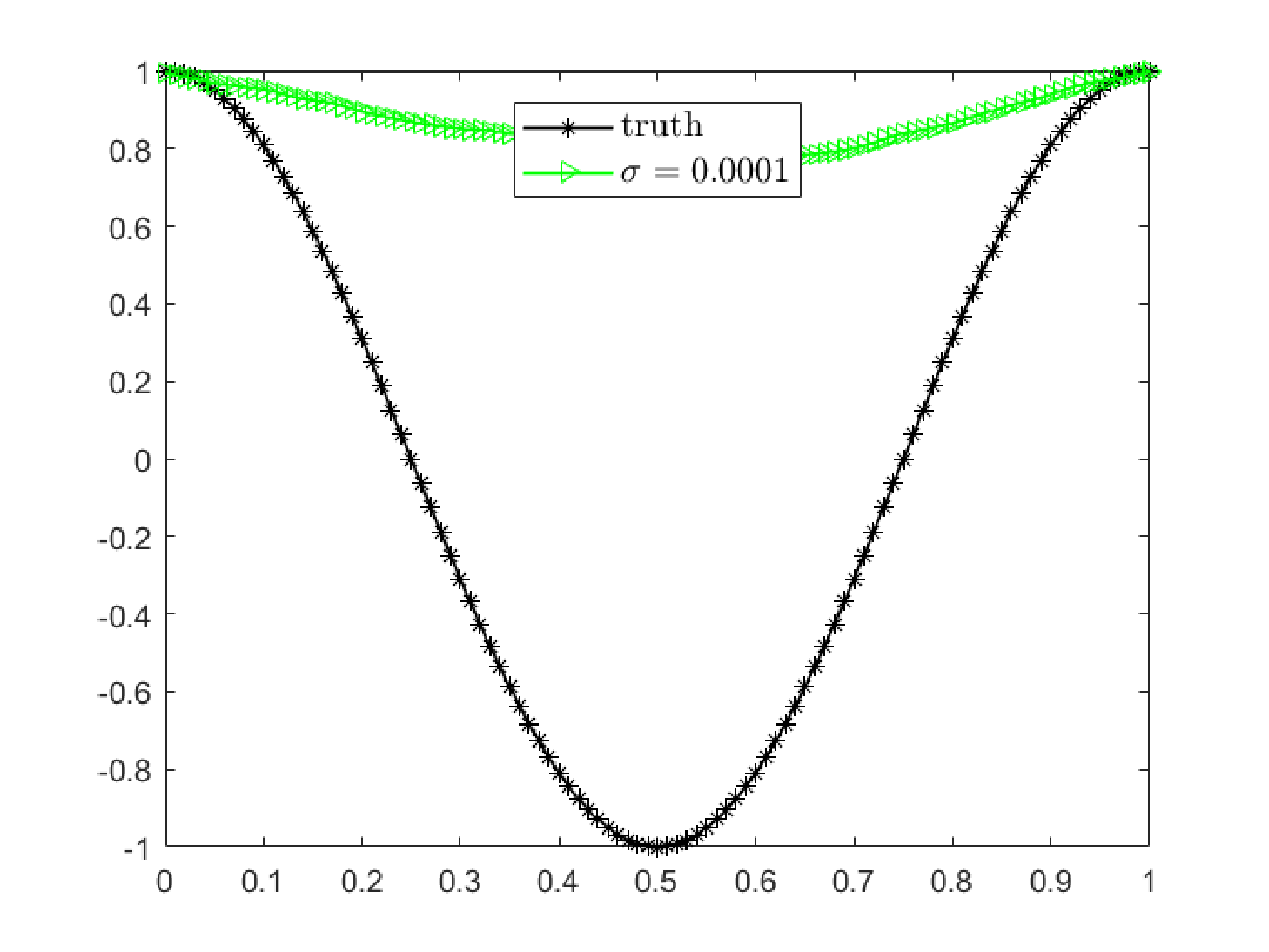}
\subcaption{Reconstruction of $\rho_1(t)$}\label{fig:rho1}
\end{minipage}
\hfill
\begin{minipage}[b]{0.45\textwidth}\centering
\includegraphics[width=\linewidth]{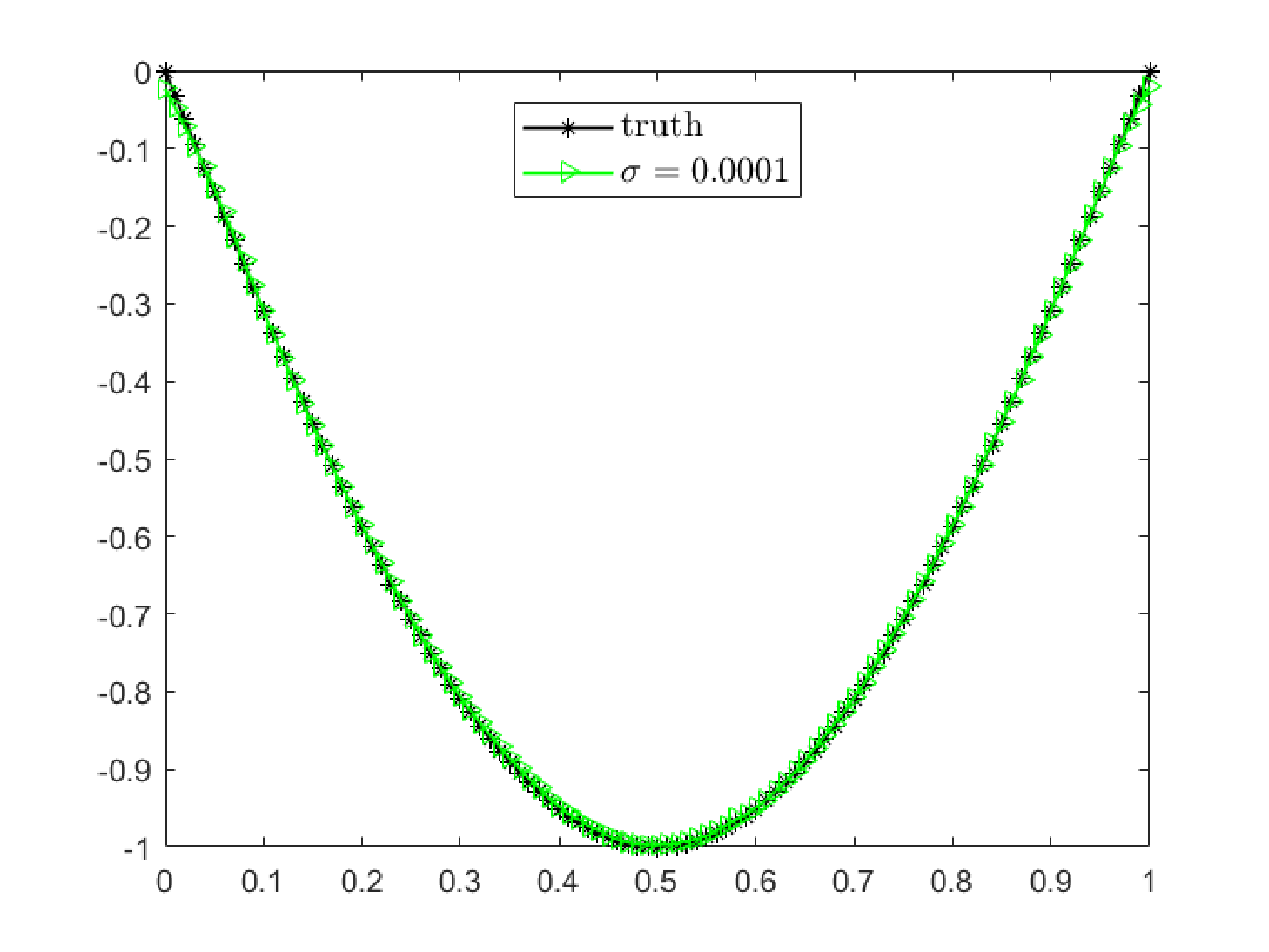}
\subcaption{Reconstruction of $\rho_2(t)$}\label{fig:rho2}
\end{minipage}
\caption{Reconstruction results for Case A ($\det\bm G(x_0)=0$) with noise level $\sigma=0.0001$}\label{fig1}
\end{figure}

\begin{figure}[htbp]\centering
\begin{minipage}[b]{0.45\textwidth}\centering
\includegraphics[width=\linewidth]{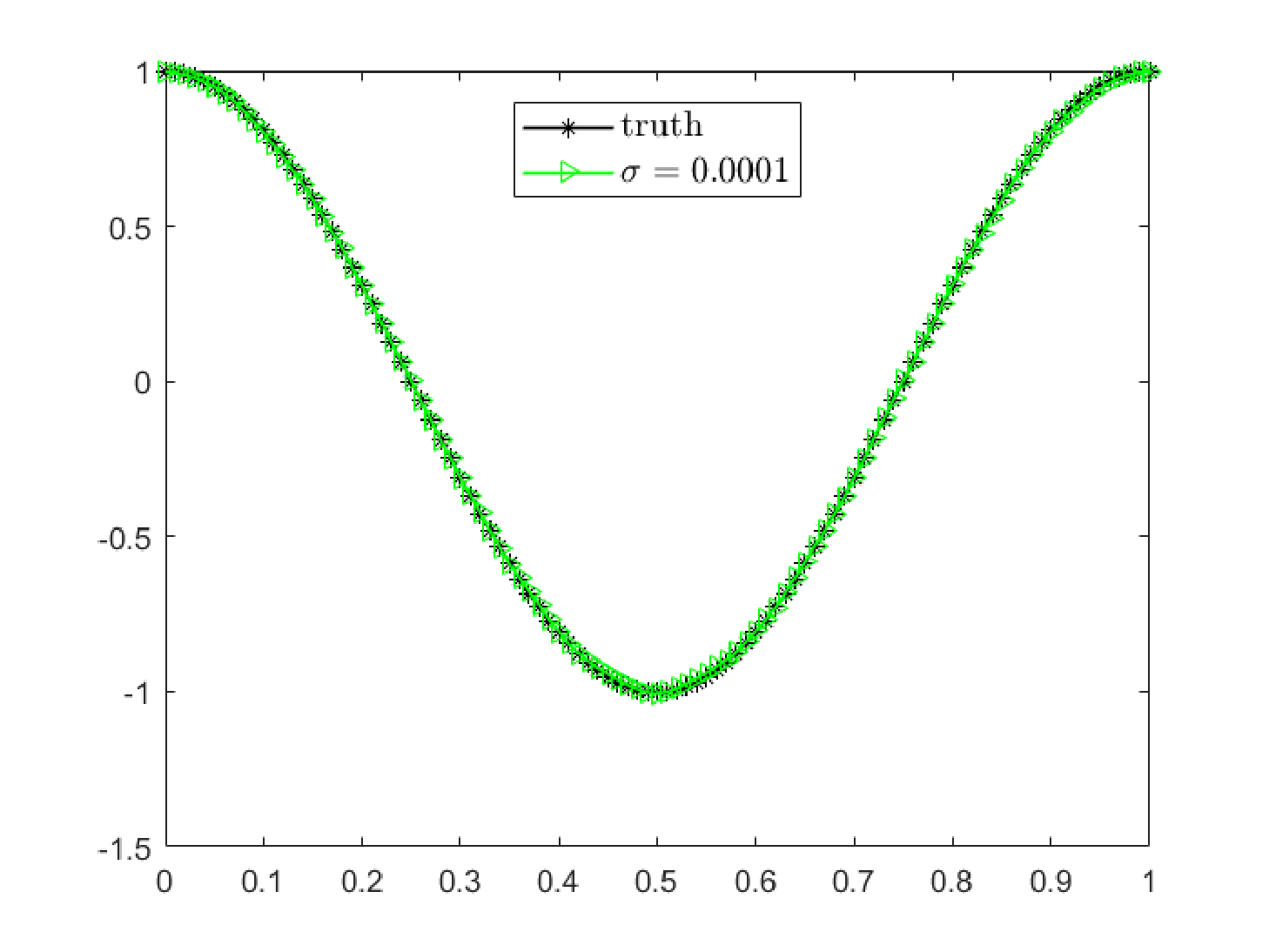}
\subcaption{Reconstruction of $\rho_1(t)$}\label{fig:rho1_caseB}
\end{minipage}
\hfill
\begin{minipage}[b]{0.45\textwidth}\centering
\includegraphics[width=\linewidth]{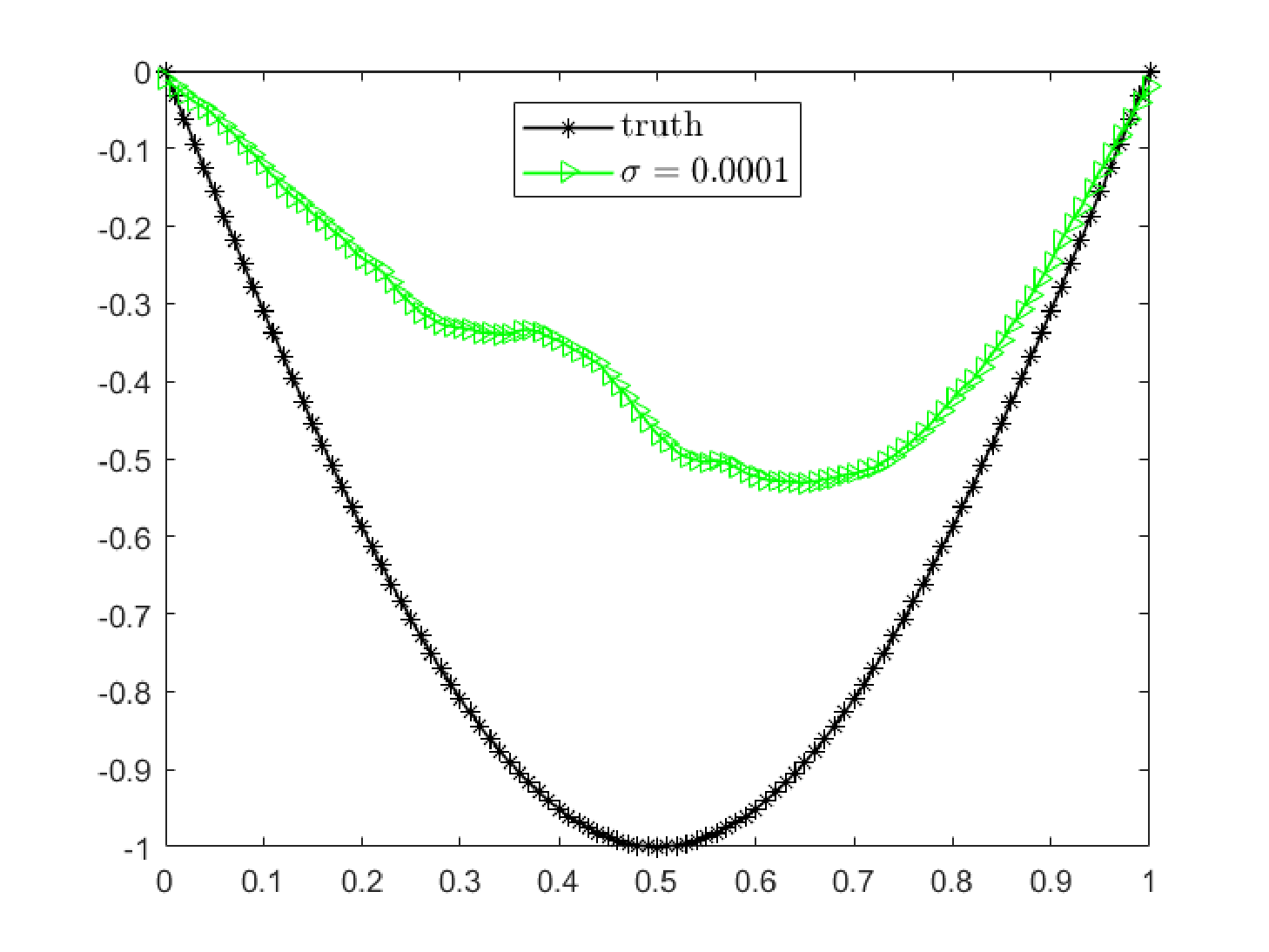}
\subcaption{Reconstruction of $\rho_2(t)$}\label{fig:rho2_caseB}
\end{minipage}
\caption{Reconstruction results for Case B ($\det\bm G(x_0)=0$) with noise level $\sigma=0.0001$}\label{fig2}
\end{figure}
\end{exa}

\begin{exa}\label{ex5.2}
Next, we test the identifiable configuration with $\det\bm G(x_0)\ne0$. We now choose
\[
g_1(x)=x+1,\quad g_2(x)=2,
\]
which ensures that $\det\bm G(x_0)\ne0$ at $x_0=0.5$. Figure \ref{fig3} shows the evolution of the absolute errors between the ensemble means and the true unknowns for different noise levels. The red curve corresponds to the initial ensemble mean, while the green curve represents the final ensemble mean.
\begin{figure}[htbp]\centering
\begin{minipage}[b]{0.31\textwidth}\centering
\includegraphics[width=\linewidth]{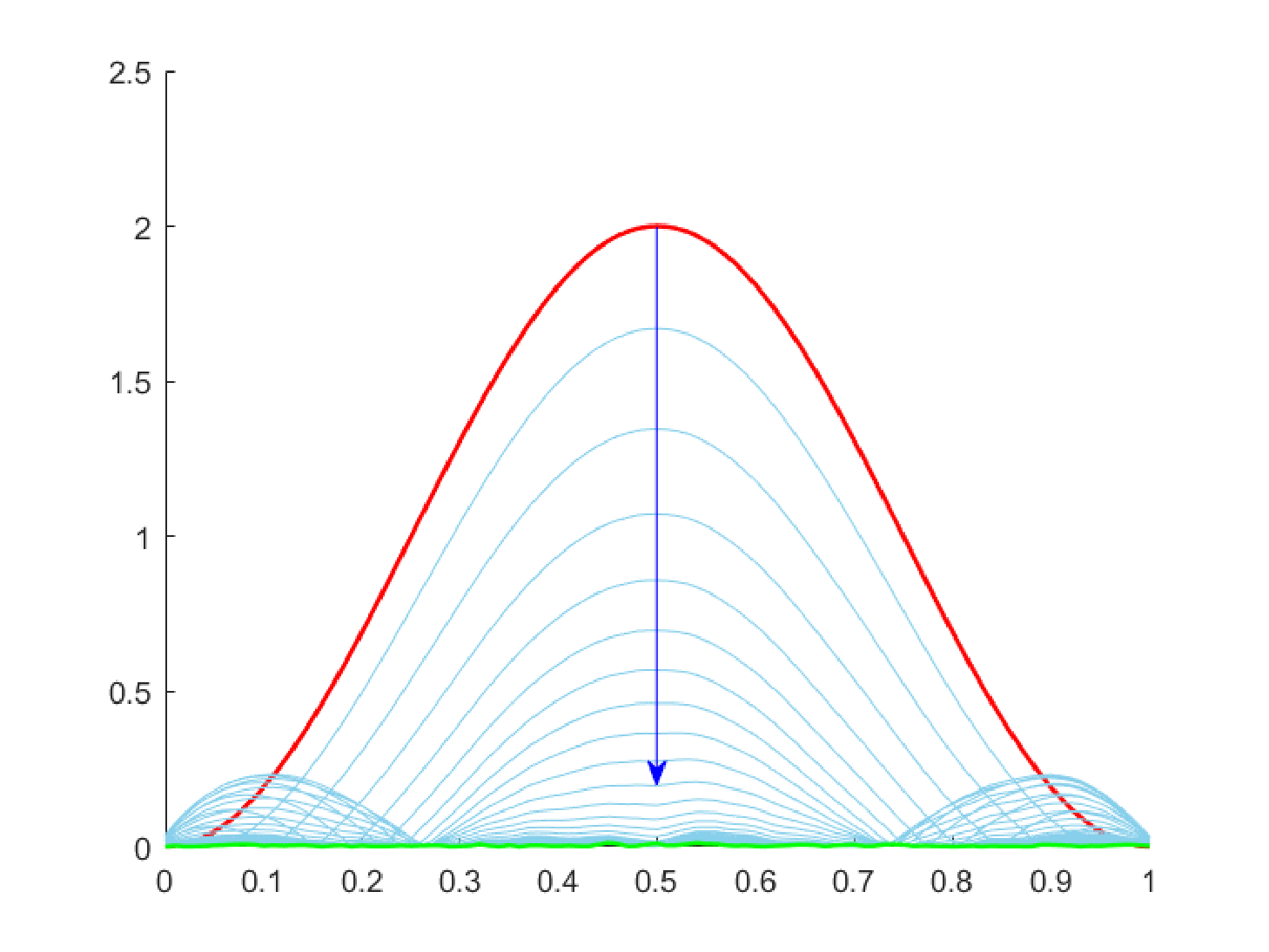}
\subcaption{$\rho_1(t)$ with $\sigma=0.0001$}\label{fig:err_g1_eps1}
\end{minipage}
\hfill
\begin{minipage}[b]{0.31\textwidth}\centering
\includegraphics[width=\linewidth]{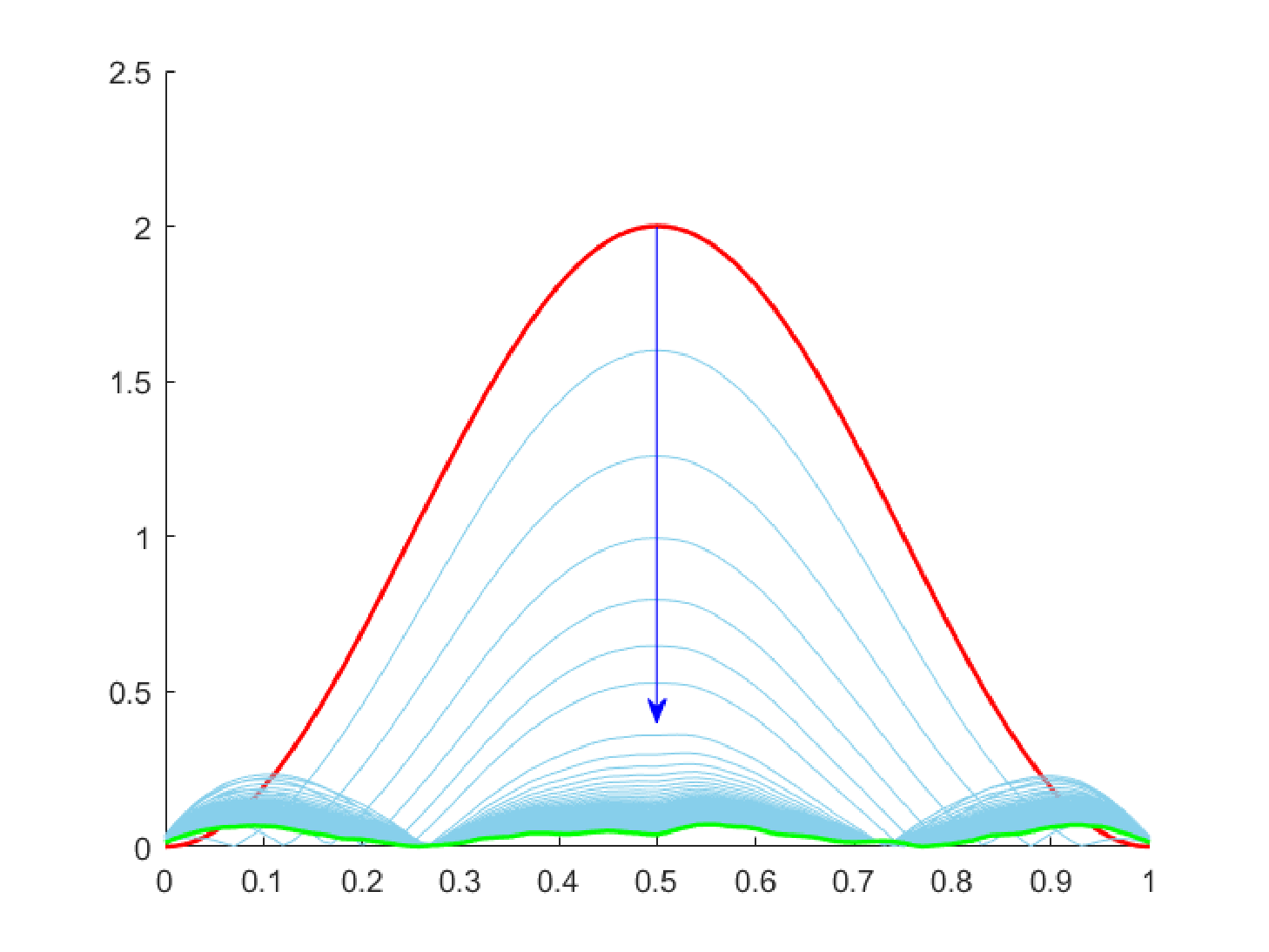}
\subcaption{$\rho_1(t)$ with $\sigma=0.001$}\label{fig:err_g1_eps2}
\end{minipage}
\hfill
\begin{minipage}[b]{0.31\textwidth}\centering
\includegraphics[width=\linewidth]{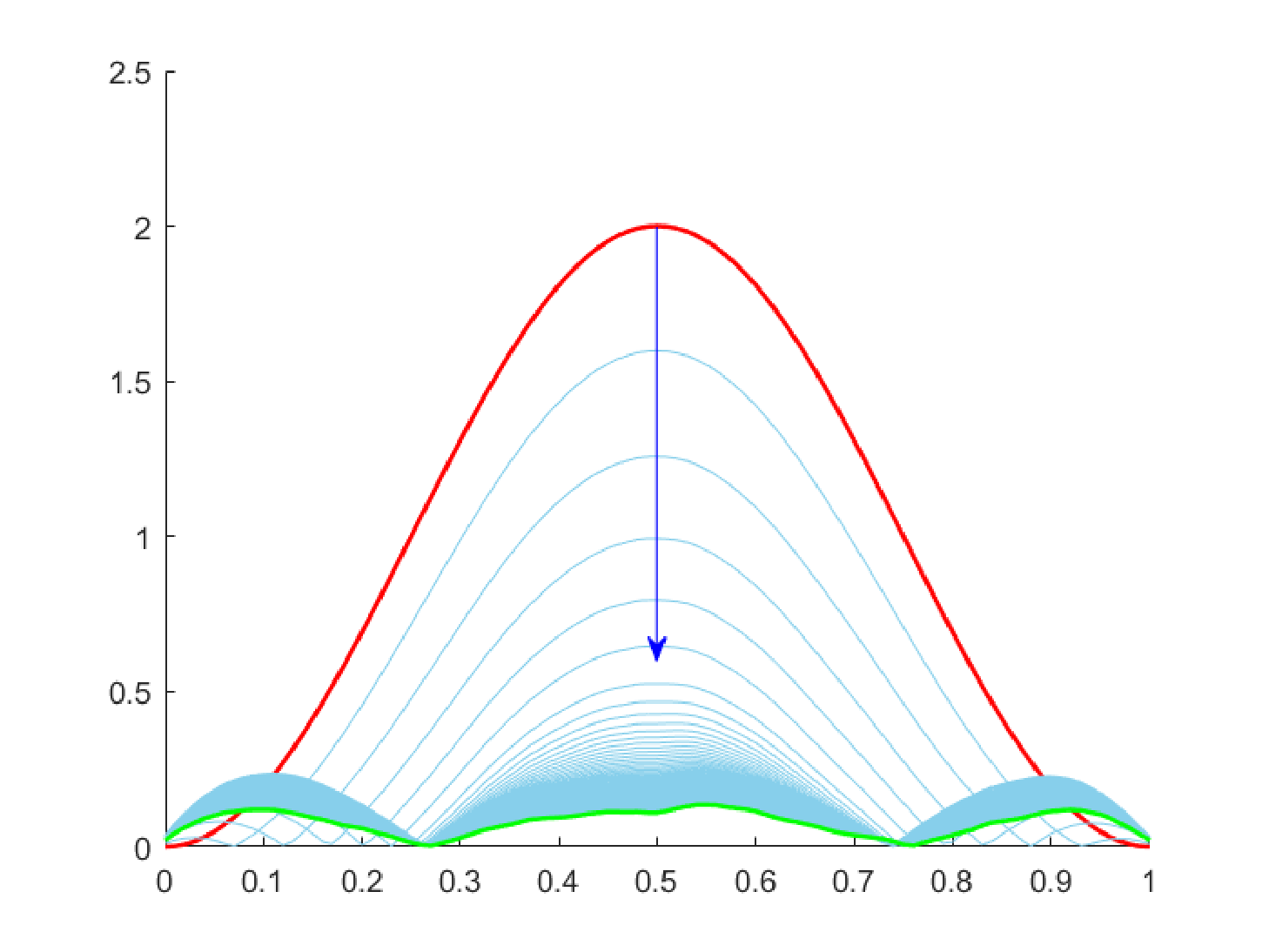}
\subcaption{$\rho_1(t)$ with $\sigma=0.01$}\label{fig:err_g1_eps3}
\end{minipage}
\vspace{0.5em}
  
\begin{minipage}[b]{0.31\textwidth}\centering
\includegraphics[width=\linewidth]{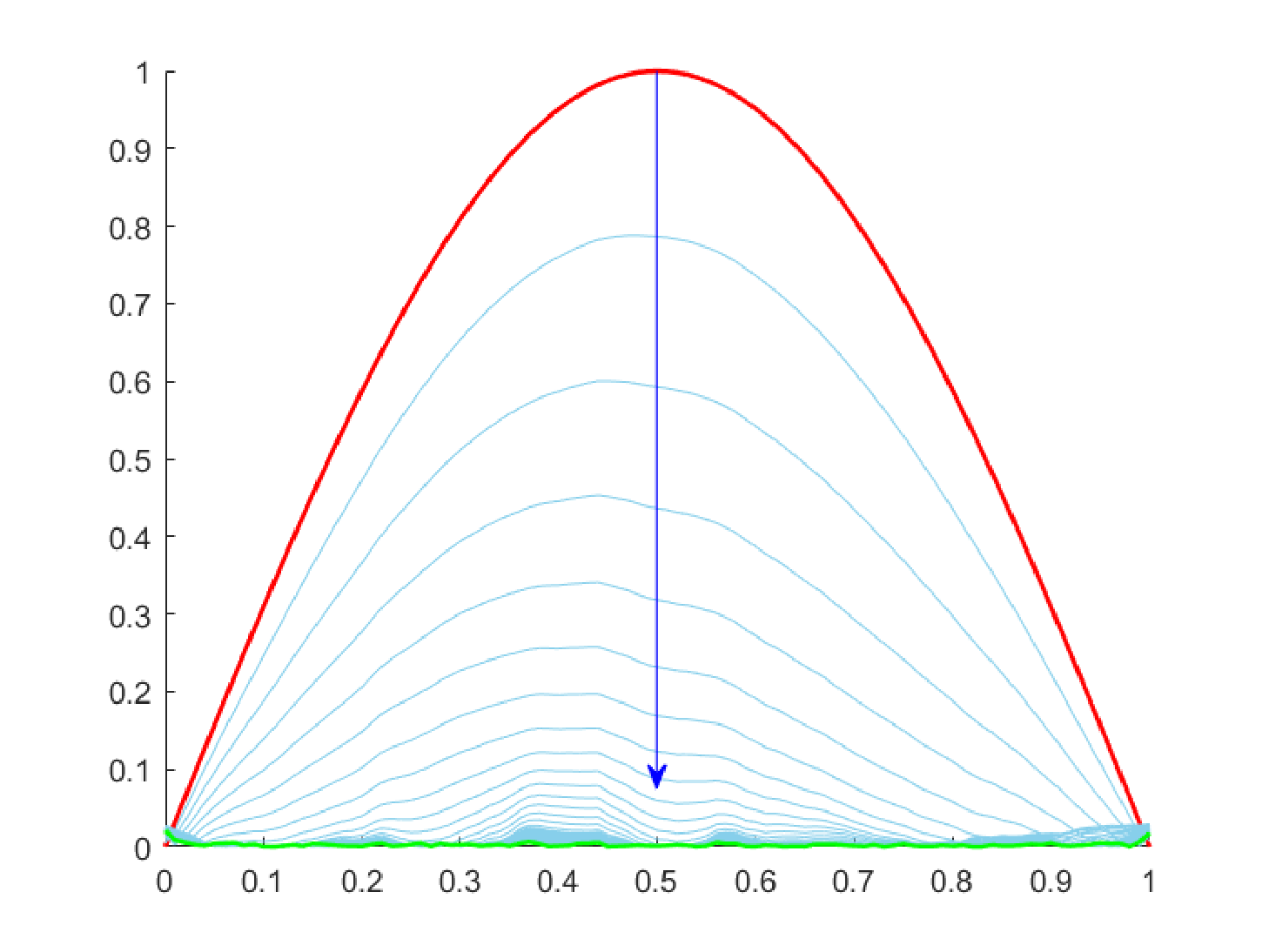}
\subcaption{$\rho_2(t)$ with $\sigma=0.0001$}\label{fig:err_g2_eps1}
\end{minipage}
\hfill
\begin{minipage}[b]{0.31\textwidth}\centering
\includegraphics[width=\linewidth]{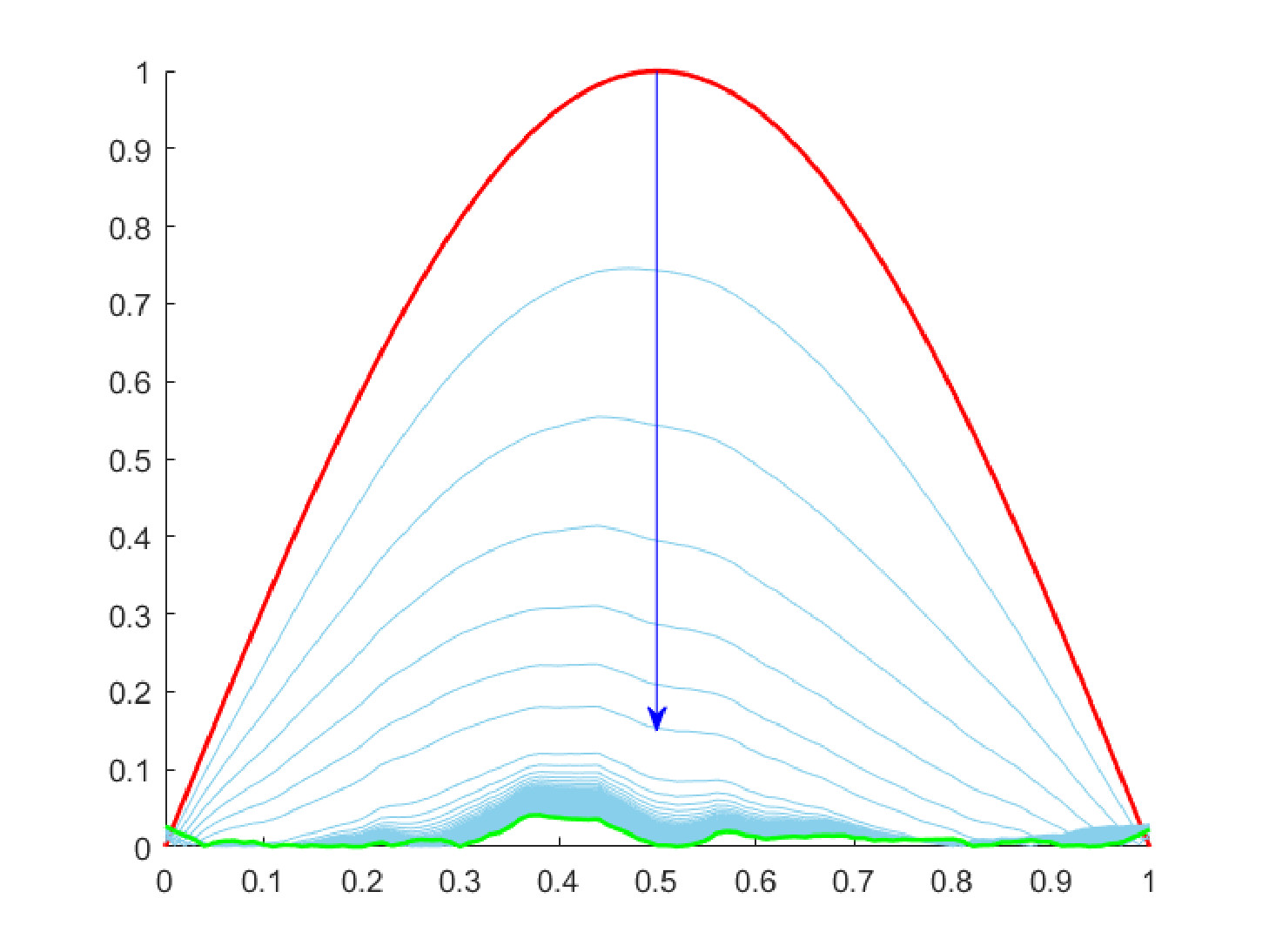}
\subcaption{$\rho_2(t)$ with $\sigma=0.001$}\label{fig:err_g2_eps2}
\end{minipage}
\hfill
\begin{minipage}[b]{0.31\textwidth}\centering
\includegraphics[width=\linewidth]{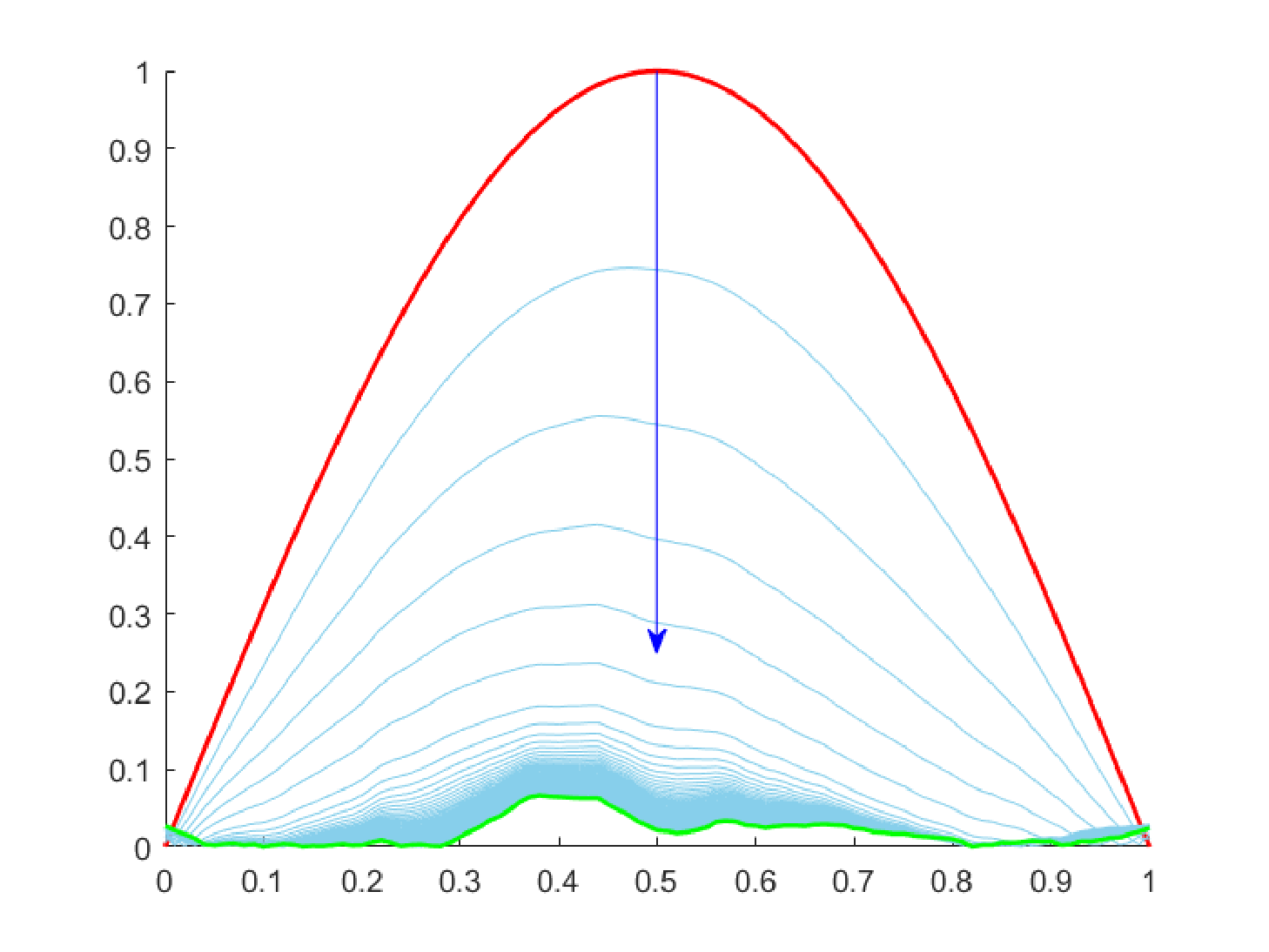}
\subcaption{$\rho_2(t)$ with $\sigma=0.01$}\label{fig:err_g2_eps3}
\end{minipage}
\caption{Absolute reconstruction errors for different noise levels with $\det\bm G(x_0)\ne0$. Top row corresponds to $\rho_1(t)$ and the bottom row to $\rho_2(t)$. From left to right, the noise levels are $\sigma=0.0001$, $0.001$ and $0.01$, respectively}\label{fig3}
\end{figure}

Figure \ref{fig4} compares the reconstructed sources with the exact solutions for $\sigma=0.0001$, $0.001$, and $0.01$. As expected, the reconstruction of both unknowns looks satisfactory, illustrating the essential role played by the non-degeneracy condition $\det\bm G(x_0)\ne0$. Therefore, this example confirms that the proposed IREKM algorithm accurately reconstructs the temporal source terms in this case. Moreover, the reconstruction quality improves as the noise level decreases, demonstrating the stability and robustness of the method under the identifiability condition.
\begin{figure}[htbp]\centering
\begin{minipage}[b]{0.45\textwidth}\centering
\includegraphics[width=\linewidth]{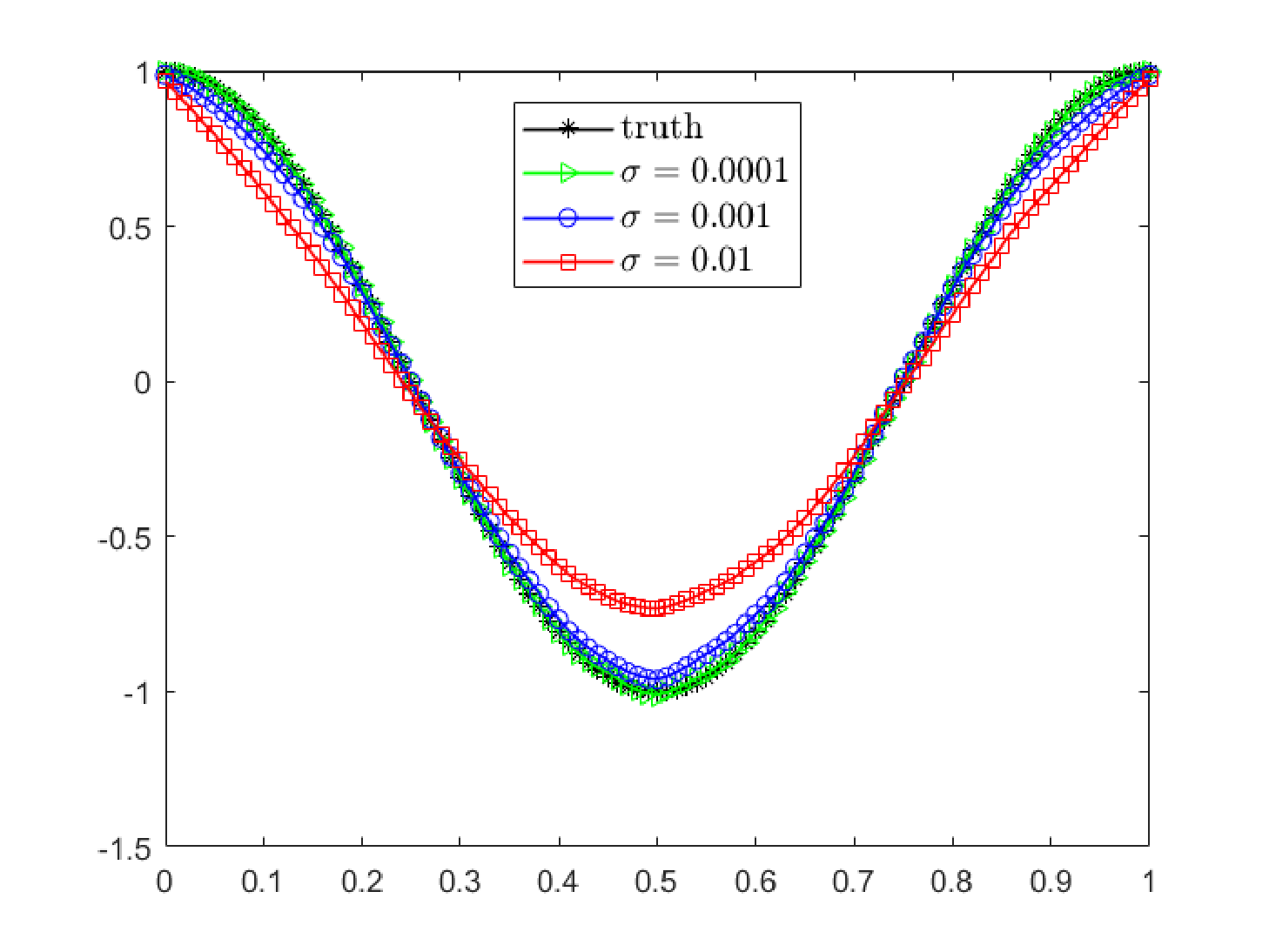}
\subcaption{Reconstruction of $\rho_1(t)$}\label{fig:recon_g1_nonzero}
\end{minipage}
\hfill
\begin{minipage}[b]{0.45\textwidth}\centering
\includegraphics[width=\linewidth]{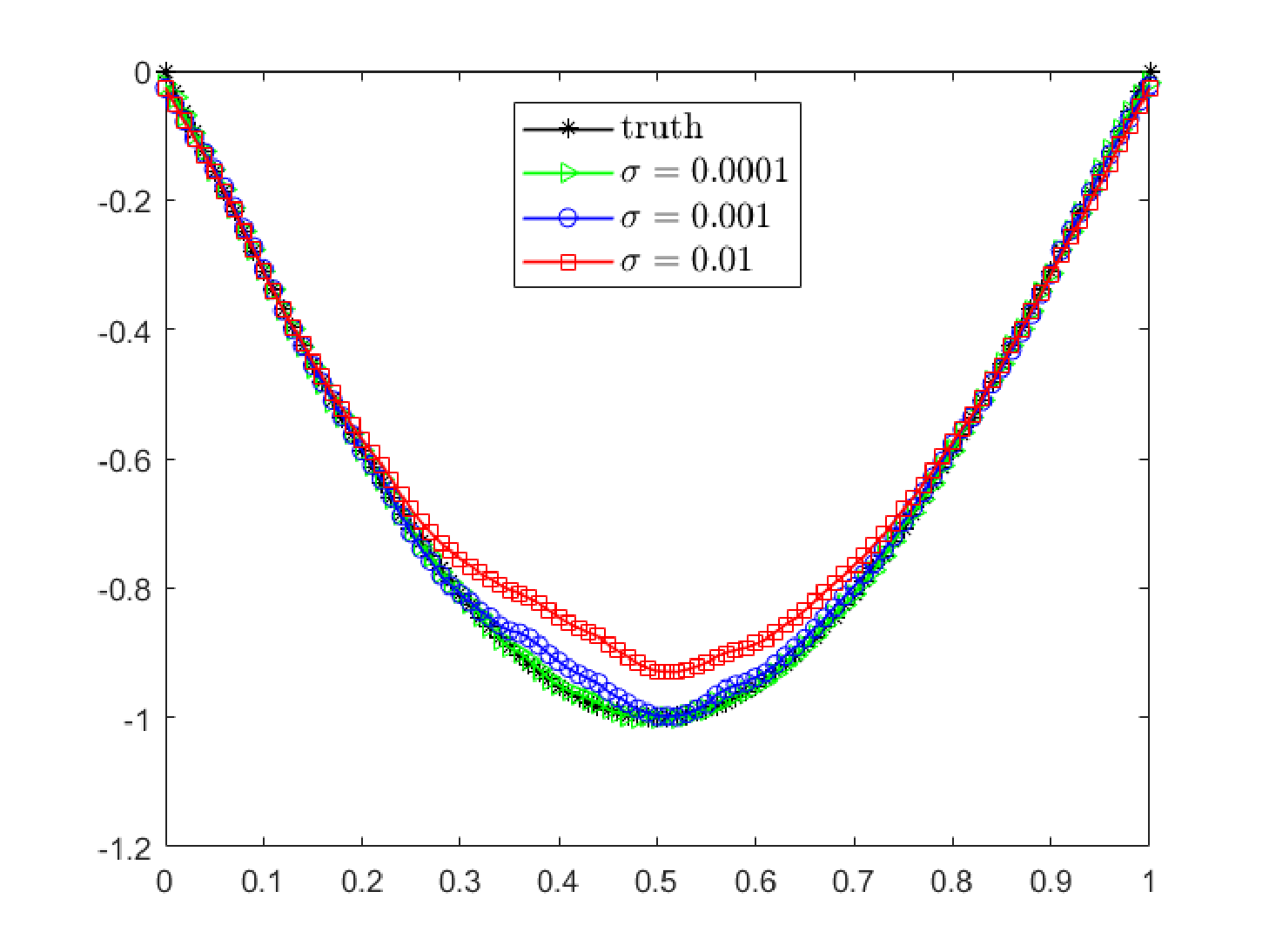}
\subcaption{Reconstruction of $\rho_2(t)$}\label{fig:recon_g2_nonzero}
\end{minipage}
\caption{Reconstruction results with $\det\bm G(x_0)\ne0$ under different noise levels}\label{fig4}
\end{figure}
\end{exa}

Based on the observations from the above examples, in the subsequent numerical tests we comply with the non-degeneracy condition $\det\bm G(x_0)\ne0$, ensuring full identifiability of all components of $\bm\rho(t)$.

To further assess the algorithm, we analyze convergence and stability for this identifiable case. Figure~\ref{fig8} shows the logarithmic relative errors and residual norms for $\rho_1(t)$ and $\rho_2(t)$ under various noise levels.
\begin{figure}[htbp]\centering
\begin{minipage}[b]{0.45\textwidth}\centering
\includegraphics[width=\linewidth]{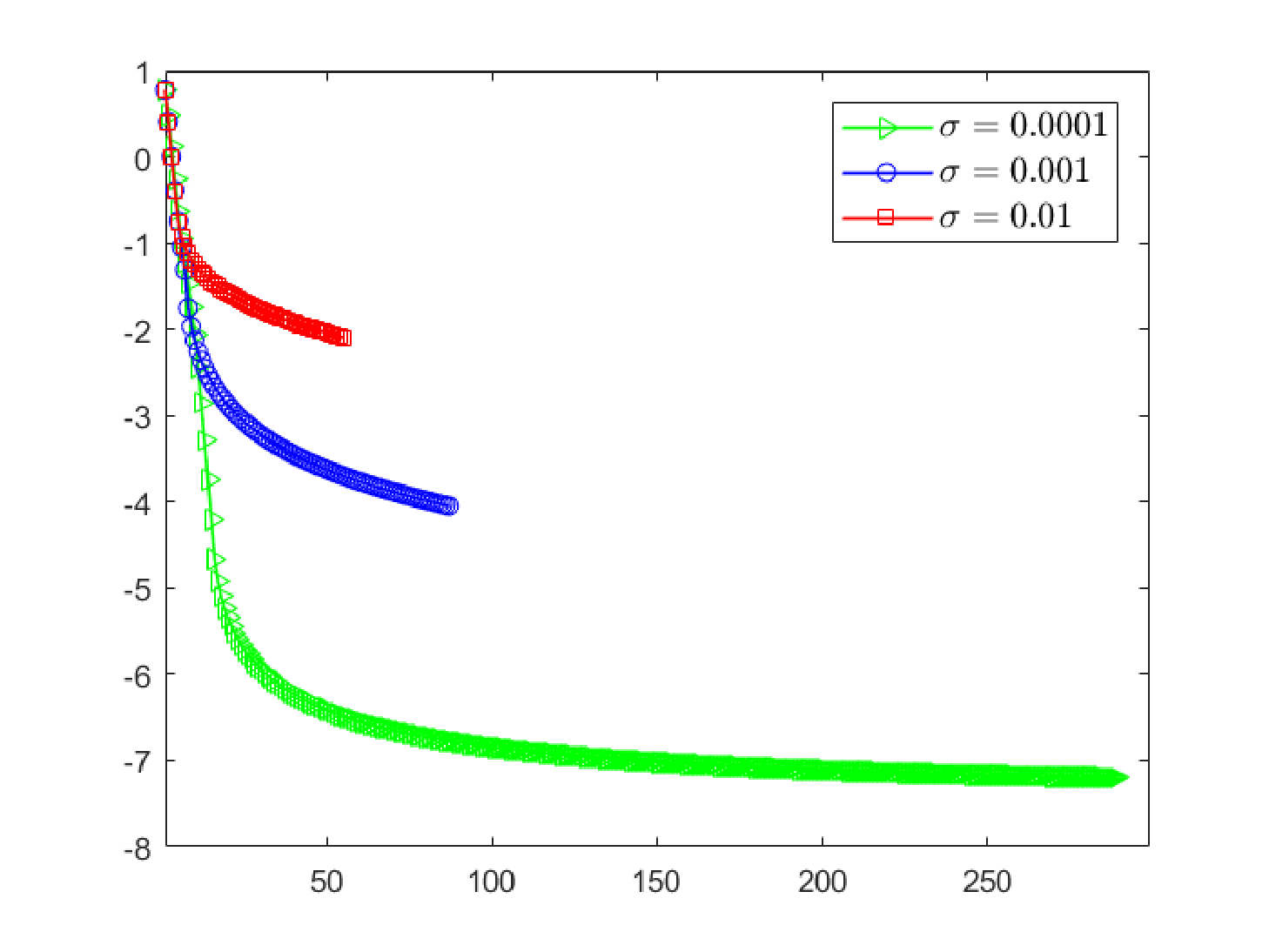}
\subcaption{$\log_2(e_{\rho_1}^n)$}\label{fig8a}
\end{minipage}
\hfill
\begin{minipage}[b]{0.45\textwidth}\centering
\includegraphics[width=\linewidth]{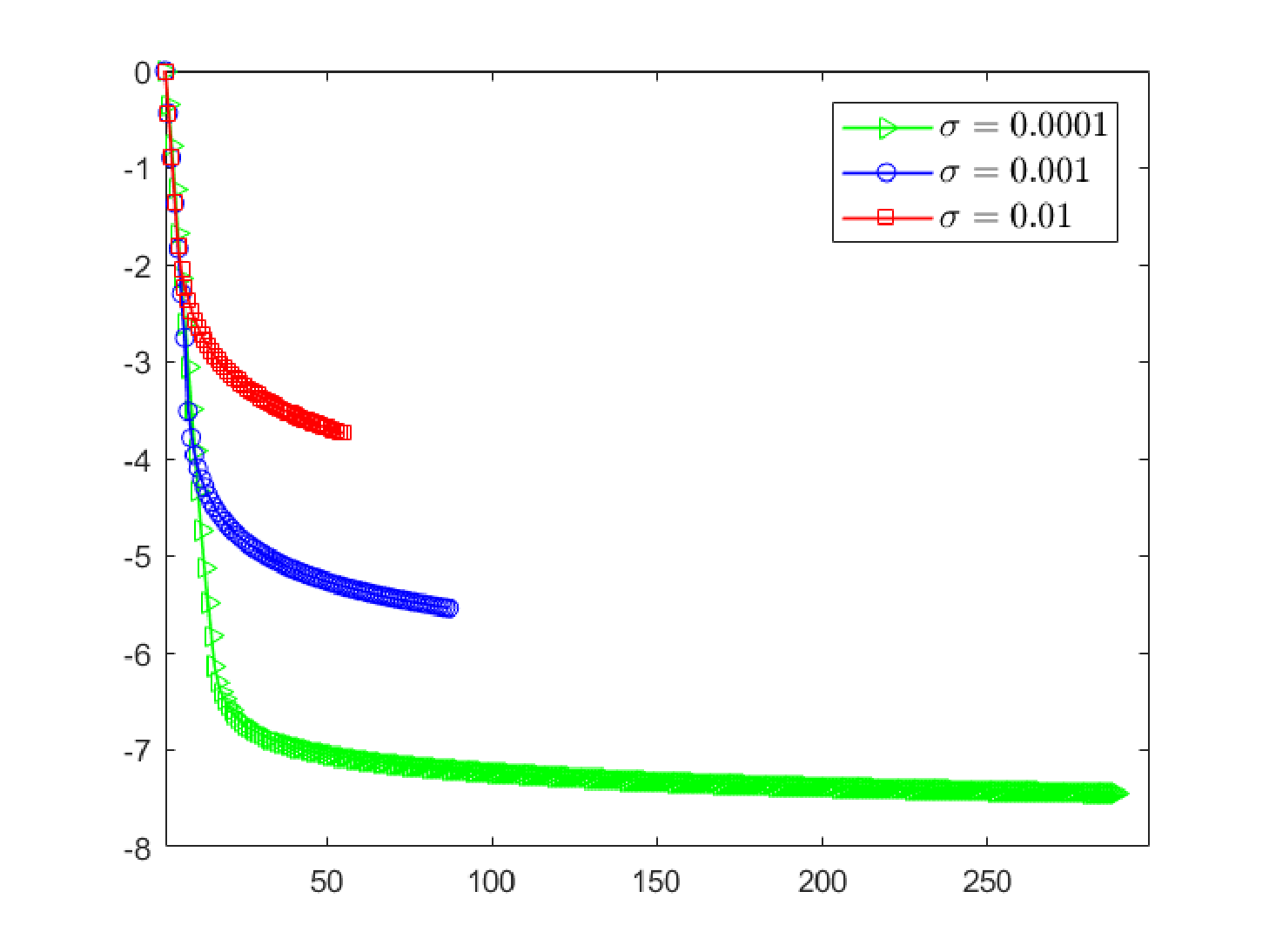}
\subcaption{$\log_2(e_{\rho_2}^n)$}\label{fig8b}
\end{minipage}
\vspace{0.5em}

\begin{minipage}[b]{0.45\textwidth}\centering
\includegraphics[width=\linewidth]{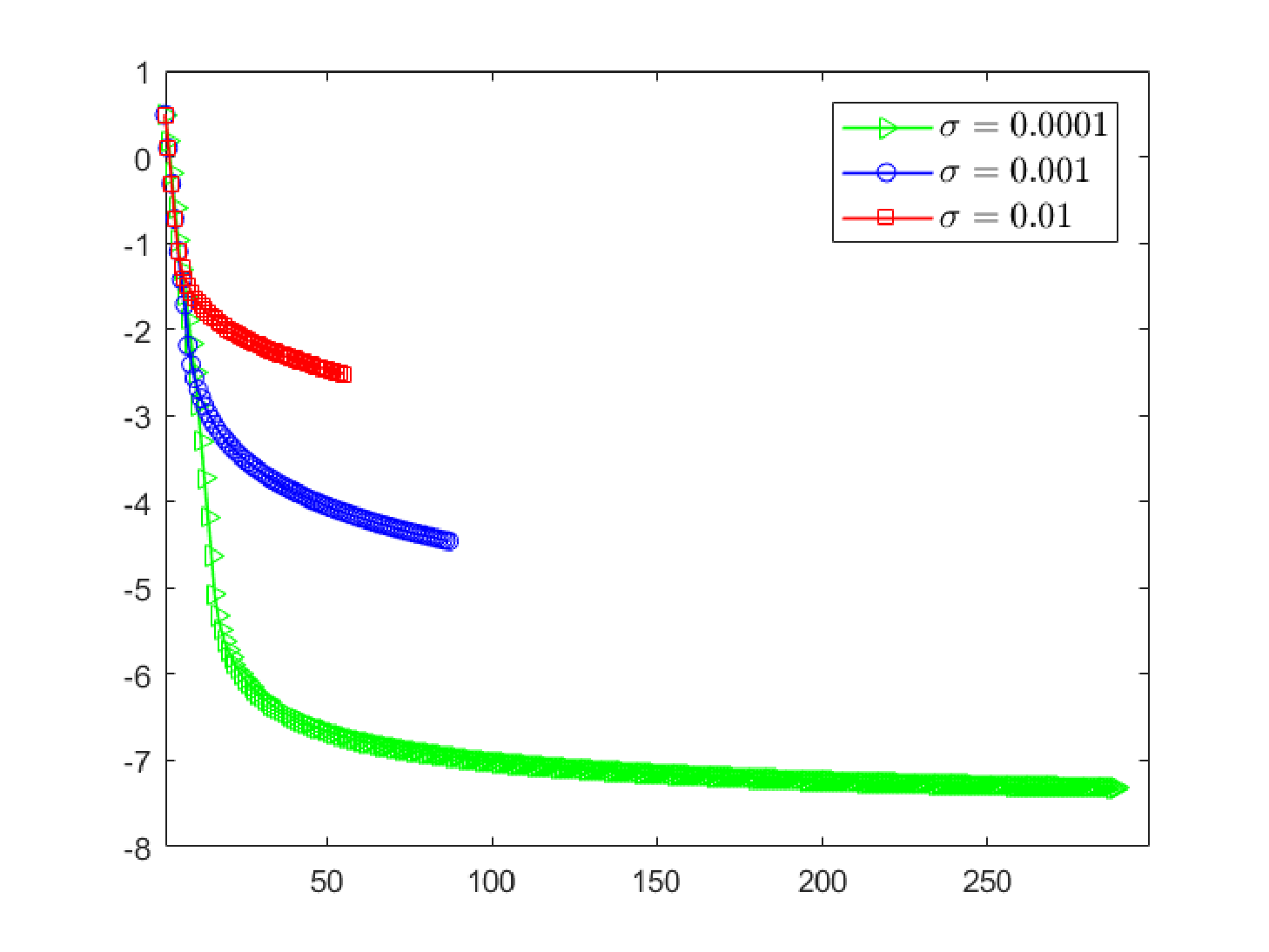}
\subcaption{$\log_2(E_n)$}\label{fig8c}
\end{minipage}
\hfill
\begin{minipage}[b]{0.45\textwidth}\centering
\includegraphics[width=\linewidth]{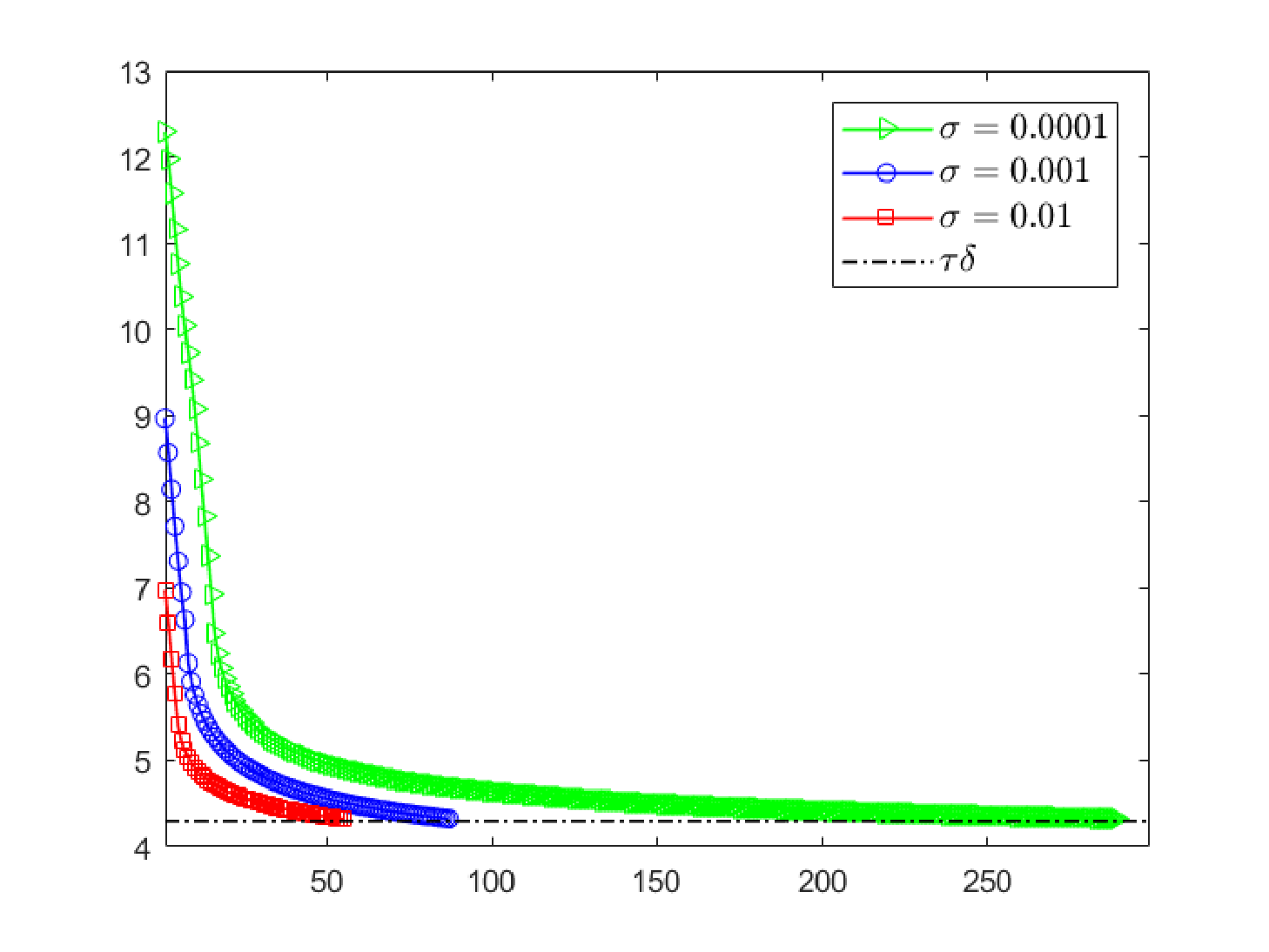}
\subcaption{$\log_2(R_n)$}\label{fig8d}
\end{minipage}
\caption{Logarithmic relative errors and residuals for varying standard deviations: $\sigma=0.0001$, $\sigma=0.001$, and $\sigma=0.01$ associated to $\rho_1(t)$ and $\rho_2(t)$ reconstructed in Example \ref{ex5.2}}\label{fig8}
\end{figure}

The results indicate monotone decay of the residuals and decreasing errors with reduced noise, confirming that the IREKM algorithm converges reliably and provides stable reconstructions under the identifiability condition.


\subsection{Measurement influence under standard identifiability}\label{sub5.2}

The aim of this subsection is to examine how the type of available measurement data influences the reconstruction of the temporal source terms. The fractional orders are fixed as $\al_1=0.7,\al_2=0.4$ and the prior mean functions are still taken as $m_1(t)=1,m_2(t)=0$. The spatial source components and the observation point are chosen as
\[
g_1(x)=\e^x,\quad g_2(x)=2+x,\quad x_0=0.7
\]
respectively, which obviously ensures $\det\bm G(x_0)\ne0$ and thus the full identifiability of the temporal components. The exact temporal sources are chosen
\[
\rho_1(t)=\begin{cases}
1.5, & 0.4\le t<0.6,\\
1, & \mbox{otherwise},
\end{cases}\quad\rho_2(t)=(t-t^2)\,\e^{-t^2}.
\]
Although Theorem \ref{thm-Lip} requires the observation of both components for the theoretical stability, here we consider the following three observation settings in order to assess the influence of the measurement configuration:
\begin{itemize}
\item{\bf Case 1:} Single measurement of $u_1$, i.e., $y(t)=u_1(x_0,t)$.
\item{\bf Case 2:} Single measurement of $u_2$, i.e., $y(t)=u_2(x_0,t)$.
\item{\bf Case 3:} Full measurement, i.e., $\bm y(t)=(u_1(x_0,t),u_2(x_0,t))^\T$.
\end{itemize}

Figure~\ref{fig5} presents the reconstruction results under noise level $\sigma=0.0001$ for the three measurement configurations, from which we clearly witness the dominating influence of the amount and component of observation to the numerical performance of the reconstruction. In Case 1, only $u_1(x_0,t)$ is observed and the reconstruction successfully identifies $\rho_1(t)$, while $\rho_2(t)$ cannot be reliably recovered. The situation in Case 2 is highly analogous. In contrast, Case 3 uses the full measurement of $\bm u(x_0,t)$ and hence achieves accurate and stable reconstruction of both $\rho_1(t)$ and $\rho_2(t)$.
\begin{figure}[htbp]\centering
\begin{minipage}[b]{0.31\textwidth}\centering
\includegraphics[width=\linewidth]{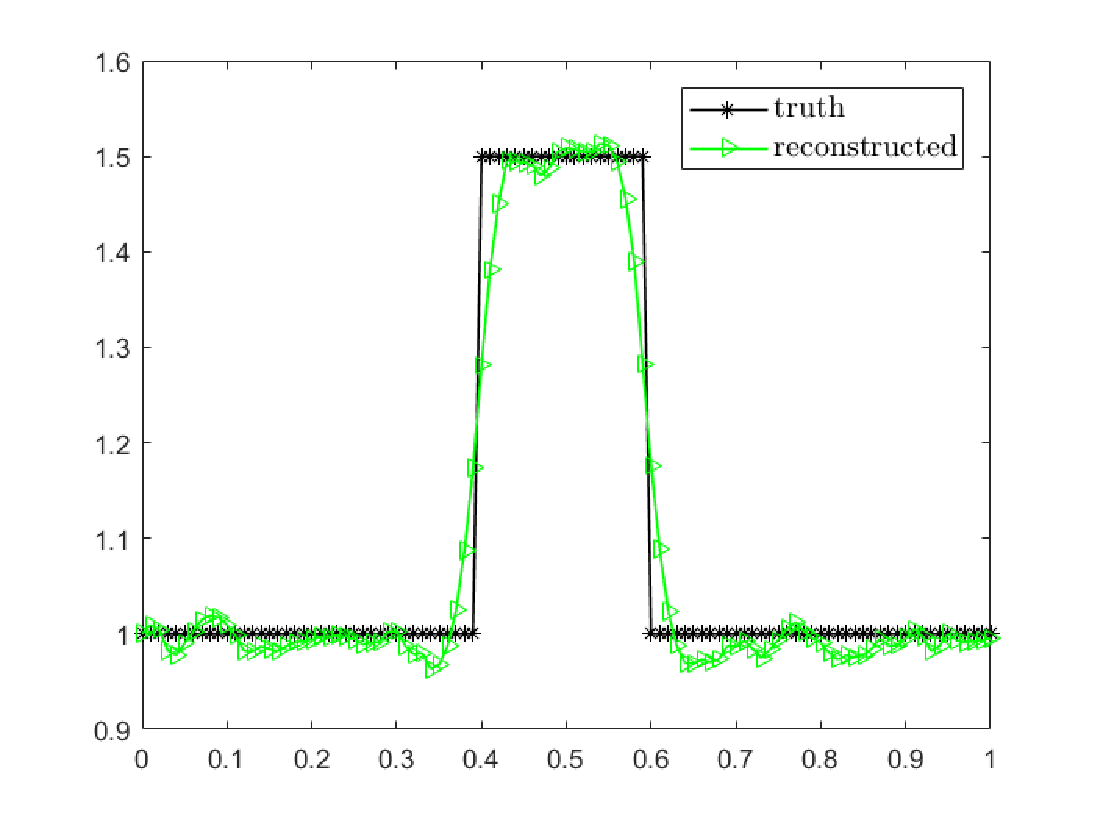}
\subcaption{}\label{fig5a}
\end{minipage}
\hfill
\begin{minipage}[b]{0.31\textwidth}\centering
\includegraphics[width=\linewidth]{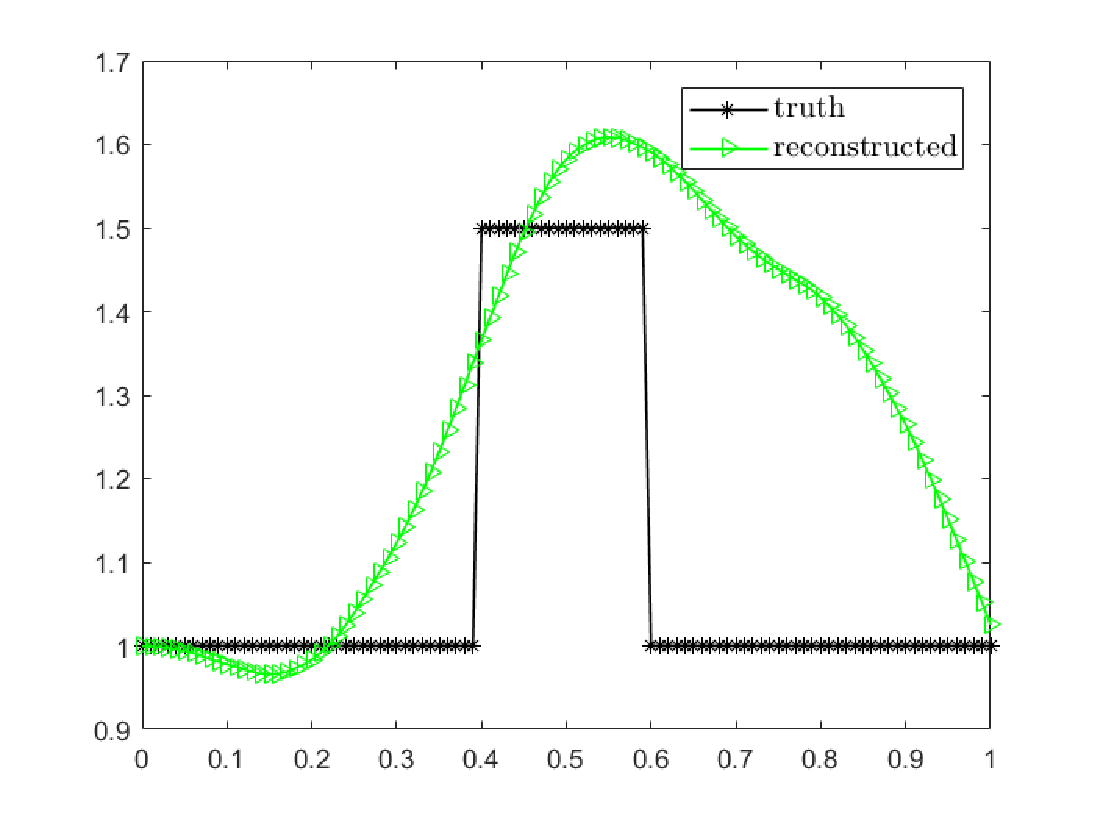}
\subcaption{}\label{fig5b}
\end{minipage}
\hfill
\begin{minipage}[b]{0.31\textwidth}\centering
\includegraphics[width=\linewidth]{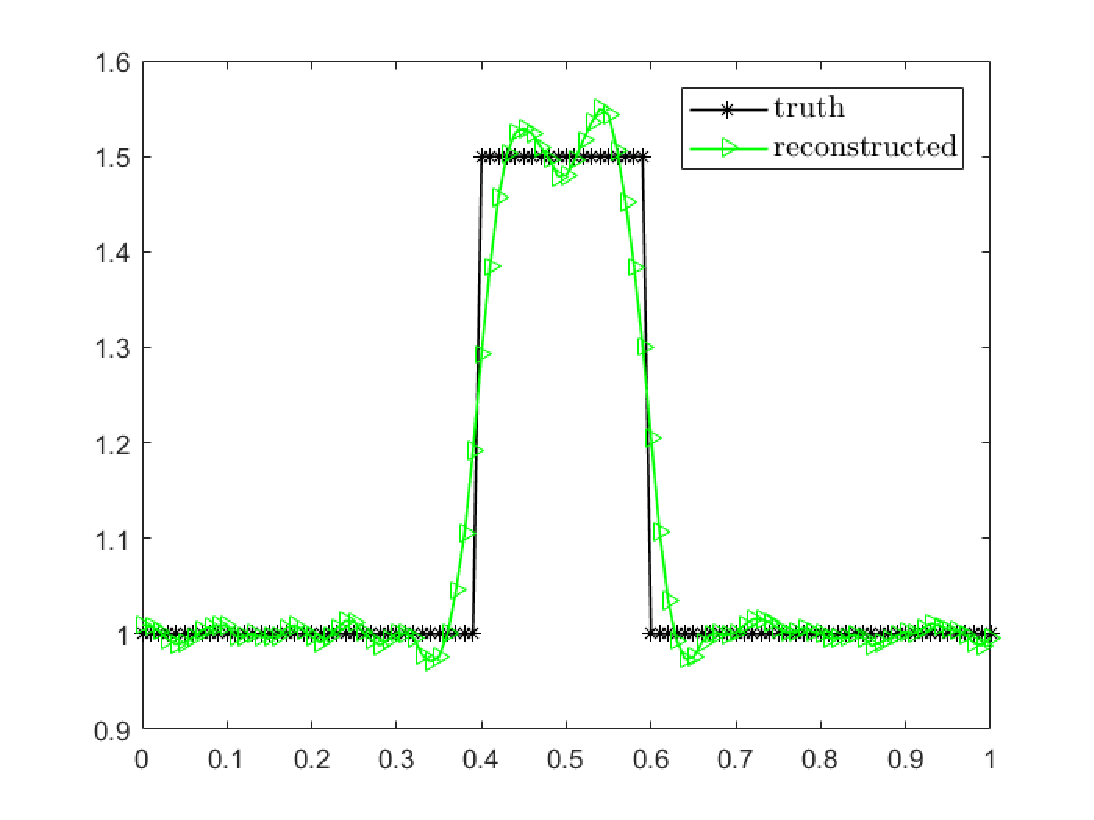}
\subcaption{}\label{fig5c}
\end{minipage}
\vspace{0.5em}

\begin{minipage}[b]{0.31\textwidth}\centering
\includegraphics[width=\linewidth]{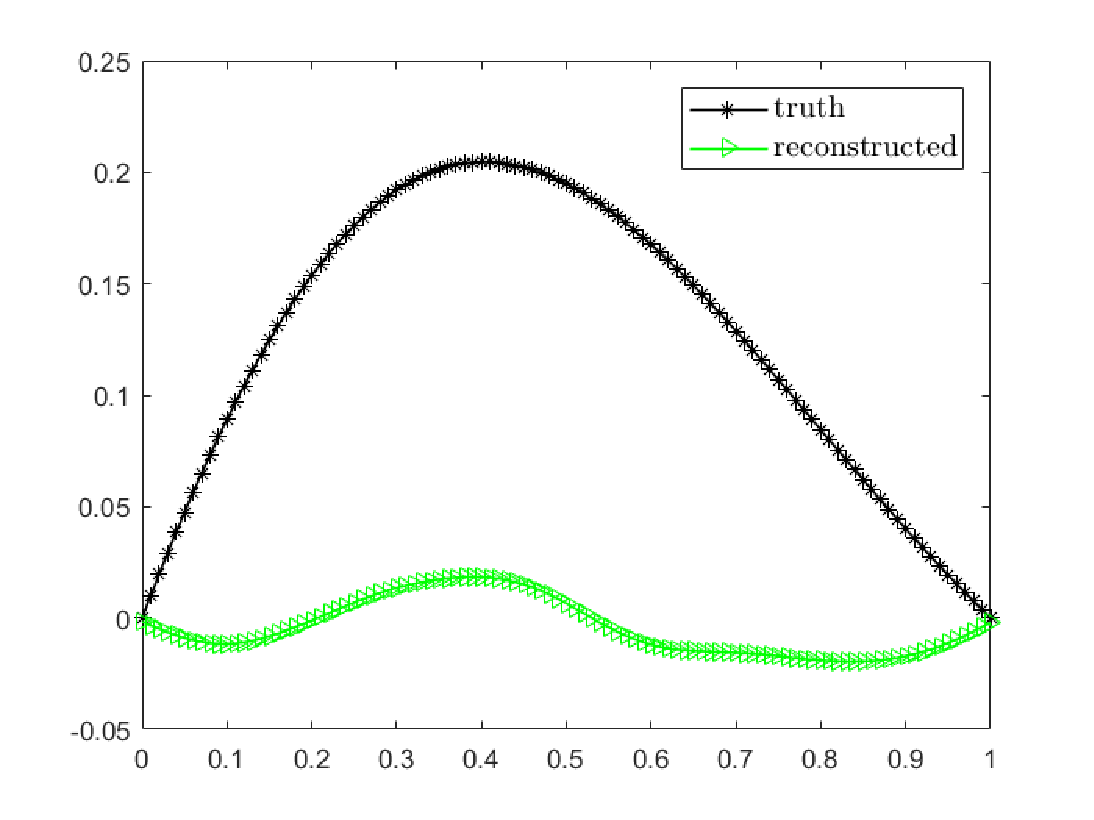}
\subcaption{}\label{fig5d}
\end{minipage}
\hfill
\begin{minipage}[b]{0.31\textwidth}\centering
\includegraphics[width=\linewidth]{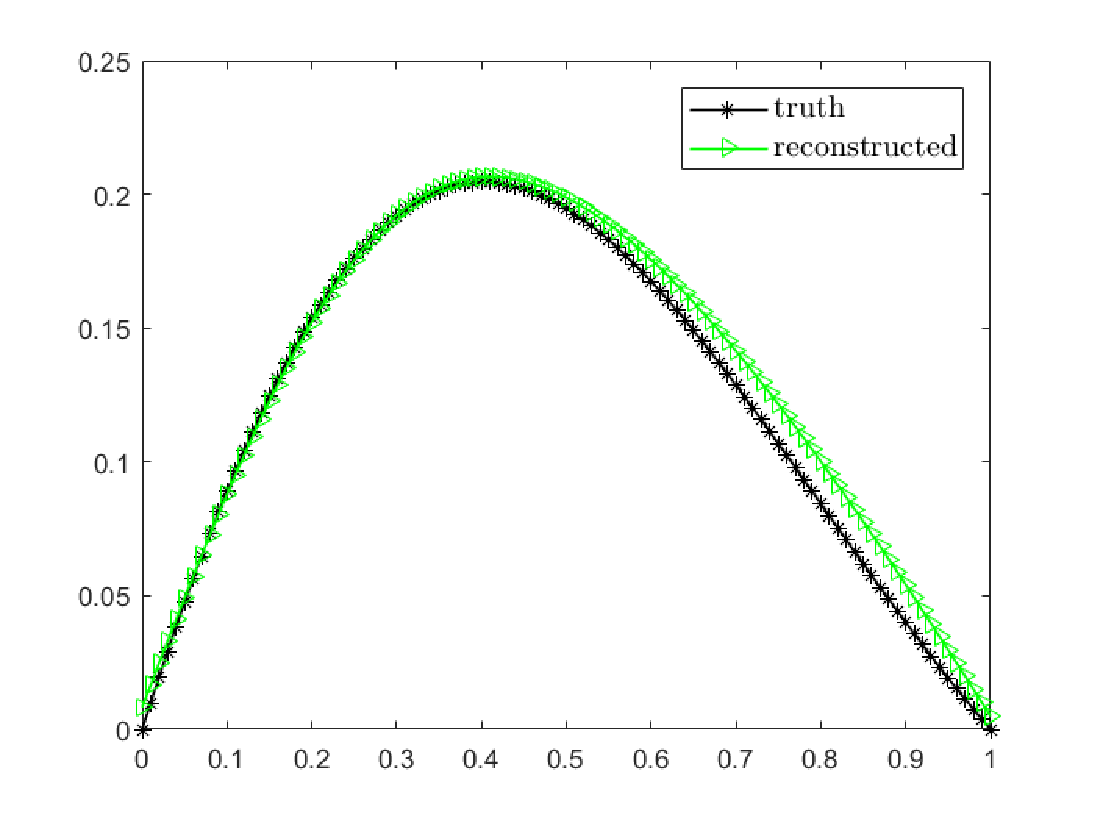}
\subcaption{}\label{fig5e}
\end{minipage}
\hfill
\begin{minipage}[b]{0.31\textwidth}\centering
\includegraphics[width=\linewidth]{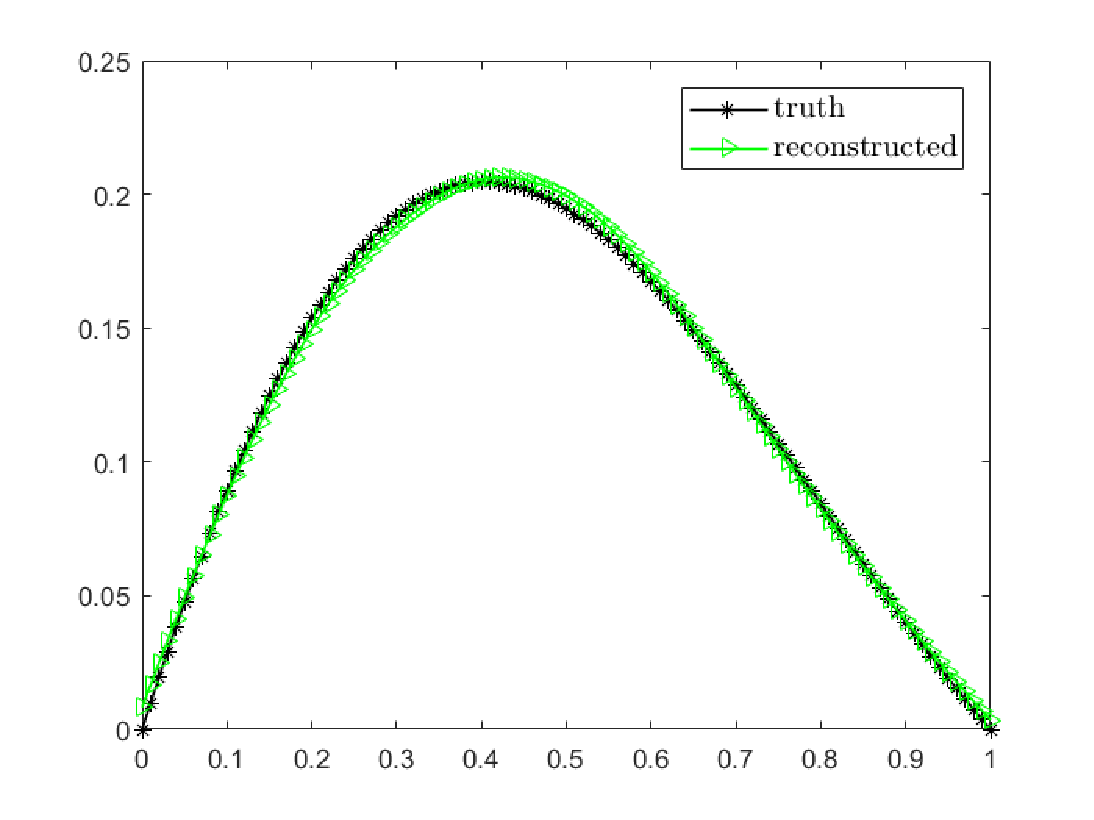}
\subcaption{}\label{fig5f}
\end{minipage}
\caption{Reconstruction results for $\rho_1(t)$ (top row) and $\rho_2(t)$ (bottom row) with noise level $\sigma=0.0001$. From left to right, the columns correspond to Case 1 (single measurement $u_1$), Case 2 (single measurement $u_2$), and Case 3 (full measurement), respectively}\label{fig5}
\end{figure}

These observations indicate that, although $\det\bm G(x_0)\ne0$ ensures theoretical identifiability, partial measurements may still lead to incomplete recovery of the temporal components. In other words, full measurement data significantly improves the stability and reliability of the inverse problem.


\subsection{Reconstruction by observing a single component}\label{sub5.3}

In the previous subsections, we mainly followed the line of Theorem \ref{thm-Lip} to recover the unknown by observing all components of the solution $\bm u$ at the observation point, and confirmed the pivotal role of the non-degeneracy condition $\det\bm G(x_0)\ne0$. We now investigate a different scenario motivated by the uniqueness result stated in Theorem~\ref{thm-ISP}. This theorem shows that, under some positivity condition on the spatial components of the source term as well as a specific structural condition on the unknown temporal ones, $\bm\rho$ can be uniquely determined from the single point observation of an arbitrary component of the solution.

More precisely, we recall that the key assumption \eqref{eq-cond-mu} in Theorem~\ref{thm-ISP} requires the existence of $\mu\in W^{1,\infty}(0,T)$ with $\mu(0)=0$ such that the unknown satisfies
\[
J^{\al_k}\rho_k=\mu,\quad k=1,\dots,K.
\]
To validate this property numerically, we retain the same input data and parameters as in the previous subsection, and construct temporal source terms that satisfy \eqref{eq-cond-mu}. Specifically, we choose
\[
\rho_1(t)=\f2{\Gamma(2.3)}\,t^{1.3},\quad\rho_2(t)=\f2{\Gamma(2.6)}\,t^{1.6}
\]
with fractional orders $\al_1=0.7$ and $\al_2=0.4$. Then one easily verifies that \eqref{eq-cond-mu} is satisfied with $\mu(t)=t^2$.

Figure~\ref{fig6} presents the reconstruction results under noise level $\sigma=0.0001$ by observing only $u_1(x_0,t)$ or $u_2(x_0,t)$. One can observe that both components $\rho_1(t)$ and $\rho_2(t)$ are accurately reconstructed by observing a single component. The reconstructed profiles closely match the exact solutions, and the effect of the noise level $\sigma=0.0001$ remains limited, indicating good numerical stability of the proposed algorithm.
\begin{figure}[htbp]\centering
\begin{minipage}[b]{0.45\textwidth}\centering
\includegraphics[width=\linewidth]{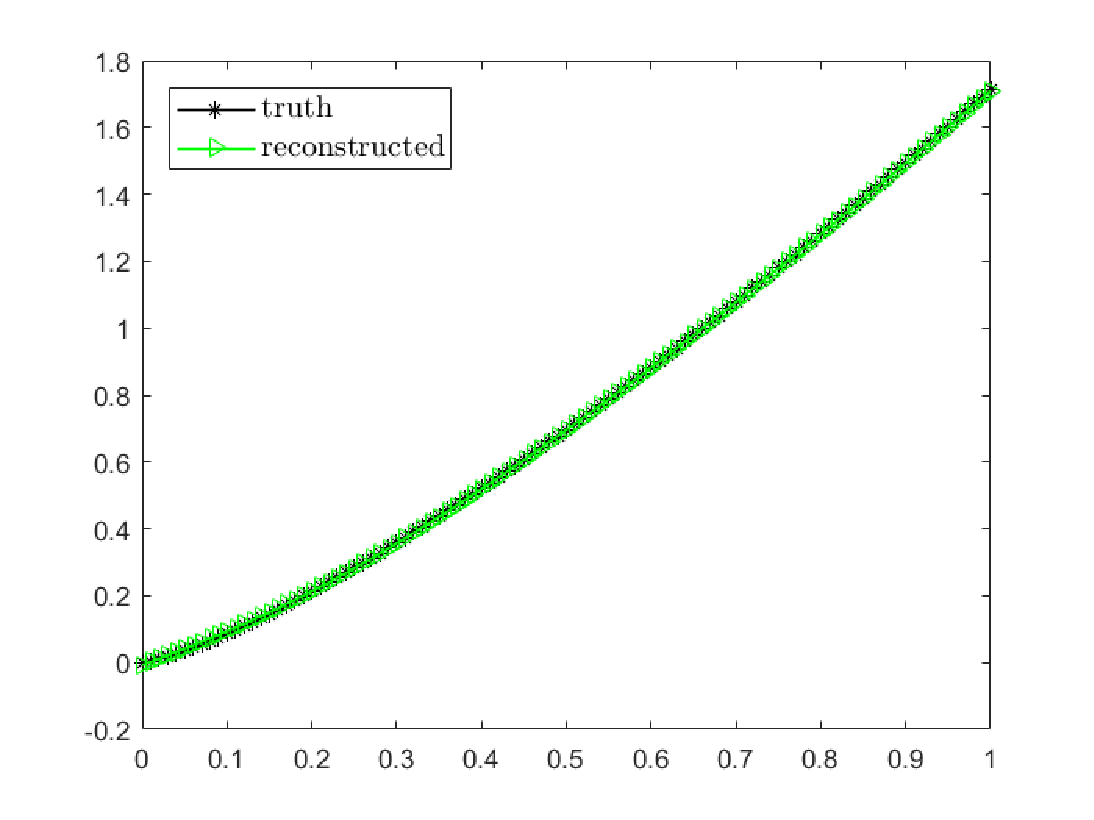}
\subcaption{Reconstruction of $\rho_1(t)$}\label{fig6a}
\end{minipage}
\hfill
\begin{minipage}[b]{0.45\textwidth}\centering
\includegraphics[width=\linewidth]{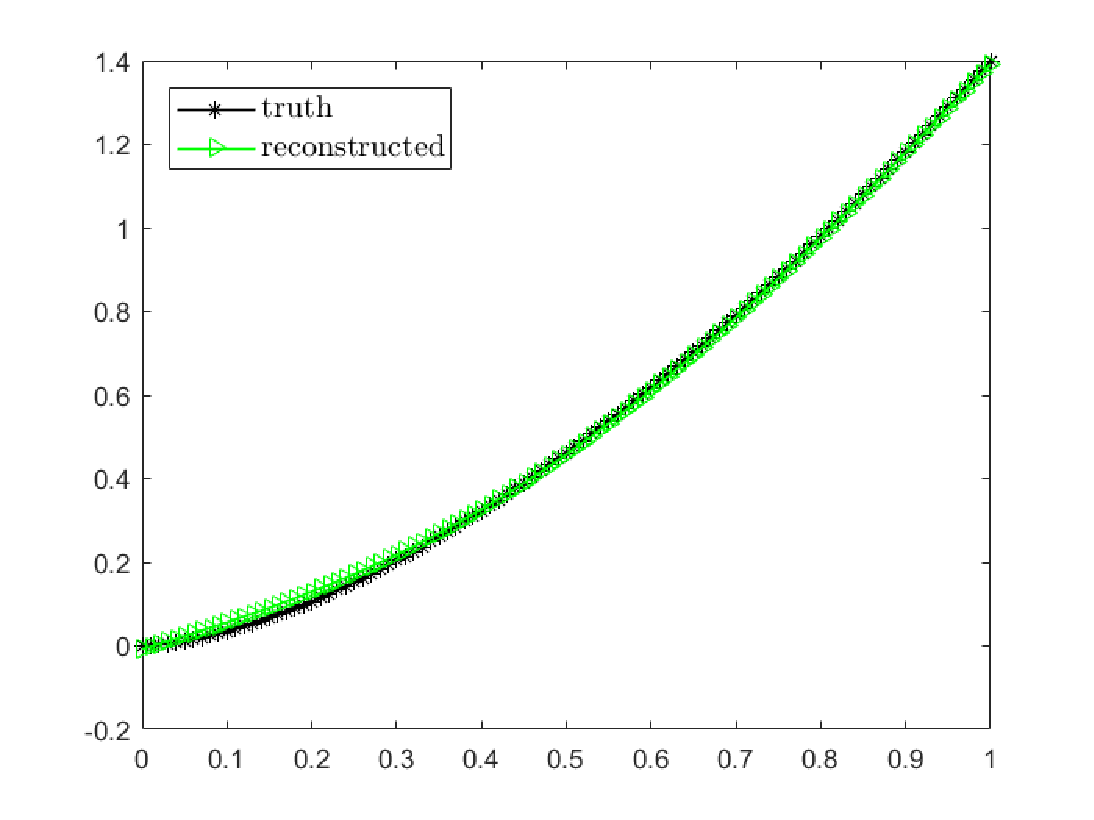}
\subcaption{Reconstruction of $\rho_2(t)$}\label{fig6b}
\end{minipage}
\vspace{0.5em}

\begin{minipage}[b]{0.45\textwidth}\centering
\includegraphics[width=\linewidth]{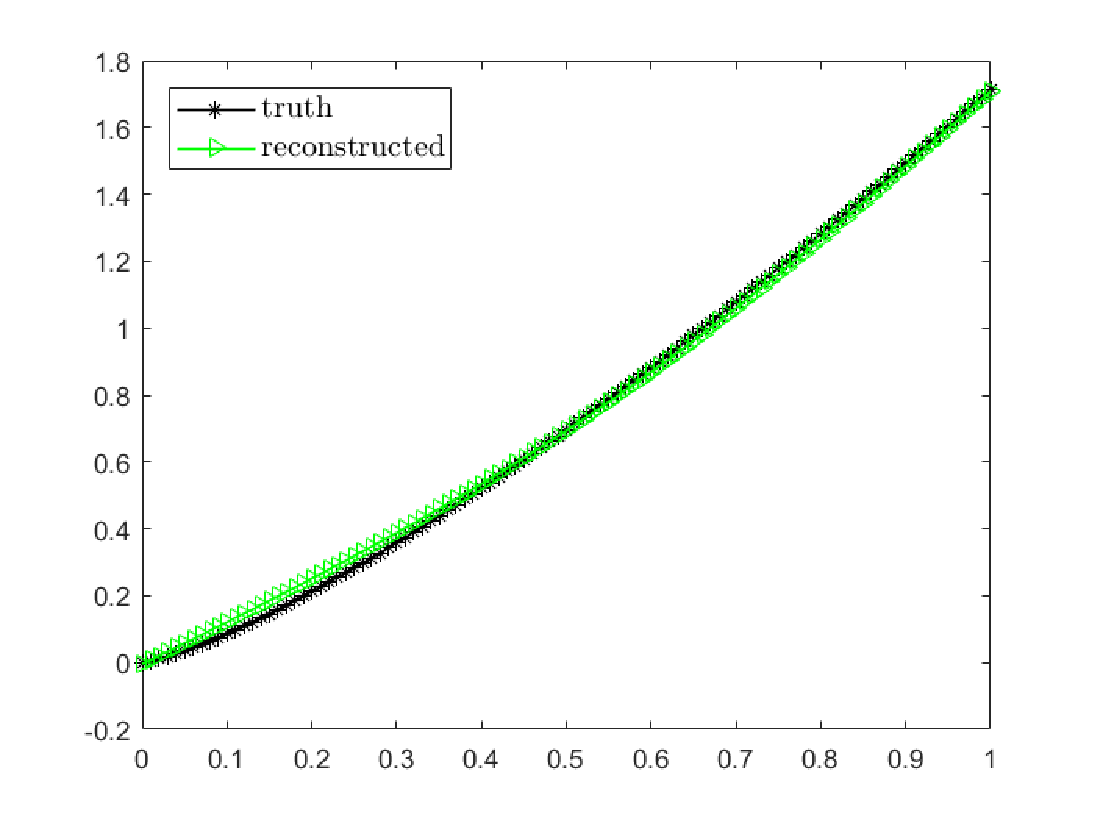}
\subcaption{Reconstruction of $\rho_1(t)$}\label{fig6c}
\end{minipage}
\hfill
\begin{minipage}[b]{0.45\textwidth}\centering
\includegraphics[width=\linewidth]{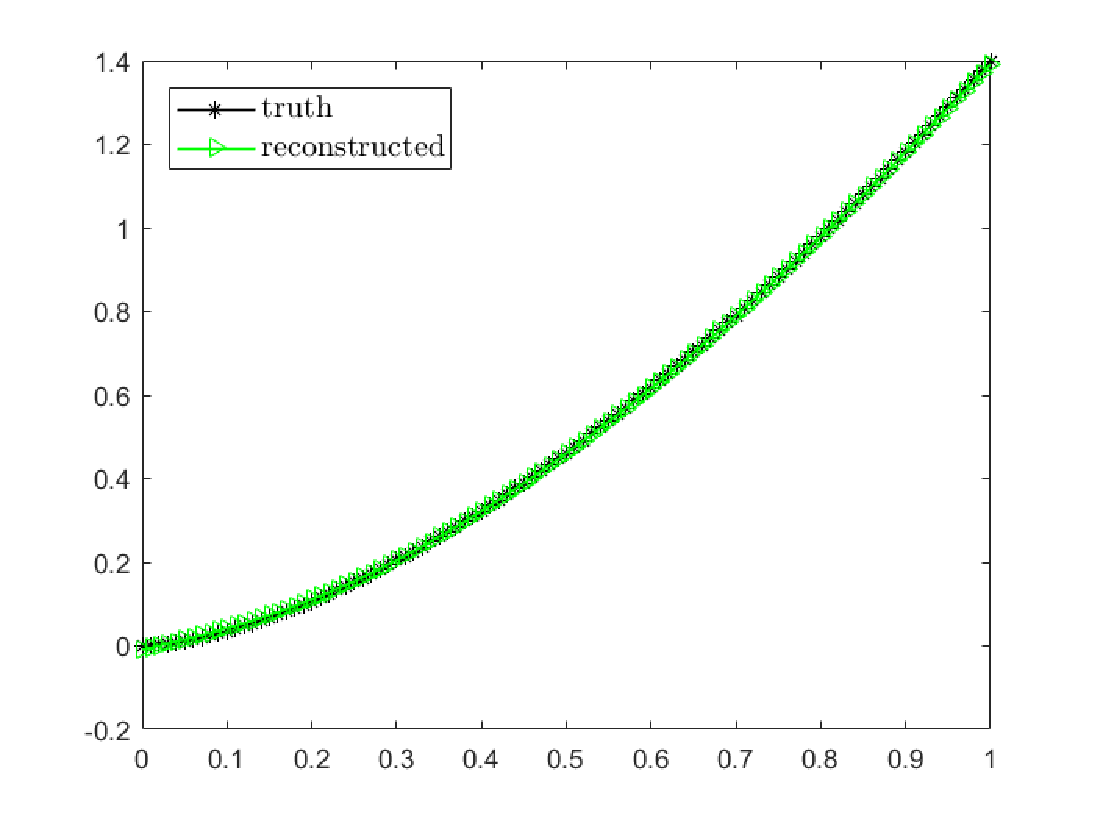}
\subcaption{Reconstruction of $\rho_2(t)$}\label{fig6d}
\end{minipage}
\caption{Reconstruction results for $\rho_1(t)$ (left column) and $\rho_2(t)$ (right column) with noise level $\sigma=0.0001$. The observation data is $u_1(x_0,t)$ in the top row, and $u_2(x_0,t)$ in the bottom row}\label{fig6}
\end{figure}

These results confirm the theoretical prediction of Theorem~\ref{thm-ISP}: under the restrictive compatibility condition \eqref{eq-cond-mu}, the inverse problem becomes uniquely solvable from a single observation. In contrast to the previous subsection, where single measurements led to partial recovery, the structural condition ensures sufficient information in each component to determine the entire unknown $\bm\rho$.

Therefore, this test provides numerical validation of the theoretical uniqueness result and highlights the crucial role played by the fractional structural condition \eqref{eq-cond-mu} in enabling reliable multi-component reconstruction from limited measurement data.


\subsection{Scalability with respect to the number of components}\label{sub5.4}

In this final numerical experiment, we investigate the scalability of the proposed IREKM algorithm with respect to the number of components in the unknown temporal source. More precisely, we consider coupled time-fractional systems with $K=3$ and $K=4$, and examine whether the reconstruction quality remains stable as the number of unknown functions in the inverse problem increases.

To this end, we consider the following two cases.
\begin{itemize}
\item{\bf Case 1:} $K=3$. The fractional orders and the matrix representing the spatial source components are chosen as
\[
\bm\al=(0.8,0.5,0.3)^\T,\quad\bm G(x)=\diag(\e^{-x},1+x^2,2+\cos(\pi x))
\]
respectively. The temporal source components to be reconstructed are given by
\[
\rho_1(t)=-t^2+t+1,\quad\rho_2(t)=|2t-1|,\quad\rho_3(t)=\cos(4\pi t).
\]
We take $m_k(t)=1$ for $k=1,2,3$.
\item{\bf Case 2:} $K=4$. We choose
\begin{gather*}
\bm\al=(0.9,0.75,0.4,0.15)^\T,\quad\bm G(x)=\diag(\e^{x},2-x^2,2+\sin(\pi x),4),\\
\rho_1(t)=t\,\e^{1-t^2},\ \rho_2(t)=1-|2t-1|,\ \rho_3(t)=\sin(3\pi t),\ \rho_4(t)=t\cos(4\pi t),
\end{gather*}
and set $m_1(t)=m_4(t)=t$, $m_2(t)=m_3(t)=0$.
\end{itemize}
In both cases, the observation data are collected at the fixed point $x_0=0.35$.

Figures \ref{fig7}--\ref{fig77} present the reconstruction results under different noise levels $\sigma=0.0001$, $\sigma=0.001$ and $\sigma=0.01$ for Case 1 ($K=3$) and Case 2 ($K=4$), respectively.
\begin{figure}[htbp]\centering
\begin{minipage}[b]{0.31\textwidth}\centering
\includegraphics[width=\linewidth]{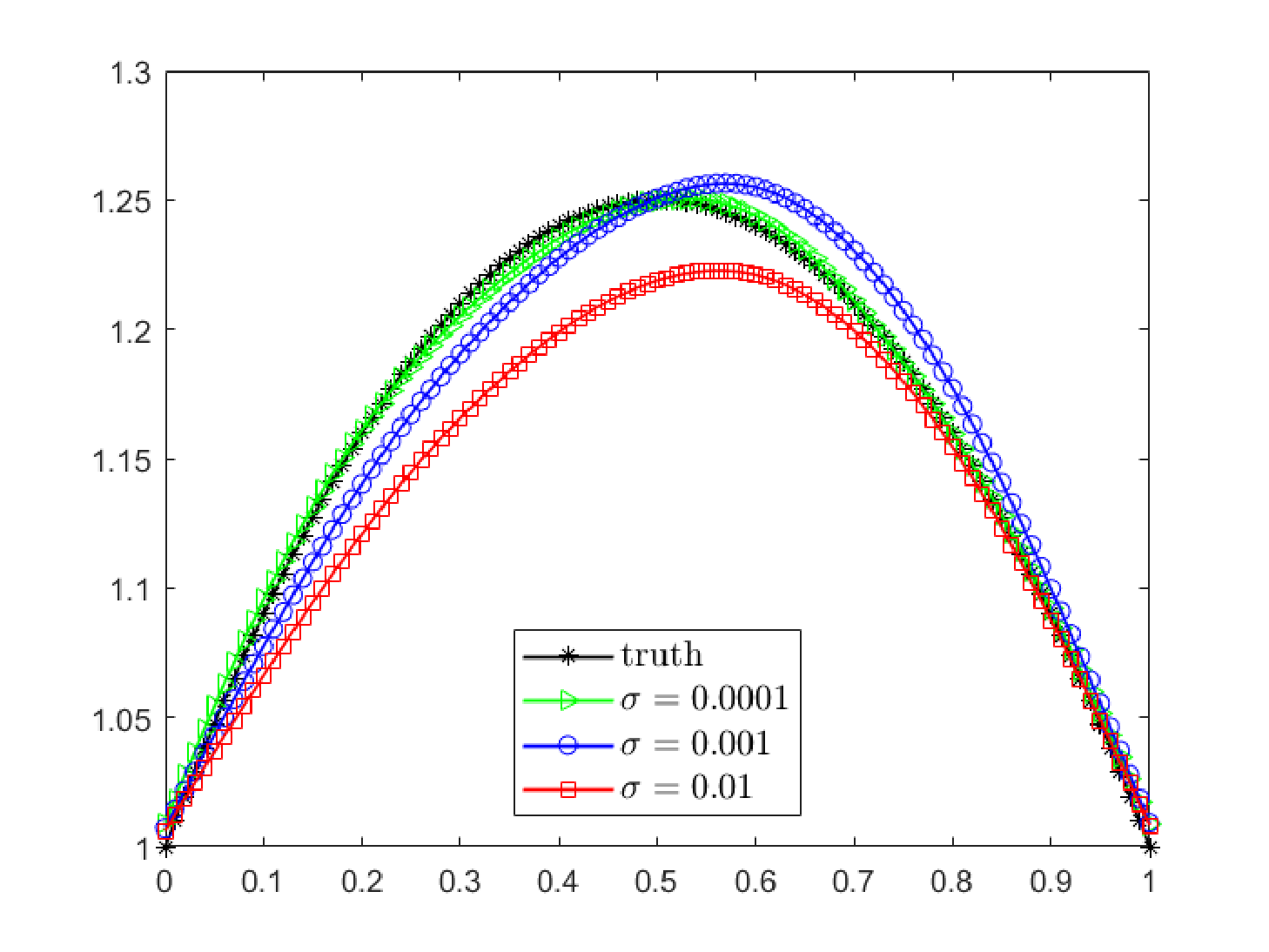}
\subcaption{Reconstruction of $\rho_1(t)$}\label{fig7a}
\end{minipage}
\hfill
\begin{minipage}[b]{0.31\textwidth}\centering
\includegraphics[width=\linewidth]{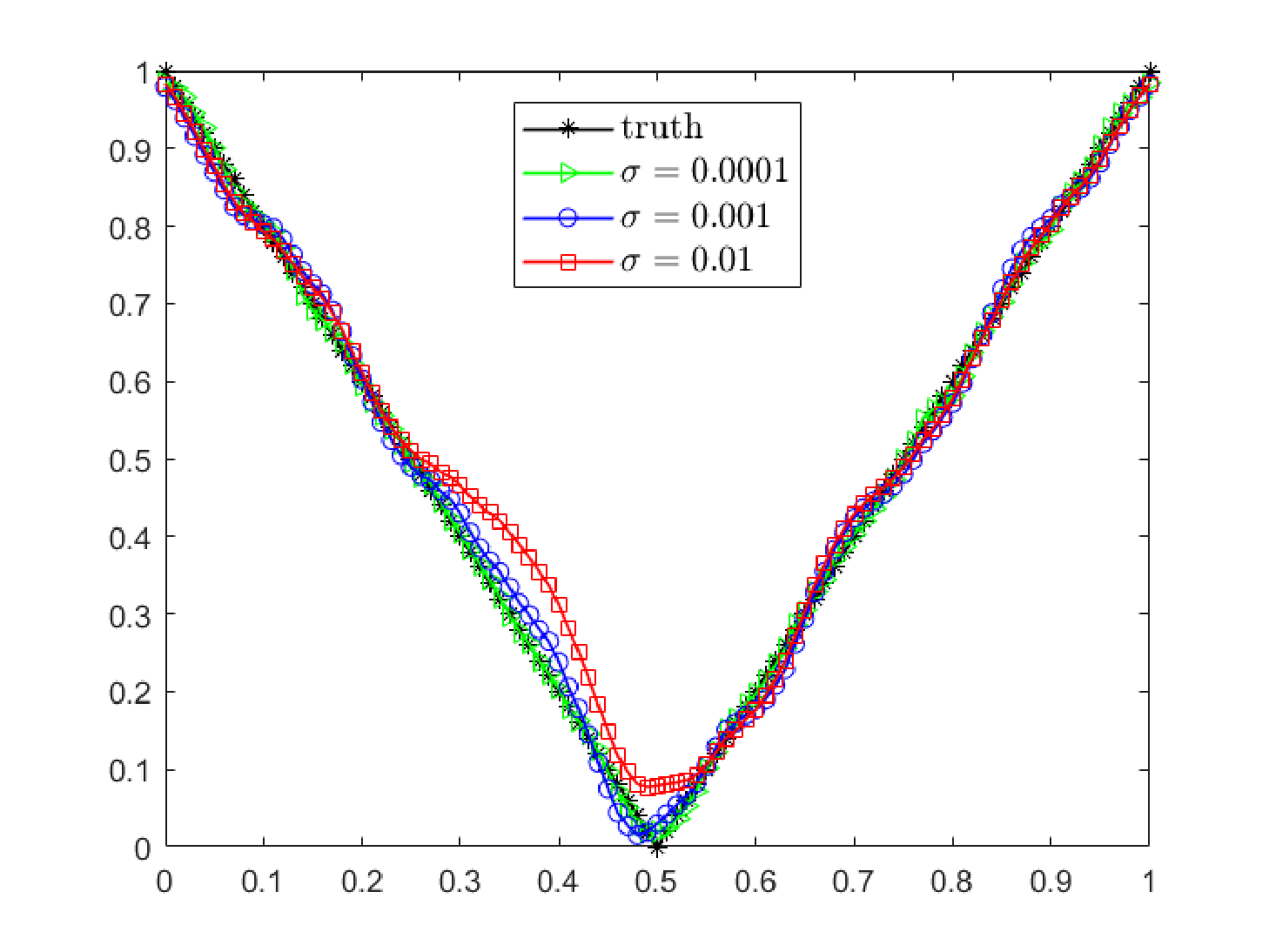}
\subcaption{Reconstruction of $\rho_2(t)$}\label{fig7b}
\end{minipage}
\hfill
\begin{minipage}[b]{0.31\textwidth}\centering
\includegraphics[width=\linewidth]{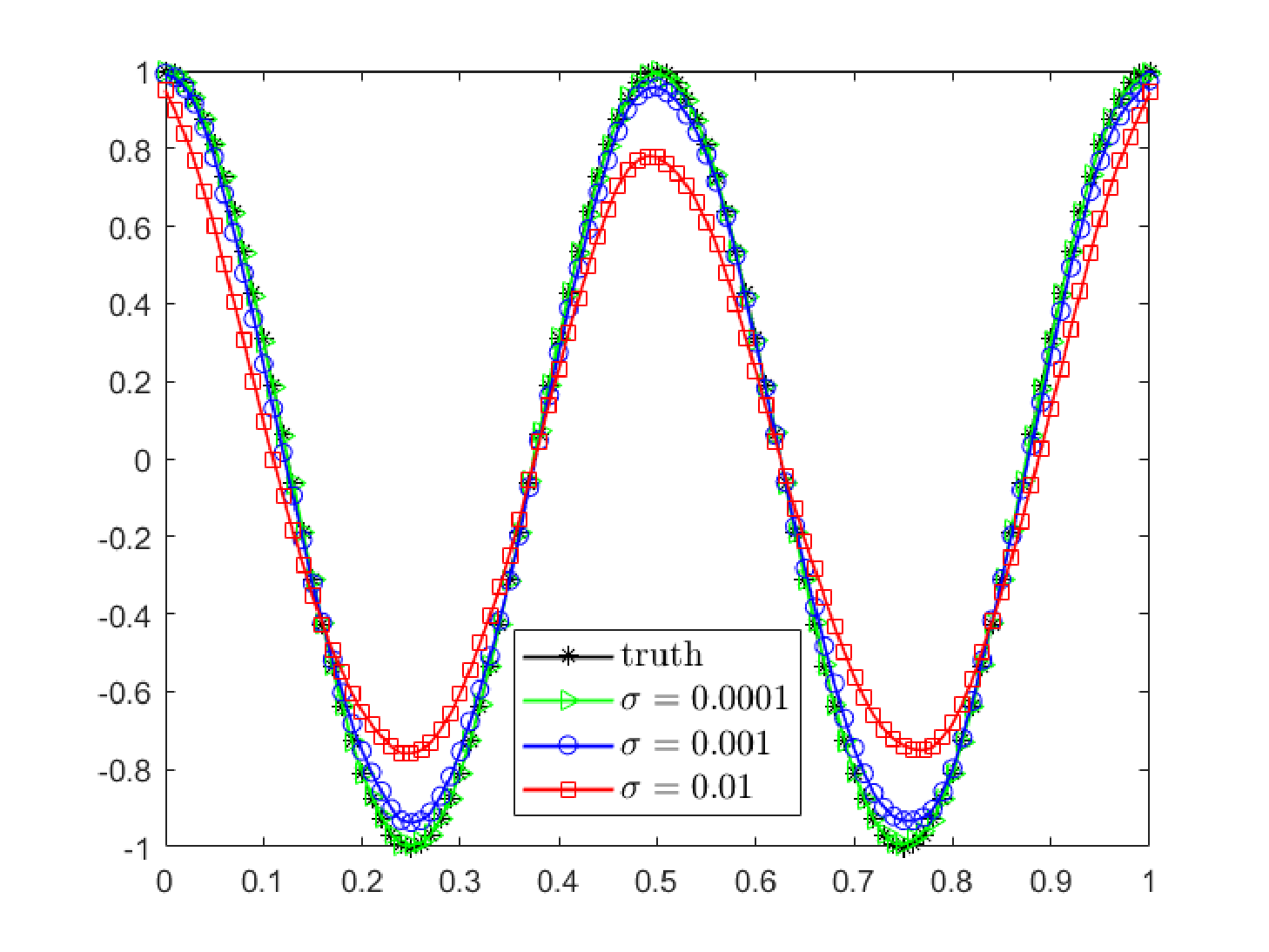}
\subcaption{Reconstruction of $\rho_3(t)$}\label{fig7c}
\end{minipage}
\caption{Reconstruction results for Case 1 ($K=3$) under noise levels $\sigma=0.0001$, $\sigma=0.001$, and $\sigma=0.01$. From left to right: $\rho_1(t)$, $\rho_2(t)$ and $\rho_3(t)$}\label{fig7}
\end{figure}

\begin{figure}[htbp]\centering
\begin{minipage}[b]{0.45\textwidth}\centering
\includegraphics[width=\linewidth]{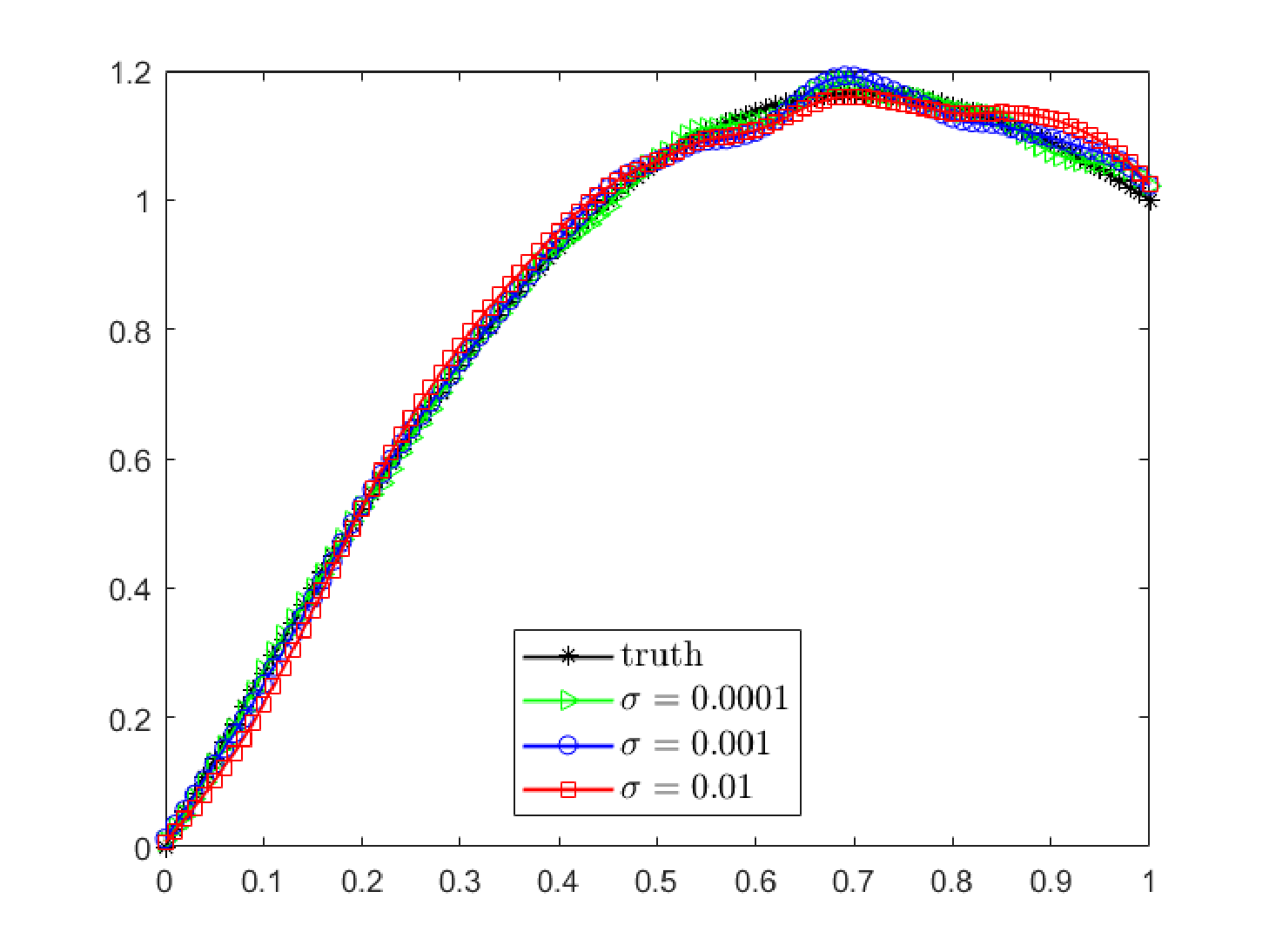}
\subcaption{Reconstruction of $\rho_1(t)$}\label{fig77a}
\end{minipage}
\hfill
\begin{minipage}[b]{0.45\textwidth}\centering
\includegraphics[width=\linewidth]{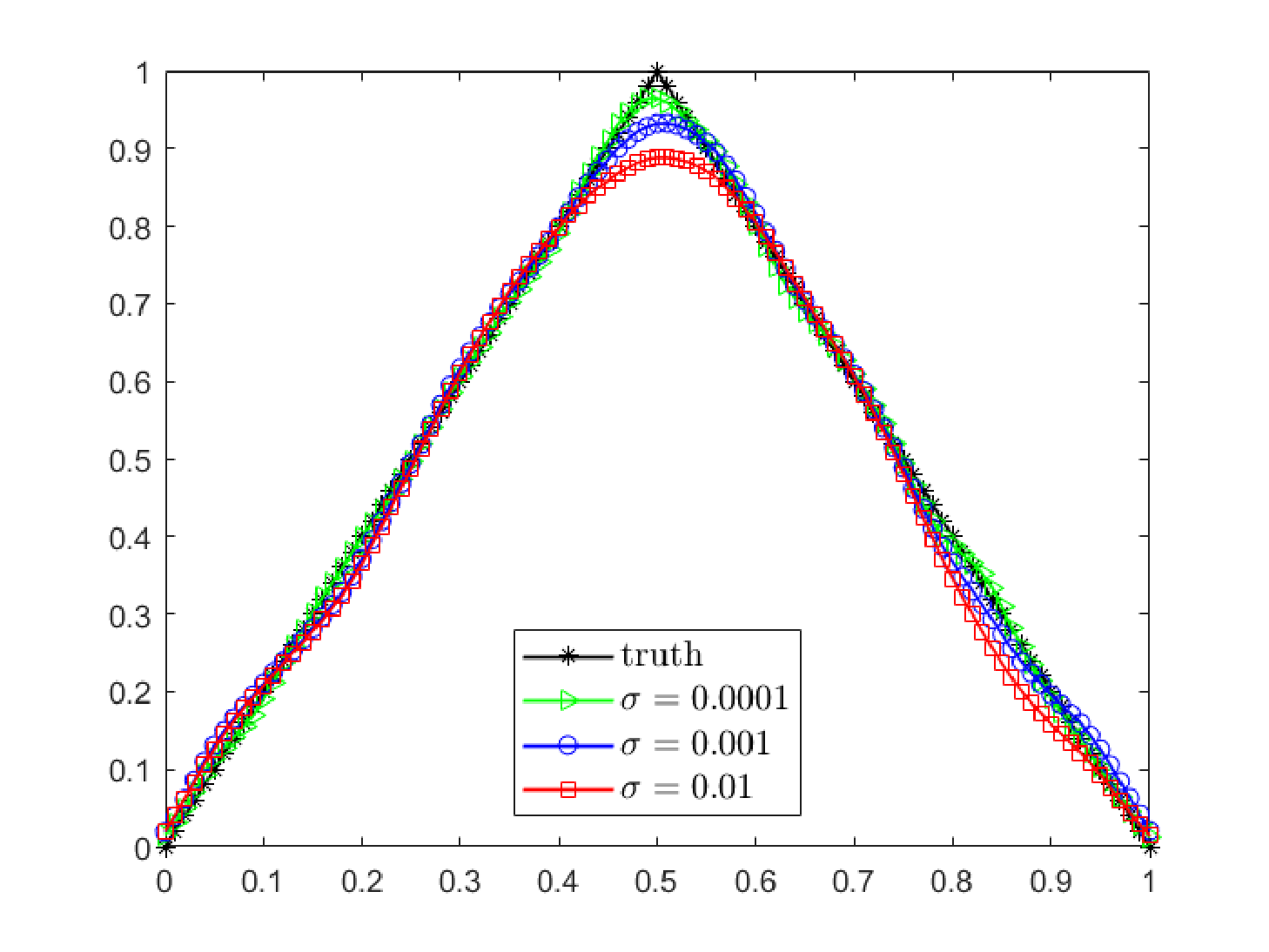}
\subcaption{Reconstruction of $\rho_2(t)$}\label{fig77b}
\end{minipage}
\vspace{0.5em}

\begin{minipage}[b]{0.45\textwidth}\centering
\includegraphics[width=\linewidth]{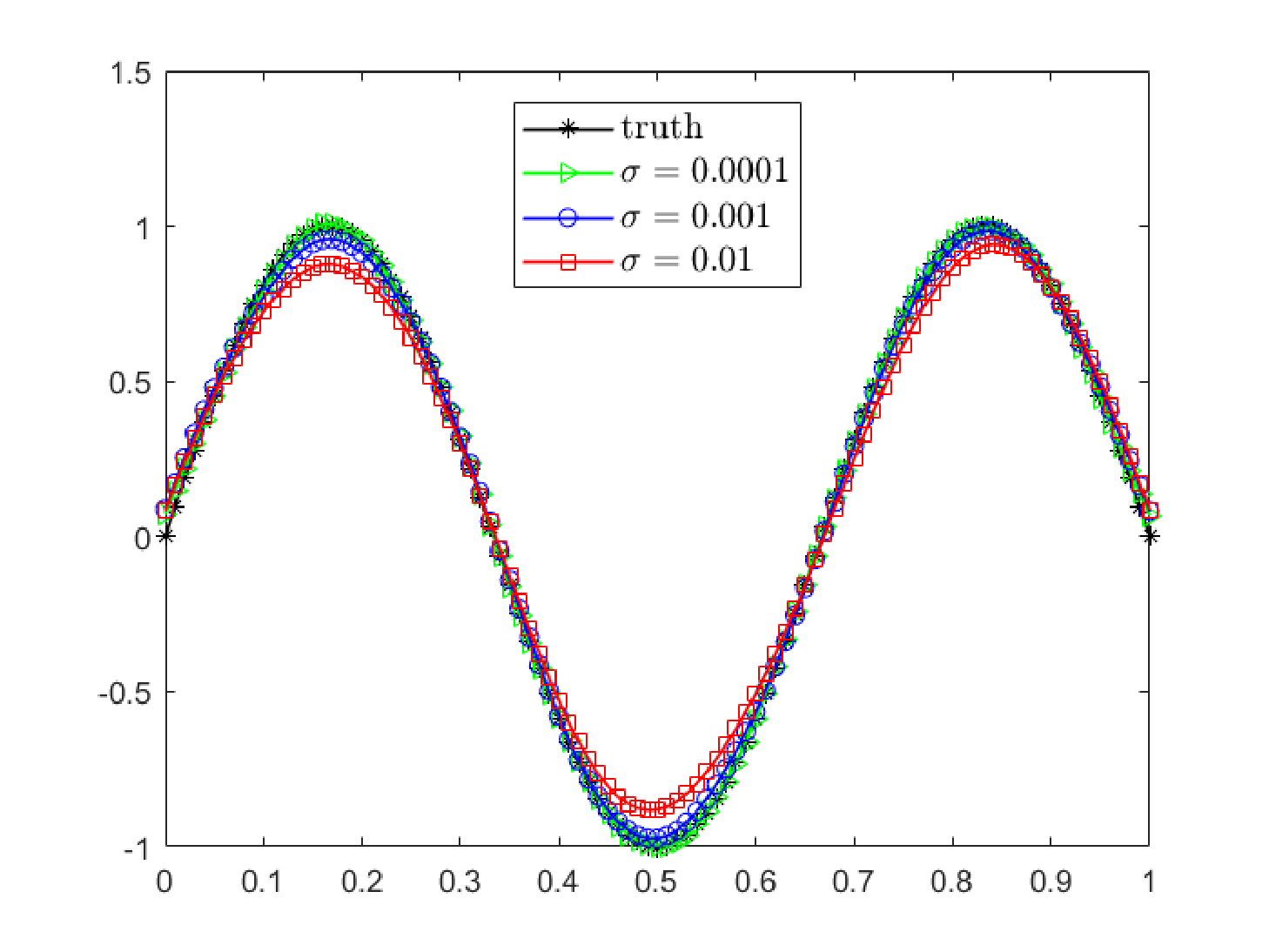}
\subcaption{Reconstruction of $\rho_3(t)$}\label{fig77c}
\end{minipage}
\hfill
\begin{minipage}[b]{0.45\textwidth}\centering
\includegraphics[width=\linewidth]{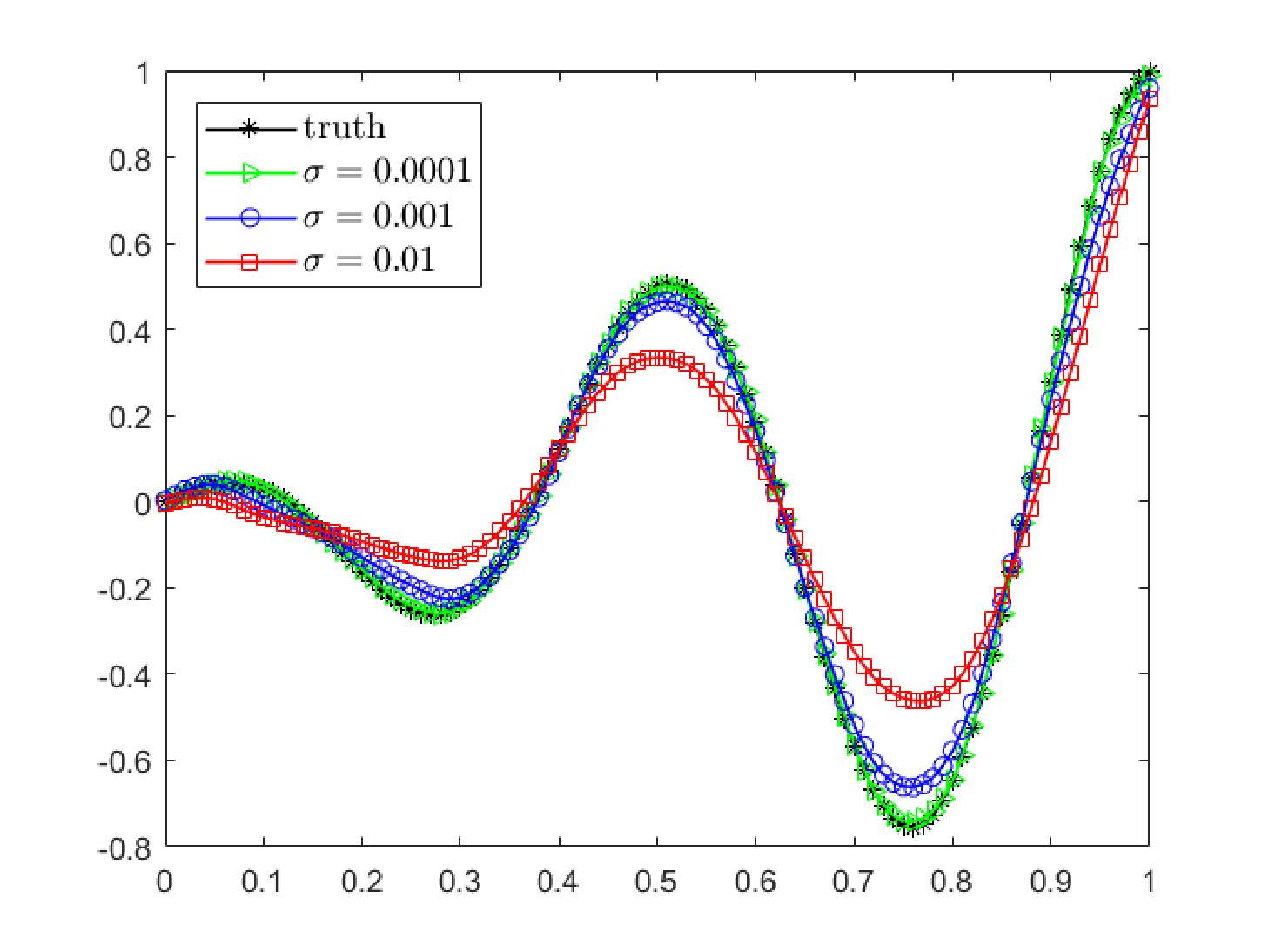}
\subcaption{Reconstruction of $\rho_4(t)$}\label{fig77d}
\end{minipage}
\caption{Reconstruction results for Case 2 ($K=4$) under noise levels $\sigma=0.0001$, $0.001$, and $0.01$. From left to right, top to bottom: $\rho_1(t)$, $\rho_2(t)$, $\rho_3(t)$ and $\rho_4(t)$}\label{fig77}
\end{figure}

It is observed that the algorithm maintains good performance as the number of components increases. Accurate and stable reconstructions are obtained for systems with $K=3$ and $K=4$, indicating good scalability of the proposed approach.


\subsection{Summary of numerical findings}\label{sub5.5}

From the results presented in Subsections \ref{sub5.1}--\ref{sub5.4}, several important conclusions can be drawn regarding the performance and reliability of the proposed IREKM algorithm.

First, it turns out that the non-degeneracy condition $\det\bm G(x_0)\ne0$ plays a crucial role in ensuring the full recovery of the temporal source components. When this condition is violated, the system loses full identifiability, and at least one temporal component becomes unrecoverable, resulting in partial reconstruction and potential instability. In contrast, when $\det \bm G(x_0) \ne 0$, all temporal components are simultaneously identifiable and can be accurately reconstructed.

Even when the identifiability condition holds, the choice of measurement data remains decisive. Partial measurements may lead to incomplete recovery of the temporal components, whereas full observation data significantly improves stability and enables reliable simultaneous reconstruction of all unknown sources.

The numerical experiments also validate Theorem~\ref{thm-ISP}. In particular, when the structural constraint $J^{\al_k}\rho_k=\mu$ is satisfied, the inverse problem becomes uniquely solvable from a single measurement. This compatibility condition restores full identifiability even in reduced observation settings, demonstrating that appropriate structural constraints can compensate for limited data.

Furthermore, the proposed IREKM algorithm successfully reconstructs a wide variety of temporal profiles, including smooth, non-differentiable and even non-continuous functions. This highlights the method’s flexibility and robustness with respect to the regularity of the source terms.

Across all numerical experiments, the reconstruction quality consistently improves as the noise level decreases. From Figures~\ref{fig8a}--\ref{fig8c}, we observe that the approximation errors decrease as the noise diminishes, confirming the improved accuracy of the reconstruction. Figure~\ref{fig8d} further illustrates the stability and convergence behavior of the IREKM algorithm. The residual norms decrease monotonically, and the stopping criterion defined in equation~\eqref{stopped} proves to be efficient and well-suited to the problem. Collectively, these results demonstrate that IREKM is both effective and robust for the simultaneous recovery of multiple temporal source components in a time-fractional coupled system, even under practical noisy conditions.


\section{Concluding Remarks}\label{sec-conclude}

We close this manuscript by reviewing the main results and discussing some future topics. From the theoretical aspect, we deal with both inverse $t$-source problem and the strict positivity issue for coupled subdiffusion systems, which were not investigated in existing literature. The 3 main theorems in this article generalize similar results in the case of single equations, which basically follows the same strategy as those in \cite{SY11,LRY16} but technically more challenging. In the proof of Theorem \ref{thm-Lip}, we refine the estimate of the solution with respect to the source term, which improves that in \cite{LHL23} and clarifies the structure of the smoothing effect. Moreover, we derive a closed representation \eqref{eq-series} for the mild solution taking a series form, which may provide great convenience in future studies of coupled systems.

Next, Theorem \ref{thm-SPP} claims strict positivity of the solution to a homogeneous problem in the sense of performing some Riemann-Liouville integral, under reasonable positivity conditions \eqref{eq-cond-C}--\eqref{eq-cond-rq} on the coupling matrix $\bm C$ and the initial value $\bm g$. Especially, in Corollary \ref{coro-SPP1} we confirm the real strict positivity provided that all components in $\bm g$ are non-negative and non-vanishing. On the other hand, obviously there is still room to improve Theorem \ref{thm-SPP} by removing the Riemann-Liouville integral, which relies completely on its scalar-valued counterpart, i.e., showing $u>0$ a.e. in $\Om\times(0,T)$ in Corollary \ref{coro-SPP0}.

Analogously, Theorem \ref{thm-ISP} also deserves further improvement due to the strong constraint \eqref{eq-cond-mu}, although it already improves Theorem \ref{thm-Lip} greatly. Indeed, at the moment we only succeed in uniquely determining a common factor $\mu$ behind $\bm\rho$ by observing a single component. The essential difficulty is rooted in the fact that for coupled evolution systems, one should pay special attention to the commutativity among matrices and operators, which prevents us from simple application of Duhamel's principle (Lemma \ref{lem-Duhamel2}) like the scalar-valued case. We shall explore the possibility of the simultaneous identification of essentially independent components in $\bm\rho$ by observing partial components of $\bm u(\bm x_0,\,\cdot\,)$ by exploiting the coupling effect of the system.


\paragraph{Acknowledgments}
The second author is supported by JSPS KAKENHI Grant Numbers JP23KK0049 and JP26K06926, Guangdong Basic and Applied Basic Research Foundation (No.\! 2025A1515012248) and FY2025 MUSUBIME of Kyoto University.


\section*{Declarations}

\paragraph{Conflict of interest}
The authors declare that there are no conflicts of interest.


\end{document}